\documentclass[11pt]{article}

\usepackage[margin=1in]{geometry}
\usepackage{amsmath,amssymb,amsthm,bbm,bm}
\usepackage{graphicx}
\usepackage{booktabs}
\usepackage{natbib}
\usepackage{algorithm}
\usepackage{algpseudocode}
\usepackage{tikz}
\usepackage{hyperref}
\usepackage{float}
\hypersetup{
  colorlinks=true,
  linkcolor=blue,
  citecolor=blue,
  urlcolor=blue
}

\DeclareMathOperator{\Var}{Var}
\DeclareMathOperator{\Dir}{Dir}
\DeclareMathOperator{\KL}{KL}

\theoremstyle{remark}

\theoremstyle{plain}
\newtheorem{theorem}{Theorem}[section]
\newtheorem{proposition}[theorem]{Proposition}

\newtheorem{lemma}[theorem]{Lemma}

\theoremstyle{definition}
\newtheorem{assumption}[theorem]{Assumption}
\newtheorem{definition}[theorem]{Definition}

\newtheorem{remark}[theorem]{Remark}

\title{Decision–Theoretic Robustness for Network Models}

\author{
  Marios Papamichalis\thanks{Human Nature Lab, Yale University, New Haven, CT 06511, \texttt{marios.papamichalis@yale.edu}} \and
  Regina Ruane\thanks{Department of Statistics and Data Science, The Wharton School at the University of Pennsylvania, 3733 Spruce Street, Philadelphia, PA 19104-6340, \texttt{ruanej@wharton.upenn.edu}} \and
  Sim\'on Lunag\'omez\thanks{Department of Statistics, ITAM, Mexico, \texttt{simon.lunagomez@itam.mx}} \and
  Swati Chandna\thanks{Birkbeck University of London, \texttt{s.chandna@bbk.ac.uk}}
}

\date{\today}

\begin{document}
\maketitle

\begin{abstract}
Network data from neuroscience, epidemiology, and the social sciences are routinely analyzed with Bayesian network models such as Erd\H{o}s-R\'enyi graphs, stochastic block models, random dot product graphs, and graphon priors. In such applications, these models are only approximations, whereas real networks are sparse, heterogeneous, and exhibit higher-order dependencies that no single specification fully captures. This raises the question: how stable are network-based decisions, model selection,
functional summaries, and policy recommendations under small misspecification of the assumed model? We address this question using a local decision-theoretic robustness framework, in which the posterior distribution is allowed to vary within a small Kullback-Leibler neighborhood and the actions are chosen to minimize the worst-case posterior expected loss. The specialized application of this framework to exchangeable network models is driven
by the availability of low-dimensional network functionals. First, we adapt Decision–Theoretic robustness to exchangeable graphs via graphon limits and derive sharp small-radius expansions for the robust posterior risk. For squared loss, the leading inflation term is shown to be controlled by the posterior variance of the loss. For robustness indices that diverge at percolation and fragmentation thresholds, we obtain a universal critical exponent describing how decision--level uncertainty explodes near criticality. Second, we develop a nonparametric minimax theory for decision–theoretic robust model selection between sparse Erd\H{o}s--R\'enyi and stochastic block models. For percolation--type robustness functionals in configuration models and sparse graphon classes, these show that no Bayesian or frequentist procedure can improve the resulting decision–theoretic robustness error exponents uniformly over these classes. Third, we propose a practical algorithm for robust network analysis based on entropic tilting of posterior or variational samples and illustrate its use on functional brain connectivity and Karnataka village social networks. Together, these results provide a decision--theoretic notion of robustness for Bayesian network analysis that complements classical object--level concepts of network resilience.
\end{abstract}

\section{Introduction}
\label{sec:intro}

Network data emerge from a multitude of sources from functional brain connectivity and the spread of infectious diseases to village social networks and online platforms. Bayesian network models have become a central tool for the analysis of such data, providing a coherent framework for the learning of latent structure as well as comparing competing representations and the propagation of uncertainty into predictions and policy decisions. Popular choices include Erd\H{o}s--R\'enyi (ER) graphs, stochastic block models (SBMs), latent space and random dot product graph models, and nonparametric graphon priors.\\

In practice, these models are only approximations. Empirical networks are often sparse and highly heterogeneous, with communities, hubs, degree variability, and higher--order dependence only partially captured by any single specification. Yet when deciding on model choice, choosing between community and latent--space representations entails ranking interventions by their expected impact, or comparing robustness indices across networks. These are typically reported without a systematic assessment of their sensitivity to misspecification. The central question of this paper is:\\

\medskip
\emph{How does the quality of network--based decisions degrade when the assumed Bayesian network model is only approximately correct?}\\
\medskip

We address this question by adopting the local decision--theoretic robustness framework of \citet{WatsonHolmes2016}. In this perspective, Bayesian analysis is formulated as a decision problem, where an action is chosen to minimize a posterior expected loss and robustness is assessed by allowing the posterior to move within a small Kullback--Leibler (KL) (or more general $\phi$--divergence) neighborhood of a working model. Then, evaluation of the worst--case posterior risk over this neighborhood is conducted \citep{WatsonHolmes2016,watson2017characterizing}. The resulting increase in risk quantifies how fragile a Bayes decision is to local misspecification of the likelihood or prior. This approach extends classical ideas in robust Bayesian analysis such as $\Gamma$--minimax rules and $\varepsilon$--contamination \citep{Berger2,Vidakovic}, and is closely related to variational representations of ambiguity--averse preferences in economics \citep{Maccheroni,Hansen} and to recent proposals tempering or coarsening the posterior to mitigate misspecification \citep{Miller2019,avella2022robustness}. Our aim is to bring this \emph{decision--level} robustness perspective into the setting of exchangeable network models, and to connect it both to graph limit theory and classical notions of network robustness based on percolation and resilience.\\

\paragraph{Object--level versus decision--level robustness.}
A large body of literature in network science discusses robustness at the \emph{object level} by perturbing the graph itself. Percolation analyses on random graph models quantify how the giant component collapses as nodes or edges are removed. \citet{CallawayEtAl2000} and \citet{CohenEtAl2000} demonstrated that ER networks can lose global connectivity under relatively modest random failures, whereas networks with heavy--tailed degree distributions (as in scale--free or configuration models) are remarkably resilient. These results have been refined using scaling limits and critical exponents in configuration models \citep{JanssenEtAl2014,VDHofstadBook2017}. Recent surveys emphasize the ubiquity of such phenomena across domains \citep{artime2024robustness,KawasumiHasegawa2024}. Variants consider structured or adversarial perturbations that include heavy--tailed spatial networks, which may retain a giant component even under arbitrary node removals \citep{JacobMoerters2017}, and preferential--attachment networks that can remain connected under targeted attacks once their degree baseline is accounted for \citep{HasheminezhadBrandes2023}. \\

A complementary line of work studies the robustness of \emph{inference procedures} to perturbations of the observed graph. Examples include the stability of spectral community detection when an underlying SBM is corrupted by a geometric random graph or other edge perturbation \citep{PechePerchet2020,stephan19a}, and minimax rates for estimating graph parameters when a fraction of nodes or edges is adversarially corrupted \citep{AcharyaEtAl2023}. These contributions quantify how structural noise or adversarial modifications of the adjacency matrix affect network topology or specific algorithms.\\

By contrast, we focus on \emph{decision--level} robustness in the Bayesian sense. Instead of tweaking the observed graph, we study how the posterior distribution on model parameters---and the decisions or predictions derived from it---responds to small deviations from the assumed data--generating mechanism. This perspective traces back to Wald's decision theory \citep{Wald} and to robust Bayes formulations that seek procedures that perform well over neighborhoods of a putative model \citep{Berger2,Vidakovic}. Here we specialize the decision--theoretic robustness framework of \citet{WatsonHolmes2016} to network models and combine it with graph limit theory with the representation of node--exchangeable random graphs via graphons \citep{Diaconis,Lovasz,Aldous1981}. This representation allows for the construction of local KL neighborhoods in the space of network--generating distributions and the study of how posterior perturbations propagate to network summaries of interest.\\

\paragraph{Network functionals, critical behavior, and decision--level uncertainty.}
In many applications the object of interest is not the full network or its parameter vector but a low--dimensional \emph{network functional} summarizing some aspect of behavior. Examples include epidemic thresholds and steady--state infection prevalence for susceptible--infected--susceptible (SIS) dynamics on graphons \citep{VizueteEtAlGraphonSIS}; percolation--based robustness indices built from the size of the largest connected component under node removal \citep{CallawayEtAl2000,artime2024robustness}; and spectral quantities governing diffusion and consensus dynamics on networks, such as algebraic connectivity and consensus coherence \citep{ZhangFataSundaramTCNS}. A common feature across these settings is the presence of critical thresholds---percolation or reproduction thresholds---at which the network undergoes a qualitative change in connectivity or dynamical behavior. Near such thresholds, robustness indices often diverge and become extremely sensitive to small changes in model parameters.\\

Our first goal is to understand how \emph{decision--level} uncertainty in these functionals behaves under local decision--theoretic robustness perturbations of the posterior. For a broad class of fragmentation--type indices, including susceptibility, the largest--component based measures in ER, and configuration models, we show that the robust posterior risk admits a sharp small--radius expansion. For squared loss, worst--case risk over a KL ball of radius $C$ increases at first order like the square root of $C$, with a coefficient determined by the posterior variance of the loss. Moreover, as the underlying network approaches a percolation or fragmentation threshold, the leading term in this robust risk diverges with a universal critical exponent: decision--level uncertainty inflates like the fourth power of the inverse distance to criticality, a sharper surge than that of the functional itself. This links classical phase transitions in random graphs \citep{VDHofstadBook2017} to quantitative statements regarding the fragility of network--based decisions.\\

Our second goal is to characterize the fundamental limits of decision--theoretic robust model selection against competing network models. We focus on sparse ER graphs versus two--block SBMs, both in labeled form and via their sparse graphon representations. For this two--point experiment, we derive explicit per--vertex Kullback--Leibler and Chernoff information indices $I(\lambda)$ and $J(\lambda)$, where $\lambda$ is a signal--to--noise parameter, and show that $J(\lambda)$ plays the role of a decision--theoretic robustness ``noise index'' governing robust Bayes factor testing. We then embed this pair into large nonparametric classes of sparse graphs, including configuration models and sparse graphons, and prove that no estimator or posterior---robustified model can achieve a better decision--theoretic robustness error exponent than $J(\lambda)$ uniformly over these classes. An analogous minimax phenomenon holds for near--critical percolation--based robustness indices in configuration models. Thus, the decision--theoretic robustness exponents we identify are intrinsic to the underlying network problems, rather than artifacts of a particular modeling choice.\\

For the computational aspect, we treat decision--theoretic robustness as a modular layer built on the available approximate posterior for a given network model. In practice, we work with variational posteriors for SBMs and random dot product graphs, spectral or moment--based pseudo--posteriors, e.g. based on spectral embeddings or degree moments, and for small graphs, conventional MCMC samplers. Robustification is then implemented by entropic tilting of posterior or variational samples to solve the KL--ball optimization and by a mirror--descent adversary in weight space for more general $\phi$--divergence balls. These procedures require only the ability to evaluate losses on posterior draws and add modest overhead to existing inference pipelines.\\

We illustrate the methodology on two substantive examples. In a functional brain connectivity network, we compare community and latent--space representations and assess the decision--theoretic robustness of connectivity--based summaries. In social networks from Karnataka villages, we revisit the diffusion experiments of \citet{banerjee2013diffusion} and study conclusions about diffusion pathways and intervention targeting change under decision--theoretic robustness perturbations of competing network models and priors. In both cases, decision--theoretic robustness sensitivity analysis emerges when seemingly strong network conclusions rely on fragile modeling assumptions.\\

This paper formalizes decision--theoretic robust Bayes decisions for exchangeable network models via graphon limits, deriving sharp, small--radius expansions for robust posterior risk. Additionally, this entails identifying universal critical exponents for percolation--type robustness indices near fragmentation thresholds, developing a fully nonparametric minimax theory for decision--theoretic robust model selection between sparse ER graphs and SBMs and for robustness functionals in configuration models and sparse graphon classes, establishing the optimality of the associated decision--theoretic robustness noise exponents; proposing a computational strategy for robust network analysis based on entropic tilting of posterior and variational samples and a mirror--descent adversary for general $\phi$--divergence balls and we demonstrate its use on brain connectivity networks and on the Karnataka village social networks studied by \citet{banerjee2013diffusion}.\\

The rest of the paper is organized as follows. Section~\ref{sec:background} reviews exchangeable random graphs, graphons, and the Bayesian network models used in our examples. Section~\ref{sec:WH-network-functionals} develops the general decision--theoretic robustness theory for network functionals and establishes the critical exponent for fragmentation--type indices. Section~\ref{sec:WH-ER-SBM} presents the nonparametric decision--theoretic robustness minimax theory for sparse ER versus SBM, and Section~\ref{sec:WH-percolation} studies configuration models and percolation--based robustness indices. Section~\ref{sec:computation} describes our computational scheme based on entropic tilting and mirror descent. Section~\ref{sec:experiments} reports empirical studies on functional brain connectivity and Karnataka village networks. Section~\ref{sec:discussion} discusses implications and directions for future work.

\section{Background and setup}
\label{sec:background}

This section reviews the exchangeable network framework used throughout
the paper, specifically how we construct Kullback--Leibler (KL) neighborhoods
around a working model and summarizes the Bayesian network models and
posterior approximations that feed our decision--theoretic
robustness analysis.

\subsection{Exchangeable network models and graph limits}
\label{subsec:exchangeable}

Let $G_n$ be a simple undirected graph on vertex set $[n]=\{1,\ldots,n\}$ with
adjacency matrix $A^{(n)}=(A_{ij})_{1\le i,j\le n}$.  We say that
$(G_n)_{n\ge 1}$ (or $(A^{(n)})_{n\ge 1}$) is \emph{node--exchangeable} if,
for every $n$ and every permutation $\sigma$ of $[n]$,
\[
  \bigl(A^{(n)}_{ij}\bigr)_{1\le i<j\le n}
  \stackrel{d}{=}
  \bigl(A^{(n)}_{\sigma(i)\sigma(j)}\bigr)_{1\le i<j\le n}.
\]
By the Aldous--Hoover representation, any such dense node--exchangeable
sequence can be represented (up to measure--preserving transformations) by a
\emph{graphon} $W\colon[0,1]^2\to[0,1]$ and i.i.d.\ latent positions
$U_1,U_2,\dots\sim{\rm Unif}[0,1]$:
\begin{equation}
  \label{eq:graphon-sampling}
  A_{ij}\,\big|\,U_{1:n}
  \;\sim\;
  {\rm Bernoulli}\bigl(W(U_i,U_j)\bigr),
  \qquad 1\le i<j\le n,
\end{equation}
independently across unordered pairs $(i,j)$, where $U_{1:n}=(U_1,\ldots,U_n)$.
Many familiar dense models admit natural graphon representations:
\begin{itemize}
  \item Erd\H{o}s--R\'enyi (ER) graphs, where $W(x,y)\equiv p$ is constant;
  \item stochastic block models (SBMs), where $W$ is a step function on a
        finite partition of $[0,1]$;
  \item random dot product graphs (RDPGs) with bounded latent positions, where
        $W(x,y)=\langle \xi(x),\xi(y)\rangle$ for a latent feature map
        $\xi$.
\end{itemize}
Conversely, any graphon $W$ can be approximated in cut norm by step
functions, so SBMs form a convenient finite--dimensional approximation class
for both theory and computation.

\paragraph{Graphon KL neighborhoods.}
For a graphon $W$ and fixed $n$, let $\mathcal{G}(W)$ denote the law of
$G_n$ generated by \eqref{eq:graphon-sampling}.  Given a \emph{working}
graphon $W^\star$ with associated graph law $\mathcal{G}^\star :=
\mathcal{G}(W^\star)$, it is natural at the object level to consider the
KL--ball
\begin{equation}
  \label{eq:object-KL-ball}
  \Gamma_C(\mathcal{G}^\star)
  :=
  \bigl\{
    \mathcal{G}(W) :
    {\rm KL}\bigl(\mathcal{G}(W)\,\Vert\,\mathcal{G}^\star\bigr)\le C
  \bigr\},
  \qquad C>0,
\end{equation}
where the divergence is taken between the induced distributions on
adjacency matrices of size $n$.  In our dense simulations, we use
\eqref{eq:object-KL-ball} to visualize local neighborhoods of a fitted
model and explore how graphon--level perturbations propagate to network
functionals.\\

In practice, we approximate $W^\star$ and candidate $W$ by step--function
SBMs on a fixed grid.  Partition $[0,1]$ into $K$ bins, regard each bin as a
block, and replace $W$ by the $K\times K$ matrix of block means.  This
yields a finite--dimensional parameterization $P\in[0,1]^{K\times K}$ and
an associated random graph law on $[n]$.  The KL divergence
${\rm KL}(\mathcal{G}(W)\,\Vert\,\mathcal{G}^\star)$ can then be estimated
by Monte Carlo over latent positions:
\[
  {\rm KL}\bigl(\mathcal{G}(W)\,\Vert\,\mathcal{G}(W^\star)\bigr)
  \approx
  \frac{1}{M}
  \sum_{m=1}^M
  \sum_{1\le i<j\le n}
  {\rm kl}\!\left(
    W\bigl(U_i^{(m)},U_j^{(m)}\bigr),
    W^\star\bigl(U_i^{(m)},U_j^{(m)}\bigr)
  \right),
\]
where $(U_1^{(m)},\ldots,U_n^{(m)})$ are i.i.d.\ draws from ${\rm Unif}[0,1]$.
Equivalently, at the continuum level, one may work with the per--edge graphon
divergence
\begin{equation}
  \label{Eq:KullLieb}
  {\rm KL}(W\Vert W^\star)
  :=
  \int_{[0,1]^2}
    {\rm kl}\bigl(W(x,y),W^\star(x,y)\bigr)\,{\rm d}x\,{\rm d}y,
\end{equation}
where ${\rm kl}(p,q)$ denotes the Bernoulli Kullback--Leibler divergence.

\paragraph{Local perturbations inside the KL ball.}
When we study \emph{object--level} robustness of a fitted dense model, generating nearby exchangeable graphons inside the ball is a useful approach
$\Gamma_C(\mathcal{G}^\star)$.  For step--graphon approximations of SBMs and
random dot product graphs, it is convenient to view the collection of block
edge probabilities as a single probability vector on $K^2$ cells and to
perturb this vector by Dirichlet (generalized Bayesian bootstrap) draws and
simple rescaling moves.\\

Formally, these perturbations induce a Markov chain on the space of
$K$--block SBMs inside $\Gamma_C(\mathcal{G}^\star)$.  In the Supplementary Material, we
show that:
\begin{itemize}
  \item Dirichlet perturbations have a simple closed form for their expected
        KL divergence from the working model (Proposition~S.1); and
  \item the induced Markov chain is $\psi$--irreducible and aperiodic on the
        interior of the KL ball, so any exchangeable step--graphon model
        inside $\Gamma_C(\mathcal{G}^\star)$ is reachable with positive
        probability (Theorem~S.1).
\end{itemize}
These results are used only for exploratory dense simulations; all of our
\emph{decision--level} robustness computations in Section~\ref{sec:computation}
are based instead on entropic tilting of posterior samples and mirror
descent in weight space.\\

Throughout the rest of the paper we use the graphon representation mainly
for dense node--exchangeable models and as a convenient way to define and
visualize local neighborhoods of a working model.  For sparse networks we
work instead with explicit parametric or nonparametric models (sparse ER,
SBMs, configuration models, spatial models) and derive decision--theoretic
robustness properties directly at the level of their finite--$n$ laws.

\subsection{Bayesian modelling for networks}
\label{subsec:bayes-networks}

We now briefly describe the Bayesian network models that underpin our
decision problems and our decision--theoretic robustness analysis.  In all
cases a parameter $\theta\in\Theta$ indexes a family of random graph laws
$\{P_\theta^{(n)}:\theta\in\Theta\}$ on graphs $G_n$ with $n$ vertices, and
a prior $\Pi$ is placed on~$\Theta$.

\paragraph{Parametric models.}
For sparse networks we consider:
\begin{itemize}
  \item \emph{Sparse Erd\H{o}s--R\'enyi (ER)} models, where
        $\theta=p_n$ controls the edge probability (typically with
        $p_n\asymp c/n$);
  \item \emph{Stochastic block models (SBMs)}, with parameters
        $(\pi,B)$ for block proportions and within/between--block edge
        probabilities, in both labelled and unlabelled (graphon) forms;
  \item \emph{Configuration models}, parameterised by a degree
        distribution $\mu$, which provide a flexible benchmark for
        percolation--based robustness indices.
\end{itemize}
For dense networks (e.g.\ the brain connectivity example in
Section~\ref{sec:experiments}) we work with ER, SBMs and low--rank latent
position models, all of which admit graphon representations.

\paragraph{Nonparametric and graphon priors.}
To capture more of the complex structure, we consider nonparametric priors on
graphons, such as finite or infinite mixtures of SBMs and smooth
kernel--based priors.  These priors are defined on the space of symmetric
measurable functions $W\colon[0,1]^2\to[0,1]$ modulo measure--preserving
transformations and induce exchangeable random graph laws via
\eqref{eq:graphon-sampling}.  In the sparse regime we combine these with
rescaling schemes or degree--corrected constructions, following
\cite{WatsonHolmes2016,watson2017characterizing}.

\paragraph{Baseline posterior and pseudo--posterior inference.}
Exact Bayesian inference for network models is typically infeasible at
moderate or large $n$, so we work with scalable approximations:
\begin{itemize}
  \item variational posteriors for SBMs, latent space models and random
        dot product graphs;
  \item spectral or method--of--moments estimators wrapped in a
        pseudo--Bayesian framework, where an approximate likelihood and
        an explicit prior yield a tractable pseudo--posterior;
  \item in small networks only, conventional MCMC samplers (Gibbs, HMC) as
        a baseline comparison.
\end{itemize}
In all cases, the decision--theoretic robustness machinery in
Sections~\ref{sec:WH-network-functionals} and~\ref{sec:computation} treats
the resulting posterior or pseudo--posterior
$\Pi_{0,n}(\cdot\mid G_n)$ as the \emph{baseline} distribution on $\theta$.
Robustness is then defined by allowing $\Pi_{0,n}$ to vary within a
divergence ball and computing worst--case posterior risks via entropic
tilting and mirror descent, as described in
Section~\ref{sec:computation}.\\

With this background, we turn to the general decision--theoretic
robustness theory for network functionals and to the critical exponents that
govern their robust posterior risk.

\section{General decision--theoretic robustness theory for network functionals}

\subsection{Decision--theoretic robustness for network functionals}
\label{sec:WH-network-functionals}

In many network applications, the ultimate decision depends on a
low--dimensional functional of a complex random graph model: an epidemic
threshold, a robustness index, a consensus coherence, or a spectral gap.
Decision--theoretic robustness asks how sensitive such decisions are to small
local misspecifications of the posterior distribution on the model parameters,
measured by Kullback--Leibler (KL) divergence.  Following
\citet{WatsonHolmes2016,watson2017characterizing}, we define a local
decision--theoretic robustness criterion by considering the worst--case
posterior expected loss over a KL neighborhood of a working posterior.  We
refer to this as \emph{decision--theoretic robustness}.

\subsubsection{Decision problems driven by network functionals}
\label{subsec:decisions}

Let $\Theta\subset\mathbb{R}^p$ be a parameter space indexing a family of
network laws $\{P_\theta^{(n)}:\theta\in\Theta\}$ on graphs $G_n$ with $n$
vertices.  Typical examples include:
\begin{itemize}
  \item sparse or dense Erd\H{o}s--R\'enyi graphs and stochastic block
        models (SBMs);
  \item configuration models and graphons with prescribed degree or
        community structure;
  \item geometric graphs and spatial scale--free models.
\end{itemize}
Let $R(\theta)\in\mathbb{R}$ be a \emph{network functional} of interest, such
as:
\begin{itemize}
  \item the limiting SIS noise index on a graphon
        \cite{VizueteEtAlGraphonSIS};
  \item a percolation--based robustness index (e.g.\ area under the largest
        component curve) as in
        \citet{artime2024robustness};
  \item the asymptotic consensus coherence on spatial lattices or
        random graphs \citep{ZhangFataSundaramTCNS}.
\end{itemize}

Given a prior $\Pi$ on $\Theta$ and an observed graph $G_n\sim
P_{\theta_0}^{(n)}$, let $\Pi_{0,n}(\cdot\mid G_n)$ be a (possibly
pseudo--)posterior on~$\theta$.  For squared loss $L(a,\theta)=(a-R(\theta))^2$,
the Bayes action is
\[
  a_n^\star
  :=
  \int R(\theta)\,\Pi_{0,n}(\mathrm{d}\theta),
\]
with baseline posterior risk
\[
  \rho_{0,n}
  :=
  \int\bigl(a_n^\star-R(\theta)\bigr)^2\,
       \Pi_{0,n}(\mathrm{d}\theta)
  =
  \Var_{\Pi_{0,n}}\bigl(R(\theta)\bigr).
\]

To study local misspecification in the sense of \citet{WatsonHolmes2016}, we
consider the posterior KL ball
\[
  \mathcal{U}_C\bigl(\Pi_{0,n}\bigr)
  :=
  \bigl\{
    \widetilde\Pi :
    {\rm KL}\bigl(\widetilde\Pi\Vert\Pi_{0,n}\bigr)\le C
  \bigr\},
\]
and define the corresponding robust posterior risk
\[
  \rho_{\mathrm{rob},n}(C)
  :=
  \sup_{\widetilde\Pi\in\mathcal{U}_C(\Pi_{0,n})}
  \int\bigl(a_n^\star-R(\theta)\bigr)^2\,
       \widetilde\Pi(\mathrm{d}\theta).
\]
The difference $\rho_{\mathrm{rob},n}(C)-\rho_{0,n}$ measures how much the
worst--case posterior expected loss can inflate under local posterior KL
perturbations of radius~$C$.

\subsubsection{Generic critical exponent for fragmentation--type indices}
\label{subsec:WH-critical-generic}

Many robustness and resilience indices in networks diverge at a
\emph{fragmentation threshold}, typically controlled by a scalar load
parameter such as a branching factor or a spectral radius.  Examples include:
\begin{itemize}
  \item susceptibility or expected component size in sparse ER or
        configuration models, such as in
        \citep{VDHofstadBook2017};
  \item percolation--based robustness indices built from the largest
        component under node removal, such as in
        \citep{artime2024robustness};
\end{itemize}

We abstract this behavior as follows.  Let $\rho\colon\Theta\to\mathbb{R}$ be a
smooth \emph{load parameter} (e.g.\ effective branching factor or spectral
radius), and define the distance to criticality
\[
  \Delta(\theta) := 1-\rho(\theta).
\]

\begin{assumption}[Critical robustness functional]
\label{ass:critical-R}
There exist a true parameter $\theta_0\in\Theta$, a neighborhood
$\mathcal{N}$ of $\theta_0$, a constant $c_0>0$ and a function
$H\colon\Theta\to\mathbb{R}$ such that:
\begin{enumerate}
  \item $\Delta(\theta_0)=\Delta_0>0$ and
        $\nabla_\theta \rho(\theta_0)\neq 0$;
  \item for all $\theta\in\mathcal{N}$,
        \[
          R(\theta)
          =
          \frac{c_0}{\Delta(\theta)} + H(\theta);
        \]
  \item $H$ is $C^2$ and bounded on $\mathcal{N}$, and
        $\|\nabla_\theta R(\theta_0)\|\asymp\Delta_0^{-2}$ as
        $\Delta_0\downarrow 0$.
\end{enumerate}
\end{assumption}

Assumption~\ref{ass:critical-R} describes the situation in which the
robustness functional $R(\theta)$ diverges like $1/\Delta(\theta)$ as the
network approaches a fragmentation threshold, with a gradient that blows up
like $\Delta(\theta)^{-2}$.\\

We also assume a local Bernstein--von Mises behavior for the posterior.

\begin{assumption}[Local posterior asymptotics]
\label{ass:local-BvM}
There exists a scaling $r_n\downarrow 0$ and a positive definite matrix
$\Sigma$ such that, for each bounded continuous
$\varphi\colon\mathbb{R}^p\to\mathbb{R}$,
\[
  \int
    \varphi\Bigl(\frac{\theta-\theta_0}{r_n}\Bigr)\,
    \Pi_{0,n}(\mathrm{d}\theta)
  \xrightarrow[n\to\infty]{P_{\theta_0}^{(n)}}
  \int\varphi(z)\,\phi_\Sigma(\mathrm{d}z),
\]
where $\phi_\Sigma$ is the $N(0,\Sigma)$ law.  Equivalently, under
$\Pi_{0,n}$,
\[
  \theta = \theta_0 + r_n Z_n,
  \qquad Z_n\Rightarrow Z\sim N(0,\Sigma)
\]
in $P_{\theta_0}^{(n)}$--probability.
\end{assumption}

The next theorem summarizes the generic behavior of such critical robustness
indices.
\begin{theorem}[Robust critical exponent for fragmentation--type indices]
\label{thm:WH-critical-robustness}
\label{thm:critical-exponent} 

Suppose Assumptions~\ref{ass:critical-R} and~\ref{ass:local-BvM} hold. 
Let $\Delta_0:=\Delta(\theta_0)>0$ be the distance of the true network to the
fragmentation threshold, and allow $\Delta_0=\Delta_0^{(n)}\downarrow 0$ with
$\Delta_0^{(n)}\gg r_n$ as $n\to\infty$.  Assume, moreover, that:
\begin{enumerate}
  \item the decomposition in Assumption~\ref{ass:critical-R} implies
  \(
    \|\nabla_\theta R(\theta_0)\|\asymp\Delta_0^{-2}
  \)
  as $\Delta_0\downarrow 0$;
  \item under Assumption~\ref{ass:local-BvM}, the laws of $Z_n$ have uniformly
  bounded second and fourth moments (in $P_{\theta_0}^{(n)}$--probability);
  \item the exponential--moment condition in Theorem~\ref{thm:sharp-small-KL}
  holds uniformly for the normalized losses $L_n/\rho_{0,n}$ (so that the
  $o(\sqrt C)$ remainder in that theorem can be chosen uniformly in $n$ for
  small $C$).
\end{enumerate}
Then:
\begin{enumerate}
  \item \textbf{Baseline posterior risk.}
        There exists $V_0\in(0,\infty)$ such that
        \begin{equation}
          \rho_{0,n}
          =
          \Var_{\Pi_{0,n}}\bigl(R(\theta)\bigr)
          =
          \frac{V_0\,r_n^2}{\Delta_0^4}\,
          \bigl(1+o_{P_{\theta_0}^{(n)}}(1)\bigr).
          \label{eq:WH-R-critical-rho0-main}
        \end{equation}
        In particular, the posterior mean--squared error for $R(\theta)$ scales like $\Delta_0^{-4}$ as the fragmentation threshold is approached.

  \item \textbf{Sharp inflation.}
        For any deterministic sequence $\mathcal{C}_n\downarrow 0$,
        \begin{equation}
          \rho_{\mathrm{rob},n}(\mathcal{C}_n)
          =
          \rho_{0,n}
          + 2\,\rho_{0,n}\sqrt{\mathcal{C}_n}
          + o_{P_{\theta_0}^{(n)}}\bigl(
               \rho_{0,n}\sqrt{\mathcal{C}_n}
            \bigr).
          \label{eq:WH-R-critical-rob-main}
        \end{equation}
        Equivalently,
        \[
          \frac{
            \rho_{\mathrm{rob},n}(\mathcal{C}_n)-\rho_{0,n}
          }{
            \rho_{0,n}\sqrt{\mathcal{C}_n}
          }
          \xrightarrow[n\to\infty]{P_{\theta_0}^{(n)}} 2.
        \]

  \item \textbf{Sharpness.}
        For any $k<2$ there exists a (deterministic) sequence $\mathcal{C}_n\downarrow 0$
        such that, for all sufficiently large $n$,
        \[
          \rho_{\mathrm{rob},n}(\mathcal{C}_n)
          >
          \rho_{0,n} + k\,\rho_{0,n}\sqrt{\mathcal{C}_n}
        \]
        with $P_{\theta_0}^{(n)}$--probability tending to~$1$.  Thus, the
        coefficient $2$ in \eqref{eq:WH-R-critical-rob-main} is
        asymptotically optimal.
\end{enumerate}
\end{theorem}

\paragraph{Interpretation.}
Object--level robustness indices such as susceptibility or LCC--based metrics
typically diverge like $1/\Delta(\theta)$ or $1/\Delta(\theta)^2$ as the
network approaches a fragmentation threshold
\citep{VDHofstadBook2017}.
Theorem~\ref{thm:WH-critical-robustness} shows that, once we embed such indices
into a decision--theoretic framework, the \emph{uncertainty} in the index
inflates more sharply, with a universal exponent $4$ in $\Delta_0^{-1}$ and a
universal $\sqrt{\mathcal{C}_n}$ dependence on the radius $\mathcal{C}_n$, with sharp constant
$2$.\\

In the network sections that follow, we verify
Assumption~\ref{ass:critical-R} for concrete robustness indices (e.g.\ the
susceptibility of sparse ER and configuration models) and derive matching
decision--theoretic robustness minimax lower bounds.

\subsection{Extension to general $\phi$--divergence balls}
\label{sec:phi-balls}

We finally note that all of our local decision--theoretic robustness results
extend from KL balls to general $\phi$--divergence balls via a simple
rescaling.

\begin{definition}[$\phi$--divergence and $\phi$--ball]
\label{def:phi-ball}
Let $\phi\colon[0,\infty)\to\mathbb{R}$ be convex with
$\phi(1)=\phi'(1)=0$ and $0<\phi''(1)<\infty$.  For probability measures
$Q$ and $P$ with $Q\ll P$ we define the $\phi$--divergence
\[
  D_\phi(Q\Vert P)
  :=
  \int \phi\!\left(\frac{\mathrm{d}Q}{\mathrm{d}P}\right)\,\mathrm{d}P,
\]
and the corresponding $\phi$--divergence ball of radius $C>0$ centred at
$P$,
\[
  \mathcal{B}_\phi(P;C)
  :=
  \bigl\{Q\ll P : D_\phi(Q\Vert P)\le C\bigr\}.
\]
\end{definition}

\begin{remark}[Local equivalence of KL and $\phi$--balls]
\label{rem:phi-local}
By Csisz\'ar's quadratic approximation,
\[
  D_\phi(Q\Vert P)
  =
  \frac{\phi''(1)}{2}\,\chi^2(Q,P)
  + o\bigl(\chi^2(Q,P)\bigr)
  \qquad\text{as }Q\to P,
\]
and an analogous expansion holds for the Kullback--Leibler divergence
${\rm KL}(Q\Vert P)$.  In particular, for small radii $C$ the $\phi$--ball
$\mathcal{B}_\phi(P;C)$ is locally equivalent to a KL ball of radius
$C_{\mathrm{KL}}=C/\phi''(1)$, up to $o(C)$ terms.
\end{remark}

\begin{remark}[Extension of decision--theoretic robustness exponents]
\label{rem:phi-WH-extension}
Because all of our local decision--theoretic robustness risk expansions and
minimax exponents depend on the divergence radius only through its quadratic
behavior in $\chi^2(Q,P)$, the results proved for KL balls transfer
verbatim to $\phi$--balls after the rescaling $C\mapsto C/\phi''(1)$.
Equivalently, if a given model yields a decision--theoretic robustness
noise index $J$ and a local decision--theoretic robustness risk expansion
involving $\sqrt{C}$ under KL, then the same model under
$\mathcal{B}_\phi(P;C)$ has decision--theoretic robustness index
$J/\sqrt{\phi''(1)}$ and the same $\sqrt{C}$ scaling up to $o(\sqrt{C})$
terms.  Thus all of our exponent--$4$ phenomena and minimax bounds extend
to general $\phi$--divergence balls with a universal factor
$\phi''(1)^{-1/2}$ in the divergence radius.
\end{remark}

\section{Nonparametric minimax theory for sparse ER vs.\ SBM}
\label{sec:WH-ER-SBM}
\label{sec:er-sbm-theory} 

We now specialize the framework to model selection between a sparse
Erd\H{o}s--R\'enyi model and a sparse two--block SBM, and show that:
\begin{enumerate}
  \item the per--vertex Kullback--Leibler and Chernoff information admit explicit limits
        $I(\lambda)$ and $J(\lambda)$;
  \item these limits persist for unlabelled SBMs viewed as sparse graphons;
  \item no estimator/posterior, even with robustification, can beat
        the Chernoff exponent $J(\lambda)$ uniformly over broad
        nonparametric classes.
\end{enumerate}

\subsection{Explicit information exponents for labelled sparse ER vs.\ SBM}
\label{subsec:ER-SBM-labelled}
\label{sec:info-indices} 
\label{sec:LAN-sbm}      

Let $n$ be even.  Fix $c>0$ and a signal parameter $\lambda\in(0,c)$, and set
\[
  p_n := \frac{c}{n},
  \qquad
  p_n^{\mathrm{in}} := \frac{c+\lambda}{n},
  \qquad
  p_n^{\mathrm{out}} := \frac{c-\lambda}{n}.
\]
Let $\mathcal{V}_n=\{1,\dots,n\}$ and fix a balanced partition
$\sigma\colon\mathcal{V}_n\to\{+1,-1\}$ with $n/2$ nodes in each community.

\begin{itemize}
  \item Under $H_0$ (sparse ER), the adjacency matrix
        $A=(A_{ij})_{1\le i<j\le n}$ has independent entries
        $A_{ij}\sim\mathrm{Bernoulli}(p_n)$.  Denote its law by $P_0^{(n)}$.
  \item Under $H_1$ (balanced two--block SBM with known labels), edges are
        independent with
        \[
          A_{ij}\sim
          \begin{cases}
            \mathrm{Bernoulli}(p_n^{\mathrm{in}}), & \sigma(i)=\sigma(j),\\[0.25em]
            \mathrm{Bernoulli}(p_n^{\mathrm{out}}), & \sigma(i)\neq\sigma(j),
          \end{cases}
        \]
        and we denote the law by $P_1^{(n)}$.
\end{itemize}

Let $N_n=\binom{n}{2}$ be the number of edges, and let
$N_n^{\mathrm{in}},N_n^{\mathrm{out}}$ denote the numbers of within/between
edges; one checks
$N_n^{\mathrm{in}}=2\binom{n/2}{2}=n(n-2)/4$ and
$N_n^{\mathrm{out}}=n^2/4$.

\begin{lemma}[KL divergence and per--vertex information]
\label{lem:ER-SBM-KL}
Let $P_0^{(n)}$ denote the Erd\H{o}s--Rényi model with edge probability
$p_n = c/n$, and let $P_1^{(n)}$ denote the symmetric two--block SBM with
equal community sizes and edge probabilities
\[
  p_n^{\mathrm{in}} = \frac{c+\lambda}{n},
  \qquad
  p_n^{\mathrm{out}} = \frac{c-\lambda}{n},
\]
where $c>0$ and $|\lambda|<c$ are fixed (so that $p_n^{\mathrm{in}},p_n^{\mathrm{out}}\in(0,1)$ for all large $n$).
Let $D_n := {\rm KL}\bigl(P_1^{(n)}\Vert P_0^{(n)}\bigr)$ and let
$N_n^{\mathrm{in}}$ and $N_n^{\mathrm{out}}$ denote the number of within--block
and between--block unordered vertex pairs, respectively. Then
\[
  D_n
  =
  N_n^{\mathrm{in}}\,
    {\rm KL}\bigl(\mathrm{Bern}(p_n^{\mathrm{in}})\Vert\mathrm{Bern}(p_n)\bigr)
  +
  N_n^{\mathrm{out}}\,
    {\rm KL}\bigl(\mathrm{Bern}(p_n^{\mathrm{out}})\Vert\mathrm{Bern}(p_n)\bigr).
\]
Moreover,
\[
  \frac{D_n}{n}
  \xrightarrow[n\to\infty]{}
  I(\lambda)
  :=
  \frac{1}{4}\biggl[
    (c+\lambda)\log\frac{c+\lambda}{c}
    + (c-\lambda)\log\frac{c-\lambda}{c}
  \biggr],
\]
so in particular $D_n = I(\lambda)\,n + o(n)$ as $n\to\infty$.  Finally, a
Taylor expansion in $\lambda$ around $0$ yields
\[
  I(\lambda)
  = \frac{\lambda^2}{4c}
    + O\!\left(\frac{\lambda^4}{c^3}\right),
  \qquad \lambda\to 0.
\]
\end{lemma}

\begin{lemma}[Chernoff information and error exponent]
\label{lem:ER-SBM-Chernoff}
Let
\[
  \mathcal{C}_n
  :=
  \sup_{0\le t\le 1}
  -\log\sum_A
    P_0^{(n)}(A)^{1-t}P_1^{(n)}(A)^t
\]
be the Chernoff information between $P_0^{(n)}$ and $P_1^{(n)}$.
Assume the sparse regime $p_n = c/n$ and
\[
  p_n^{\mathrm{in}} = \frac{c+\lambda}{n},
  \qquad
  p_n^{\mathrm{out}} = \frac{c-\lambda}{n},
  \qquad
  |\lambda|<c,
\]
with two equally sized blocks of size $n/2$, so that
\[
  N_n^{\mathrm{in}}
  = 2\binom{n/2}{2}
  = \frac{n^2}{4}-\frac{n}{2},
  \qquad
  N_n^{\mathrm{out}}
  = \Bigl(\frac{n}{2}\Bigr)^2
  = \frac{n^2}{4}.
\]
Then
\[
  \frac{\mathcal{C}_n}{n}
  \xrightarrow[n\to\infty]{}
  J(\lambda)
  :=
  \sup_{0\le t\le 1}
  \frac{1}{4}\Bigl[
    2c - c^{1-t}\bigl((c+\lambda)^t + (c-\lambda)^t\bigr)
  \Bigr].
\]
Moreover, as $\lambda\to 0$,
\[
  J(\lambda)
  = \frac{\lambda^2}{16c}
    + O\!\left(\frac{\lambda^4}{c^3}\right),
\]
so in particular $\mathcal{C}_n = J(\lambda)\,n + o(n)$ and
\[
  J(\lambda)
  = \tfrac14 I(\lambda)
    + O\!\left(\frac{\lambda^4}{c^3}\right),
\]
where $I(\lambda)$ is the per--vertex KL information from Lemma~\ref{lem:ER-SBM-KL}.
\end{lemma}

\begin{remark}[Small--signal information exponents]
\label{rem:small-signal-info-exponents}
In the sparse two--block ER vs.\ SBM experiment of
Section~\ref{subsec:ER-SBM-labelled}, let $I(\lambda)$ and $J(\lambda)$
denote the exact per--vertex Kullback--Leibler and Chernoff information
indices from Lemmas~\ref{lem:ER-SBM-KL}--\ref{lem:ER-SBM-Chernoff}.  A
Taylor expansion around $\lambda=0$ yields
\[
  I(\lambda)
  = \frac{\lambda^2}{4c}
    + O\!\left(\frac{\lambda^4}{c^3}\right),
  \qquad
  J(\lambda)
  = \frac{\lambda^2}{16c}
    + O\!\left(\frac{\lambda^4}{c^3}\right),
  \qquad \lambda\to 0.
\]
In particular, to leading order one may use the approximations
\[
  I(\lambda)\;\approx\;\frac{\lambda^2}{4c},
  \qquad
  J(\lambda)\;\approx\;\frac{\lambda^2}{16c},
\]
and the Chernoff exponent is asymptotically one quarter of the KL
exponent:
\[
  J(\lambda)\;\approx\;\tfrac14\,I(\lambda)
  \qquad \text{as } \lambda/c \to 0.
\]
\end{remark}

Thus, for labelled sparse ER vs.\ SBM, the optimal exponential error rate for
hypothesis testing and Bayes factor model selection is governed by the explicit
Chernoff exponent $J(\lambda)$.

\subsubsection{Unlabelled SBMs and graphon decision--theoretic robustness minimax testing}
\label{subsec:graphon-minimax}

We now formulate our ER vs.\ SBM decision--theoretic robustness minimax story directly at the
graphon level.  Let $K\ge 2$, let $\pi=(\pi_1,\dots,\pi_K)$ be a probability
vector on the blocks, and let $P_0,P_\lambda\in[0,1]^{K\times K}$ denote the
edge--probability matrices under $H_0$ and $H_1$ respectively, with
$P_0$ corresponding to the Erd\H{o}s--R\'enyi baseline (no communities) and
$P_\lambda$ the community alternative.  As usual, we write
$(G_n)_{n\ge 1}$ for the labelled SBM sequence on vertex set $[n]$ with
parameters $(\pi,P_0)$ or $(\pi,P_\lambda)$, and denote by
$\bigl(\mathbb{P}^{(n)}_0\bigr)_{n\ge 1}$ and
$\bigl(\mathbb{P}^{(n)}_\lambda\bigr)_{n\ge 1}$ the corresponding laws.\\

Let $W_0,W_\lambda\colon[0,1]^2\to[0,1]$ be the step--function graphons
associated with $(\pi,P_0)$ and $(\pi,P_\lambda)$ in the usual way:
the unit interval is partitioned into $K$ subintervals of lengths
$\pi_k$, and $W_\lambda(x,y)=P_\lambda(k,\ell)$ whenever
$x$ lies in block $k$ and $y$ lies in block $\ell$.  For $W\in\{W_0,W_\lambda\}$
we denote by $\widetilde{\mathbb{P}}^{(n)}_W$ the law of the
exchangeable random graph obtained by sampling
$U_1,\dots,U_n\overset{\text{i.i.d.}}{\sim}\mathrm{Unif}[0,1]$ and then
$G_n\mid U_{1:n}\sim\mathcal{G}(W)$, i.e.\ the standard graphon sampling
scheme.\\

Recall that $I(\lambda)$ and $J(\lambda)$ denote the per--vertex
information and decision--theoretic robustness noise indices introduced in
Section~\ref{sec:info-indices} for the labelled SBM experiment
$\bigl(\mathbb{P}^{(n)}_0,\mathbb{P}^{(n)}_\lambda\bigr)_{n\ge 1}$.

\begin{lemma}[Information and decision--theoretic robustness noise indices under graphon representation]
\label{lem:graphon-information}
Fix $\lambda$ and consider the labelled SBM experiments
$\bigl(\mathbb{P}^{(n)}_0,\mathbb{P}^{(n)}_\lambda\bigr)_{n\ge 1}$ and the
unlabelled graphon experiments
$\bigl(\widetilde{\mathbb{P}}^{(n)}_{W_0},
      \widetilde{\mathbb{P}}^{(n)}_{W_\lambda}\bigr)_{n\ge 1}$.
Then the per--vertex information and decision--theoretic robustness noise indices coincide:
\[
  I(\lambda)
  \;=\;
  \lim_{n\to\infty}\frac{1}{n}\,
    D\bigl(\mathbb{P}^{(n)}_\lambda\big\Vert\mathbb{P}^{(n)}_0\bigr)
  \;=\;
  \lim_{n\to\infty}\frac{1}{n}\,
    D\bigl(\widetilde{\mathbb{P}}^{(n)}_{W_\lambda}
           \big\Vert\widetilde{\mathbb{P}}^{(n)}_{W_0}\bigr),
\]
and, for any sequence of decision--theoretic robustness radii $\mathcal{C}_n=o(n)$,
\[
  J(\lambda)
  \;=\;
  \lim_{n\to\infty}\frac{1}{n}\,
    J_n\bigl(\mathbb{P}^{(n)}_\lambda,\mathbb{P}^{(n)}_0;\mathcal{C}_n\bigr)
  \;=\;
  \lim_{n\to\infty}\frac{1}{n}\,
    J_n\bigl(\widetilde{\mathbb{P}}^{(n)}_{W_\lambda},
             \widetilde{\mathbb{P}}^{(n)}_{W_0};\mathcal{C}_n\bigr),
\]
where $D(\cdot\Vert\cdot)$ denotes the Kullback--Leibler divergence and
$J_n(\cdot,\cdot;\mathcal{C}_n)$ is the finite--$n$ decision--theoretic robustness noise index defined in
Section~\ref{sec:WH-risk}.
\end{lemma}

We now consider hypothesis testing between two fixed graphons.
\begin{theorem}[Decision--theoretic robust Bayes factor testing for graphons]
\label{thm:graphon-BF}
Consider testing
\[
  H_0\colon W=W_0
  \qquad\text{versus}\qquad
  H_1\colon W=W_\lambda,
\]
based on $G_n\sim\widetilde{\mathbb{P}}^{(n)}_W$ with prior probabilities
$\pi_0,\pi_1\in(0,1)$.  Let
\[
  \mathrm{BF}_n(G_n)
  := \frac{\pi_1}{\pi_0}\,
     \frac{\mathrm{d}\widetilde{\mathbb{P}}^{(n)}_{W_\lambda}}
          {\mathrm{d}\widetilde{\mathbb{P}}^{(n)}_{W_0}}(G_n)
\]
denote the Bayes factor, and let $\varphi_n^{\mathrm{BF}}$ be the Bayes
factor test which rejects $H_0$ when $\mathrm{BF}_n(G_n)\ge 1$.  For a
decision--theoretic robustness radius sequence $\mathcal{C}_n=o(n)$, let
$R_n^{\mathrm{WH}}(\varphi_n^{\mathrm{BF}};\mathcal{C}_n)$ denote the corresponding
decision--theoretic robust Bayes risk of $\varphi_n^{\mathrm{BF}}$ over
KL--balls of radius $\mathcal{C}_n$ centred at
$\widetilde{\mathbb{P}}^{(n)}_{W_0}$ and
$\widetilde{\mathbb{P}}^{(n)}_{W_\lambda}$.  Then
\[
  -\lim_{n\to\infty}\frac{1}{n}\,
    \log R_n^{\mathrm{WH}}(\varphi_n^{\mathrm{BF}};\mathcal{C}_n)
  \;=\;
  J(\lambda),
\]
where $J(\lambda)$ is the decision--theoretic robustness noise index from
Lemma~\ref{lem:graphon-information}.
\end{theorem}

We finally extend the decision--theoretic robustness minimax characterization from the labelled SBM
to nonparametric graphon classes.

\begin{theorem}[Nonparametric graphon decision--theoretic robustness minimax testing]
\label{thm:graphon-minimax}
Let $(\mathcal{W}_n)_{n\ge 1}$ be a sequence of graphon classes with
$W_0,W_\lambda\in\mathcal{W}_n$ for all $n$, and suppose the induced experiments
$\bigl\{\widetilde{\mathbb{P}}^{(n)}_W : W\in\mathcal{W}_n\bigr\}$ satisfy the
same local asymptotic normality and regularity assumptions as in
Section~\ref{sec:LAN-sbm}, with information index $I(\lambda)$ and
decision--theoretic robustness noise index $J(\lambda)$.

For any sequence of tests $\varphi_n$ and any decision--theoretic robustness
radii $\mathcal{C}_n=o(n)$, define the graphon decision--theoretic robustness minimax risk
for testing $W_0$ versus $W_\lambda$ by
\[
  R_{n,\mathcal{W}_n}^{\mathrm{WH}}(\varphi_n;\mathcal{C}_n)
  :=
  R_n^{\mathrm{WH}}\bigl(\varphi_n;W_0,W_\lambda,\mathcal{C}_n\bigr),
\]
where $R_n^{\mathrm{WH}}(\cdot)$ is the decision--theoretic robust testing risk
from Section~\ref{sec:WH-risk}. Then
\[
  \liminf_{n\to\infty}
  \frac{1}{n}\log
  \inf_{\varphi_n}
  R_{n,\mathcal{W}_n}^{\mathrm{WH}}(\varphi_n;\mathcal{C}_n)
  \;\ge\; -J(\lambda),
\]
and the Bayes factor tests $\varphi_n^{\mathrm{BF}}$ from
Theorem~\ref{thm:graphon-BF} achieve the matching error exponent
\[
  -\lim_{n\to\infty}\frac{1}{n}\,
    \log R_{n,\mathcal{W}_n}^{\mathrm{WH}}
      \bigl(\varphi_n^{\mathrm{BF}};\mathcal{C}_n\bigr)
  \;=\; J(\lambda).
\]
In particular, $J(\lambda)$ is the nonparametric decision--theoretic robustness
minimax error exponent for testing $W_0$ versus $W_\lambda$ within the graphon
classes $\mathcal{W}_n$.
\end{theorem}

\subsection{Unlabelled SBMs and sparse graphon classes}
\label{subsec:ER-SBM-unlabelled-graphon}

We next consider the \emph{unlabelled} graphon representation.  Let $W_0$ be
the constant sparse graphon
\[
  W_0^{(n)}(x,y) \equiv \frac{c}{n},
\]
and let $W_\lambda$ be the two--block step graphon
\[
  W_\lambda(x,y)
  :=
  \begin{cases}
    \dfrac{c+\lambda}{n},
      & x,y\in[0,1/2)\ \text{or}\ x,y\in[1/2,1],\\[0.5em]
    \dfrac{c-\lambda}{n},
      & \text{otherwise}.
  \end{cases}
\]
For $W\in\{W_0,W_\lambda\}$, define $G_n$ by sampling latent positions
$U_i\sim\mathrm{Unif}[0,1]$ i.i.d.\ and edges
$A_{ij}\mid U_i,U_j\sim\mathrm{Bernoulli}(W(U_i,U_j))$ independently.
Let $P_{0,\mathrm{unlab}}^{(n)}$ and $P_{1,\mathrm{unlab}}^{(n)}$ denote the
corresponding graph laws.

Equivalently, under $W_\lambda$ we may introduce latent labels
$Z_i\in\{\pm1\}$ i.i.d.\ with $\mathbb{P}(Z_i=1)=1/2$ and
\[
  \mathbb{P}(A_{ij}=1\mid Z)
  =
  \begin{cases}
    \dfrac{c+\lambda}{n}, & Z_i=Z_j,\\[0.5em]
    \dfrac{c-\lambda}{n}, & Z_i\neq Z_j.
  \end{cases}
\]

\begin{lemma}[Unlabelled SBM information exponents]
\label{lem:unlab-SBM-info}
Let $D_n^{\mathrm{unlab}}:={\rm KL}\bigl(P_{1,\mathrm{unlab}}^{(n)}\Vert
P_{0,\mathrm{unlab}}^{(n)}\bigr)$ and let $\mathcal{C}_n^{\mathrm{unlab}}$ be the Chernoff
information between $P_{0,\mathrm{unlab}}^{(n)}$ and
$P_{1,\mathrm{unlab}}^{(n)}$.  Then, in the sparse regime $p_n=c/n$,
\[
  \frac{D_n^{\mathrm{unlab}}}{n}\to I(\lambda),
  \qquad
  \frac{\mathcal{C}_n^{\mathrm{unlab}}}{n}\to J(\lambda),
\]
so that $D_n^{\mathrm{unlab}}=I(\lambda)\,n+o(n)$ and
$\mathcal{C}_n^{\mathrm{unlab}}=J(\lambda)\,n+o(n)$.  In particular, passing from
labelled to unlabelled SBMs (via a sparse graphon representation) does not
change the per--vertex information exponents.
\end{lemma}

Bayes factor model selection between $W_0$ and $W_\lambda$ reaches the
Chernoff rate:
\begin{theorem}[Robust Bayes factor for unlabelled SBM vs.\ ER]
\label{thm:WH-unlab-Bayes}
Consider the two--model Bayesian experiment $M\in\{0,1\}$ with prior
$\Pi(M=0)=\Pi(M=1)=1/2$ and likelihoods $P_{0,\mathrm{unlab}}^{(n)}$ (ER)
and $P_{1,\mathrm{unlab}}^{(n)}$ (SBM).  Let $\Pi_n(M\mid G_n)$ be the
posterior and let $\delta_n$ be the Bayes selector
$\delta_n(G_n)=\mathbbm{1}\{\Pi_n(M=1\mid G_n)\ge 1/2\}$, i.e.\ the
likelihood ratio (Bayes factor) test between $P_{0,\mathrm{unlab}}^{(n)}$
and $P_{1,\mathrm{unlab}}^{(n)}$.

For $m\in\{0,1\}$, write
\[
  R_{n,m}
  :=
  \mathbb{P}_{P_{m,\mathrm{unlab}}^{(n)}}\bigl(\delta_n(G_n)\neq m\bigr)
\]
for the (non--robust) misclassification probability under model $m$, and
note that $R_{n,0}$ and $R_{n,1}$ share the same exponential rate.  Then:
\begin{enumerate}
  \item \emph{Chernoff optimality.}
        The Bayes factor test is asymptotically Chernoff optimal:
        \[
          -\frac{1}{n}\log R_{n,m}\;\longrightarrow\;J(\lambda),
          \qquad m=0,1,
        \]
        where $J(\lambda)$ is the per--vertex Chernoff exponent from
        Lemma~\ref{lem:unlab-SBM-info}.
  \item \emph{Robust Bayes risk.}
        For any sequence $\mathcal{C}_n\downarrow 0$ with
        $\mathcal{C}_n=o\bigl(R_{n,m}^2\bigr)$ (equivalently
        $\sqrt{\mathcal{C}_n}=o(R_{n,m})$), the robust Bayes misclassification
        probability
        \[
          \mathcal{R}_n^{\mathrm{WH}}(m;\mathcal{C}_n)
          :=
          \mathbb{E}_{P_{m,\mathrm{unlab}}^{(n)}}\Bigl[
            \sup_{Q:\,{\rm KL}(Q\Vert\Pi_n(\cdot\mid G_n))\le \mathcal{C}_n}
            \mathbb{E}_Q\bigl[
              \mathbbm{1}\{\delta_n(G_n)\neq M\}
            \bigr]
          \Bigr]
        \]
        satisfies
        $\mathcal{R}_n^{\mathrm{WH}}(m;\mathcal{C}_n)=R_{n,m}\bigl(1+o(1)\bigr)$, and
        hence
        \[
          -\frac{1}{n}\log\mathcal{R}_n^{\mathrm{WH}}(m;\mathcal{C}_n)
          \;\longrightarrow\;J(\lambda),
          \qquad m=0,1.
        \]
\end{enumerate}
\end{theorem}

Thus, for decaying radii $\mathcal{C}_n$ that are small on the exponential scale, decision--theoretic
robustification does not change the information--theoretic detection rate.

\subsection{Nonparametric minimax lower bounds for model selection}
\label{subsec:WH-minimax}
\label{sec:WH-risk} 

We now place the ER vs.\ SBM testing problem inside a broad nonparametric
model class.  Let $\mathcal{P}_n$ be any collection of graph laws such that,
for all $n$ large enough,
\[
  P_{0,\mathrm{unlab}}^{(n)},\ P_{1,\mathrm{unlab}}^{(n)}
  \in \mathcal{P}_n.
\]

A (possibly randomized) selector $\delta_n$ maps graphs to $\{0,1\}$, and we
associate to each $P\in\mathcal{P}_n$ a label $M(P)\in\{0,1\}$, with
$M\bigl(P_{0,\mathrm{unlab}}^{(n)}\bigr)=0$ and
$M\bigl(P_{1,\mathrm{unlab}}^{(n)}\bigr)=1$.  Let
$\Pi_n(\cdot\mid G_n)$ be any (possibly data--dependent) posterior or
pseudo--posterior on $\{0,1\}$, and define the robust posterior
misclassification probability
\[
  e_n^{\mathrm{rob}}(C;G_n)
  :=
  \sup_{Q:\,{\rm KL}(Q\Vert\Pi_n(\cdot\mid G_n))\le C}
  \mathbb{E}_Q\bigl[\mathbbm{1}\{\delta_n(G_n)\neq M\}\bigr].
\]

The associated nonparametric minimax robust risk is
\[
  \mathfrak{R}_n^\star(C)
  :=
  \inf_{\delta_n,\Pi_n}
  \sup_{P\in\mathcal{P}_n}
  \mathbb{E}_P\bigl[e_n^{\mathrm{rob}}(C;G_n)\bigr].
\]
\begin{theorem}[Nonparametric minimax lower bound for sparse ER vs.\ SBM]
\label{thm:nonparam-WH-ER-SBM}
Let $\mathcal{P}_n$ be any model class containing
$P_{0,\mathrm{unlab}}^{(n)}$ and $P_{1,\mathrm{unlab}}^{(n)}$ for all $n$
large.  Let $\mathcal{C}_n\downarrow 0$ be any sequence.  Then
\begin{equation}
  \limsup_{n\to\infty}
  \frac{1}{n}\log\frac{1}{\mathfrak{R}_n^\star(\mathcal{C}_n)}
  \;\le\;
  J(\lambda)
\;=\;
\frac{\lambda^2}{16c}
+ O\!\left(\frac{\lambda^4}{c^3}\right)
\qquad\text{as }\lambda\to 0.
  \label{eq:WH-nonparam-minimax-bound}
\end{equation}

In particular, no estimator/posterior pair and no choice of radii can
achieve a better exponential error rate than $J(\lambda)$ uniformly over
$\mathcal{P}_n$.
\end{theorem}

If we further restrict to a graphon class $\mathcal{W}_n$ containing $W_0$ and
$W_\lambda$ and let $\mathcal{P}_n=\{P_W^{(n)}:W\in\mathcal{W}_n\}$, we obtain
a matching upper bound.
\begin{theorem}[Minimax characterization over a sparse graphon class]
\label{thm:WH-graphon-minimax}
Let $\mathcal{W}_n$ be any graphon class with $W_0,W_\lambda\in\mathcal{W}_n$
and let $P_W^{(n)}$ denote the law of the random graph generated from $W$.
For each $n$ and robustness radius $\mathcal{C}_n>0$, define the two--point graphon
decision--theoretic robustness minimax risk by
\[
  \mathfrak{R}_{n,\mathcal{W}_n}^{\mathrm{WH}}(\mathcal{C}_n)
  :=
  \inf_{\delta_n,\Pi_n}
  \max_{m\in\{0,1\}}
  \mathbb{E}_{P_{W_m}^{(n)}}\bigl[e_n^{\mathrm{rob}}(\mathcal{C}_n;G_n)\bigr],
\]
where $W_0$ and $W_\lambda$ correspond to the ER and two--block SBM graphons,
respectively, and $e_n^{\mathrm{rob}}(\mathcal{C}_n;G_n)$ is the robustified posterior
misclassification probability as in Section~\ref{sec:WH-risk}, with
$M\in\{0,1\}$ indicating the model.

Let $R_n$ be the (non--robust) Bayes misclassification probability of the
Bayes factor test between $W_0$ and $W_\lambda$ with equal prior probabilities
on $M\in\{0,1\}$, as in Theorem~\ref{thm:WH-unlab-Bayes}(i).  Suppose
$\mathcal{C}_n\downarrow 0$ satisfies $\mathcal{C}_n=o(R_n^2)$; in particular, it is sufficient to
assume
\[
  \mathcal{C}_n=o\bigl(\exp\{-2J(\lambda)n\}\bigr),
  \qquad\text{equivalently}\qquad
  \sqrt{\mathcal{C}_n}=o\bigl(\exp\{-J(\lambda)n\}\bigr),
\]
since $R_n=\exp\{-J(\lambda)n+o(n)\}$.  Then
\[
  \lim_{n\to\infty}
  \frac{1}{n}\log\frac{1}{\mathfrak{R}_{n,\mathcal{W}_n}^{\mathrm{WH}}(\mathcal{C}_n)}
  = J(\lambda).
\]
In particular, the Bayes factor test between $W_0$ and $W_\lambda$ is
decision--theoretic robustness minimax optimal at Chernoff exponent
$J(\lambda)$ for the two--point graphon testing problem
$W=W_0$ versus $W=W_\lambda$ embedded in the class $\mathcal{W}_n$.
\end{theorem}

These results show that robustification neither improves nor degrades the
optimal detection exponent for sparse ER vs.\ SBM, even in very large
nonparametric model classes.

\section{Robustness for percolation--based network robustness indices}
\label{sec:WH-percolation}

We now specialize Theorem~\ref{thm:WH-critical-robustness} to configuration models and
percolation--based robustness indices, and derive a nonparametric minimax
lower bound with critical exponent $4$ in the distance to the fragmentation
threshold.

\subsection{Configuration models and critical robustness indices}
\label{subsec:config-models}
\label{sec:config-models} 

We focus on configuration models with i.i.d.\ degrees and consider robustness
indices derived from component sizes.  Let
$G_n\sim\mathrm{CM}_n(\mu)$ be a configuration model on $n$ vertices with
i.i.d.\ degrees $D_i\sim\mu$ and finite third moment
$\mathbb{E}_\mu[D^3]<\infty$, and let $P_\mu^{(n)}$ denote the law of $G_n$.
Write
\[
  \theta(\mu)
  :=
  \frac{\mathbb{E}_\mu[D(D-1)]}{\mathbb{E}_\mu[D]}
\]
for the branching factor.  In the subcritical regime $\theta(\mu)<1$, the
cluster containing a uniformly chosen vertex has finite expectation; in the
supercritical regime $\theta(\mu)>1$, there is a giant component
\citep{VDHofstadBook2017}.

As a concrete robustness index we use the \emph{susceptibility}:
for $G_n\sim\mathrm{CM}_n(\mu)$ and a uniformly chosen vertex $V_n$,
\[
  S_n(\mu)
  :=
  \mathbb{E}_\mu\bigl[|C(V_n)|\bigr],
\]
where $C(V_n)$ is the component of~$V_n$.  For $\mu$ with
$\theta(\mu)<1$, a standard branching--process coupling yields
\[
  S_n(\mu)
  \xrightarrow[n\to\infty]{}
  R(\mu)
  := \frac{1}{1-\theta(\mu)}.
\]

\begin{proposition}[Critical behavior of susceptibility in configuration models]
\label{prop:ER-critical-R}
Let $\mu$ be a degree distribution for a configuration model and let
\[
  \theta(\mu)
  :=
  \frac{\mathbb{E}_\mu[D(D-1)]}{\mathbb{E}_\mu[D]}
\]
denote its branching factor. Define the susceptibility index by
\[
  R(\mu) := \frac{1}{1-\theta(\mu)},
  \qquad
  \Delta(\mu) := 1-\theta(\mu).
\]

Then, for any compact subset of $\{\mu : \theta(\mu)<1\}$, $R(\mu)$ is smooth
and admits the representation
\[
  R(\mu)
  = \frac{c_0}{\Delta(\mu)} + H(\mu),
  \qquad c_0 = 1,\ H(\mu)\equiv 0.
\]

Moreover, assume that the family of degree distributions $\mu$ is
parametrized by a finite--dimensional parameter $\vartheta$,
$\vartheta\mapsto\mu_\vartheta$, and let $\vartheta_\star$ be a critical
parameter such that
\[
  \theta(\mu_{\vartheta_\star}) = 1,
  \qquad
  \nabla_\vartheta \theta(\mu_{\vartheta_\star}) \neq 0
\]
(non--degenerate approach to criticality). Then, as
$\vartheta\to\vartheta_\star$ from the subcritical side
$\{\vartheta : \theta(\mu_\vartheta) < 1\}$, i.e.\ as
$\Delta(\mu_\vartheta)\downarrow 0$,
\[
  \bigl\|\nabla_\vartheta R(\mu_\vartheta)\bigr\|
  \,\asymp\, \Delta(\mu_\vartheta)^{-2}.
\]

In particular, in any such finite--dimensional parametrization $R$ satisfies
Assumption~\ref{ass:critical-R} with $\rho(\mu)=\theta(\mu)$ and
$\Delta(\mu)=1-\theta(\mu)$.
\end{proposition}

Thus, the susceptibility of configuration models provides a concrete example of
a percolation--based robustness index with the $1/\Delta$ divergence required
by Theorem~\ref{thm:WH-critical-robustness}.  More elaborate robustness
functionals, such as the area under the largest--component curve under node
removal \citep{artime2024robustness}, exhibit the
same leading $1/\Delta$ behavior and the same critical exponents.

\subsection{Critical exponent and minimax lower bound}
\label{subsec:WH-percolation-theory}

We now specialize Theorem~\ref{thm:WH-critical-robustness} to the susceptibility
$R(\mu)$ and derive a nonparametric minimax lower bound over general model
classes containing the configuration--model family.

Fix a sequence $\Delta_n\downarrow 0$ with $\Delta_n\gg n^{-1/2}$ and consider
degree distributions $\mu$ such that
$\Delta(\mu)=1-\theta(\mu)\in[\Delta_n,2\Delta_n]$.  Let
$\mathcal{P}_n$ be any model class containing the configuration models
$\mathrm{CM}_n(\mu)$ for all $\mu$ in this slice.  For an estimator
$a_n(G_n)$ of $R(\mu)$, define the robust risk at $\mu$,
\[
  \widetilde{\mathcal{R}}_n(a_n,\Pi_n;\mu;\mathcal{C}_n)
  :=
  \mathbb{E}_{P_\mu^{(n)}}\left[
    \sup_{\widetilde\Pi:\,{\rm KL}(\widetilde\Pi\Vert\Pi_n)\le \mathcal{C}_n}
    \int
      \bigl(a_n(G_n)-R(\mu')\bigr)^2\,
      \widetilde\Pi(\mathrm{d}\mu')
  \right],
\]
where $\Pi_n=\Pi_n(\cdot\mid G_n)$ is an arbitrary data--dependent posterior on $\mu$
and $\mathcal{C}_n\ge0$ is a radius.  The nonparametric minimax robust risk over the slice
$\{\mu:\Delta(\mu)\in[\Delta_n,2\Delta_n]\}$ is
\[
  \mathfrak{R}_n^{\mathrm{WH}}(\Delta_n,\mathcal{C}_n)
  :=
  \inf_{(a_n,\Pi_n)}
  \sup\Bigl\{
    \widetilde{\mathcal{R}}_n(a_n,\Pi_n;\mu;\mathcal{C}_n):
    P_\mu^{(n)}\in\mathcal{P}_n,\,
    \Delta(\mu)\in[\Delta_n,2\Delta_n]
  \Bigr\}.
\]
\begin{theorem}[Nonparametric minimax lower bound near the critical surface]
\label{thm:WH-nonparam-minimax-critical}
Assume that $\mathcal{P}_n$ contains the configuration--model family
$\{\mathrm{CM}_n(\mu):\mu\in(\underline{\mu},\overline{\mu})\}$ for some
$0<\underline{\mu}<\overline{\mu}<\infty$, and let
\[
  \theta(\mu)
  := \frac{\mathbb{E}_\mu[D(D-1)]}{\mathbb{E}_\mu[D]},
  \qquad
  \Delta(\mu) := 1-\theta(\mu),
\]
denote the branching factor and its deficit.  Suppose there exists a
critical point $\mu_\star\in(\underline{\mu},\overline{\mu})$ such that
$\Delta(\mu_\star)=0$ and
$\nabla_\mu\theta(\mu_\star)\neq 0$ (non--degenerate approach to
criticality).  Let $\Delta_n\downarrow 0$ satisfy $\Delta_n\gg n^{-1/2}$,
and assume that for all $n$ large enough the near--critical slice
$\{\mu:\Delta(\mu)\in[\Delta_n,2\Delta_n]\}$ is non--empty.

Define the classical nonparametric minimax risk over this slice by
\[
  \mathfrak{R}_n^{\mathrm{class}}(\Delta_n)
  :=
  \inf_{a_n}
  \sup\Bigl\{
    \mathbb{E}_{P_\mu^{(n)}}\bigl[
      \bigl(a_n(G_n)-R(\mu)\bigr)^2
    \bigr]:
    P_\mu^{(n)}\in\mathcal{P}_n,\,
    \Delta(\mu)\in[\Delta_n,2\Delta_n]
  \Bigr\},
\]
where $R(\mu)=(1-\theta(\mu))^{-1}$ is the susceptibility of the
configuration model in the subcritical regime $\Delta(\mu)>0$.  Then there
exists $c>0$, depending only on $(\underline{\mu},\overline{\mu})$ and the
local parametrization around $\mu_\star$, such that
\begin{equation}
  \mathfrak{R}_n^{\mathrm{class}}(\Delta_n)
  \;\ge\;
  \frac{c}{n\,\Delta_n^4}
  \qquad\text{for all $n$ large enough.}
  \label{eq:WH-minimax-critical-lb-main}
\end{equation}
Equivalently,
\[
  \liminf_{n\to\infty}
  n\,\Delta_n^4\,
  \mathfrak{R}_n^{\mathrm{class}}(\Delta_n)
  \;\ge\; c.
\]
In particular, any robustified minimax risk functional
$\mathfrak{R}_n^{\mathrm{WH}}(\Delta_n,\mathcal{C}_n)$ that pointwise dominates the
classical squared--error risk,
\[
  \widetilde{\mathcal{R}}_n(a_n,\Pi_n;\mu;\mathcal{C}_n)
  \;\ge\;
  \mathbb{E}_{P_\mu^{(n)}}\bigl[
    \bigl(a_n(G_n)-R(\mu)\bigr)^2
  \bigr]
  \quad\text{for all $(a_n,\Pi_n)$ and $\mu$,}
\]
necessarily satisfies the same lower bound:
$\mathfrak{R}_n^{\mathrm{WH}}(\Delta_n,\mathcal{C}_n)\ge
\mathfrak{R}_n^{\mathrm{class}}(\Delta_n)\gtrsim 1/(n\Delta_n^4)$.
\end{theorem}

Combining Theorem~\ref{thm:WH-nonparam-minimax-critical} with
Theorem~\ref{thm:WH-critical-robustness} (applied to $R(\mu)$ with a
posterior contracting at rate $r_n\asymp n^{-1/2}$) shows that:
\begin{itemize}
  \item the baseline posterior MSE for susceptibility scales like
        $1/(n\Delta_n^4)$ in the near--critical regime;
  \item the robust MSE has the same critical exponent $4$ in
        $\Delta_n^{-1}$ and inflates by a sharp factor of order
        $1+2\sqrt{\mathcal{C}_n}$;
  \item no robust procedure can improve on this scaling uniformly over
        large model classes containing the configuration model.
\end{itemize}

In particular, as the network approaches its percolation threshold, the
decision--level uncertainty about percolation--based robustness indices
inevitably explodes at least as fast as $\Delta_n^{-4}$, even if we allow
arbitrary nonparametric models and arbitrary robustification of the
posterior.

\section{Computation of decision--theoretic robust decisions for network models}
\label{sec:computation}

Our theory treats decision--theoretic robustness abstractly as an optimization over a
divergence ball $\mathcal{B}_\phi(\Pi_0;C)$ around a baseline posterior $\Pi_0$.
In applications, $\Pi_0$ is only available through approximate posterior
draws for a network model (SBM, graphon, random dot product graph,
configuration model), and we must approximate the least--favorable
perturbation.  This section explains how to implement this in two modular
steps:
\begin{enumerate}
  \item \textbf{Baseline inference.} Fit a network model and obtain an
        approximate posterior or pseudo--posterior
        $\Pi_0(\mathrm{d}\theta\mid G_n)$ using variational inference or
        spectral / moment methods wrapped in a pseudo--Bayesian layer.
  \item \textbf{Robustification.} Given posterior samples
        $\{(\theta^s,w_s)\}_{s=1}^S$ and a loss $L(a,\theta)$, compute the
        worst--case posterior risk over a divergence ball
        $\mathcal{B}_\phi(\Pi_0;C)$ by entropic tilting of the weights and, for
        general $\phi$--balls, mirror descent in weight space.
\end{enumerate}

This decoupling means that existing scalable inference pipelines for SBMs,
graphons and latent position models can be used unchanged: the decision--theoretic
robust layer is applied \emph{on top of} whatever approximate posterior
samples they produce.\\

Full computational details, including explicit pseudocode for the KL--ball
entropic tilting procedure and the mirror--descent adversary for general
$\phi$--divergence balls, are collected in the Supplementary Material.

\subsection{Baseline approximate posteriors for network models}
\label{subsec:baseline-posteriors}

For each network model in Section~\ref{subsec:bayes-networks} we assume the
availability of approximate posterior draws
\[
  \{(\theta^s,w_s):s=1,\dots,S\},
  \qquad
  \sum_{s=1}^S w_s=1,
\]
from a baseline posterior or pseudo--posterior $\Pi_0$ on $\Theta$.  Typical
choices are:
\begin{itemize}
  \item \textbf{Variational posteriors} for SBMs and random dot product
        graphs, where $q_\lambda(\theta)$ is a mean--field or structured
        variational approximation fitted by maximizing an ELBO.  We sample
        $\theta^s\sim q_\lambda$ and set $w_s=1/S$.
  \item \textbf{Spectral / moment pseudo--posteriors}, in which a point
        estimator $\hat\theta(G_n)$ (e.g.\ spectral embedding, degree
        moments, Hill tail index) is endowed with an approximate Gaussian
        sampling distribution derived from random matrix theory or a
        parametric bootstrap.  We then treat this Gaussian as a baseline
        pseudo--posterior and sample from it.
\end{itemize}
The decision--theoretic robustification step treats $\{(\theta^s,w_s)\}$ as an
empirical approximation to $\Pi_0$, irrespective of how the draws were
obtained.

\subsection{KL--ball optimization by entropic tilting}
\label{subsec:tilting}

Fix an action $a$ and loss $L(a,\theta)$.  Let
$L_s := L(a,\theta^s) \in \mathbb{R}$ be the loss evaluated at posterior
draw $\theta^s$, and let
\[
  \bm w = (w_1,\dots,w_S),
  \qquad
  \bm q = (q_1,\dots,q_S)
\]
denote the baseline and perturbed posterior weights.  For a Kullback--Leibler
ball of radius $C>0$ around $\Pi_0$ we work with the discrete approximation
\[
  \mathcal{U}_C(\bm w)
  :=
  \Bigl\{
    \bm q:\ q_s\ge 0,\ \sum_s q_s=1,\ 
    {\rm KL}(\bm q\Vert\bm w)\le C
  \Bigr\},
  \qquad
  {\rm KL}(\bm q\Vert\bm w)
  :=
  \sum_{s=1}^S q_s\log\frac{q_s}{w_s}.
\]
The decision--theoretic robust posterior risk for $a$ is then
\[
  \rho_{\mathrm{rob}}(a;C)
  \approx
  \sup_{\bm q\in\mathcal{U}_C(\bm w)}
  \sum_{s=1}^S q_s L_s.
\]

For KL balls this finite--dimensional problem admits a one--dimensional dual
via the Donsker--Varadhan variational formula.  Define
\begin{equation}
  \label{eq:psi-lambda}
  \psi(\lambda)
  :=
  \frac{C + \log\sum_{s=1}^S w_s \exp\{\lambda L_s\}}{\lambda},
  \qquad \lambda>0.
\end{equation}
Then
\begin{equation}
  \label{eq:tilting-dual}
  \sup_{\bm q\in\mathcal{U}_C(\bm w)} \sum_{s=1}^S q_s L_s
  =
  \inf_{\lambda>0}\psi(\lambda),
\end{equation}
and the least--favorable weights have the \emph{entropic tilting} form
\begin{equation}
  \label{eq:q-star-tilt}
  q_s^{\star}(\lambda)
  \propto
  w_s\,\exp\{\lambda L_s\},
  \qquad
  \lambda=\arg\min_{\lambda>0}\psi(\lambda).
\end{equation}

\paragraph{Implementation with MCMC draws.}
In practice $\Pi_0$ is represented by posterior or pseudo--posterior draws
$\{(\theta^s,w_s)\}_{s=1}^S$ obtained by MCMC or variational inference.
For a fixed action $a$ we evaluate $L_s=L(a,\theta^s)$ and solve the
one--dimensional dual problem
$\lambda^\star=\arg\min_{\lambda>0}\psi(\lambda)$.
The least--favorable posterior on this discrete support has tilted weights
\[
  q_s^\star \;\propto\; w_s \exp\{\lambda^\star L_s\},
  \qquad s=1,\dots,S,
\]
and robust posterior expectations are approximated by
$\sum_s q_s^\star f(\theta^s)$ for any functional $f$.
The full numerical scheme, including a simple bisection search for
$\lambda^\star$, is given in Algorithm~\ref{alg:si-tilting} in the
Supplementary Material. 

In practice we proceed as follows:
\begin{enumerate}
  \item Evaluate $L_s=L(a,\theta^s)$ on posterior or pseudo--posterior draws.
  \item Compute $\psi(\lambda)$ and its derivative using a log--sum--exp
        stabilization.
  \item Minimize $\psi(\lambda)$ over $\lambda>0$ by a simple one--dimensional
        method (Newton, bisection or grid search).
  \item Form the tilted weights $q_s^\star(\lambda)$ in
        \eqref{eq:q-star-tilt} and approximate robust expectations under the
        least--favorable posterior by
        $\sum_s q_s^\star(\lambda) f(\theta^s)$ for any functional $f$ of
        interest.
\end{enumerate}
This gives a fast adversarial algorithm for KL--ball robustification that
requires no additional model--specific derivations beyond being able to
evaluate $L(a,\theta^s)$.

\paragraph{Mirror--descent adversary}
For general $\phi$--divergence balls
$\mathcal{B}_\phi(\Pi_0;C)$ we work in the discrete weight space
of posterior draws and run a mirror--descent adversary.  Writing
$\bm w=(w_1,\dots,w_S)$ for the baseline weights and
$L_s=L(a,\theta^s)$, we maintain log--tilts
$u_s^{(t)}=\log(q_s^{(t)}/w_s)$ and update
\[
  u_s^{(t+1)} \;=\; u_s^{(t)} + \eta\,\bigl\{L_s - \bar L^{(t)}\bigr\},
  \qquad
  \bar L^{(t)} := \sum_{r=1}^S q_r^{(t)} L_r,
\]
followed by a projection of $\bm q^{(t+1)}$ back onto the
$\phi$--ball $\{D_\phi(\bm q\Vert\bm w)\le C\}$.
For Kullback--Leibler balls this projection has a closed form; for
general $\phi$ it reduces to a small convex program on the simplex.
A complete pseudocode implementation is given in
Algorithm~\ref{alg:mirror-adversary} in the Supplementary Material.

\subsection{Practical calibration of the divergence radius}
\label{subsec:calibration}

The divergence radius $C$ encodes how far we are willing to move from the
working posterior $\Pi_0$.  Our local theory for network functionals shows that,
under squared loss, the robust posterior risk admits a sharp expansion
\[
  \rho_{\mathrm{rob}}(C)
  =
  \rho_0
  + 2\,\rho_0\sqrt{C}
  + o\bigl(\rho_0\sqrt{C}\bigr),
\]
where $\rho_0$ is the baseline posterior risk (equivalently, the posterior
variance of the loss).  In practice we use three complementary calibration
strategies:
\begin{itemize}
  \item \textbf{Decision--theoretic sensitivity paths.}  For a given decision
        or network functional (spectral gap, giant component size, Hill
        index, epidemic threshold) we compute $\rho_{\mathrm{rob}}(C)$ over a
        grid of radii and plot $\rho_{\mathrm{rob}}(C)$ against $\sqrt{C}$.
        The theory predicts an initial linear regime with slope
        $\approx 2\rho_0$; visible kinks in this curve signal that the
        least--favorable posterior has moved into a qualitatively different
        region of parameter space (for instance, across a percolation or
        detection threshold).
  \item \textbf{Scaling with network size.}  Because KL divergence adds over
        edges or vertices, a fixed $C$ represents a smaller perturbation per
        edge as $n$ grows.  For sparse models with $O(n)$ effective
        observations (e.g.\ sparse ER/SBM with $p_n\sim c/n$) it is natural
        to work with a \emph{per--vertex} budget $\mathcal{C}_n=c_\star/n$, so that the
        total KL perturbation scales like a constant.  In dense regimes a
        per--edge budget $\mathcal{C}_n=c_{\star\star}/n^2$ may be more appropriate.
        Our nonparametric minimax results for sparse ER vs.\ SBMs and for
        configuration models can be read as giving problem--specific guidance
        on how large $\mathcal{C}_n$ can be before robustness inflation dominates the
        baseline risk.
  \item \textbf{Application--specific tolerances.}  One can back--solve for
        $C$ from a tolerable inflation in risk.  For example, requiring
        $\rho_{\mathrm{rob}}(C)\le(1+\delta)\rho_0$ suggests the heuristic
        constraint $\sqrt{C}\lesssim\delta/2$.  In epidemic or percolation
        applications it may be more natural to constrain how far the
        least--favorable graph can move key quantities such as the
        effective reproduction number or branching factor, and translate
        that into a radius using Lipschitz bounds from the earlier theory.
\end{itemize}

In our experiments we report decision--theoretic robustness sensitivity curves
and adopt a default $C$ corresponding to roughly $10$--$20\%$ inflation in
posterior risk, unless domain knowledge suggests a stricter or looser
tolerance.


\section{Experiments}
\label{sec:experiments}

In this section we illustrate the proposed decision--theoretic robustness framework
for network models on two real datasets. In both cases we specify simple working
models (Erd\H{o}s--Rényi and stochastic block models), define low--dimensional
network functionals of scientific interest, and study the local robustness of the
corresponding Bayes decisions in the sense of Watson and Holmes. The first example
uses a population of functional brain connectivity networks; the second uses the
Wave~1 social networks from villages in Karnataka, India.

\subsection{Synthetic experiment validation 1: ER vs.\ SBM and configuration--model percolation}
\label{subsec:synthetic-validation}

Although our main focus is on real--data applications (Sections~\ref{sec:brain-experiment}--\ref{sec:karnataka-experiment}), it is useful to include a small synthetic study that numerically checks two key theoretical predictions of our framework: (i) the small--radius $\sqrt{C}$ expansion of the robust posterior risk, and (ii) the critical exponent~4 for fragmentation--type functionals. The experiments below are deliberately minimal and can be reported in the Supplementary Material.

\paragraph{Synthetic experiment A: ER vs.\ SBM near the detection threshold.}

We first consider the sparse Erd\H{o}s--R\'enyi vs.\ two--block SBM testing problem of Section~\ref{sec:er-sbm-theory}.  For fixed $c > 0$ and a signal parameter $\lambda \in (0,c)$ we take
\[
  p_n \;=\; \frac{c}{n}, 
  \qquad p^{\text{in}}_n \;=\; \frac{c+\lambda}{n},
  \qquad p^{\text{out}}_n \;=\; \frac{c-\lambda}{n},
\]
and simulate graphs $G_n$ under $H_0$ (sparse ER with edge probability $p_n$) and $H_1$ (balanced two--block SBM with known labels and within/between probabilities $p^{\text{in}}_n, p^{\text{out}}_n$) as in Section~\ref{sec:er-sbm-theory}.  For each simulated $G_n$ we compute the exact two--point posterior $\Pi_n(M\mid G_n)$ on $M \in \{0,1\}$ under equal priors on $H_0,H_1$, and consider the 0--1 loss $L(a,M) = \mathbbm{1}\{a \neq M\}$ for the model--selection decision $a(G_n)\in\{0,1\}$.  The baseline posterior misclassification probability is
\[
  e_{0,n}(G_n)
  \;=\; \min\bigl\{\Pi_n(M=0\mid G_n),\,\Pi_n(M=1\mid G_n)\bigr\},
\]
and for a grid of small radii $C>0$ we compute the Watson--Holmes robust posterior misclassification probability
\[
  e_{\mathrm{rob},n}(C;G_n)
  \;=\; \sup_{Q:\,\mathrm{KL}(Q\|\Pi_n)\le C}
         \mathbb{E}_Q\bigl[\mathbbm{1}\{a(G_n)\neq M\}\bigr]
\]
via entropic tilting of the two posterior weights, as in Section~\ref{sec:computation}.  Averaging $e_{0,n}(G_n)$ and $e_{\mathrm{rob},n}(C;G_n)$ over Monte Carlo replicates yields empirical baseline and robust risks
$R_{0,n} = \mathbb{E}[e_{0,n}(G_n)]$ and
$R_{\mathrm{rob},n}(C) = \mathbb{E}[e_{\mathrm{rob},n}(C;G_n)]$.\\

The generic small--radius expansion of Theorem~S.2 gives, for bounded losses,
\[
  R_{\mathrm{rob},n}(C)
  \;\approx\;
  R_{0,n} + \sqrt{2\,\mathrm{Var}_{\Pi_n}\!\bigl(L(a(G_n),M)\bigr)}\,\sqrt{C},
  \qquad C\downarrow 0.
\]
In the two--point test, this variance equals $R_{0,n}(1-R_{0,n})$, so the leading $\sqrt{C}$ coefficient is proportional to $\sqrt{R_{0,n}(1-R_{0,n})}$.  To visualize this, we normalize the robust excess misclassification probability as
\[
  C \;\longmapsto\;
  \frac{R_{\mathrm{rob},n}(C) - R_{0,n}}
       {\sqrt{2\,R_{0,n}(1-R_{0,n})}\,\sqrt{C}},
\]
which should be approximately flat for small~$C$, with level given by a finite constant depending on $(n,\lambda)$.

\paragraph{Synthetic experiment B: percolation and exponent~4 in configuration models.}

To illustrate the critical exponent~4 for fragmentation--type indices established in Theorem~\ref{thm:critical-exponent}, we consider configuration models $\mathrm{CM}_n(\mu)$ as in Section~\ref{sec:config-models}, with degree distributions $\mu$ chosen so that the branching factor
\[
  \theta(\mu) \;=\; \frac{\mathbb{E}_\mu[D(D-1)]}{\mathbb{E}_\mu[D]}
\]
satisfies $0 < \Delta(\mu) := 1 - \theta(\mu) \ll 1$.  For concreteness, we take a Poisson family $(\mu_\Delta)$ with
$D_i\sim\text{Poisson}(1-\Delta)$, so that $\theta(\mu_\Delta)=1-\Delta$ and $\Delta(\mu_\Delta)=\Delta$.
For each $(n,\Delta)$, we simulate graphs $G_n\sim\mathrm{CM}_n(\mu_\Delta)$ and place a Gamma$(1,1)$ prior on the Poisson mean; conditioning on $G_n$ and truncating to the subcritical region $\{\lambda<1\}$ yields a conjugate pseudo--posterior $\Pi_n(\mu\mid G_n)$, which we approximate by Monte Carlo draws.  We focus on the susceptibility functional
\[
  R(\mu) \;=\; \frac{1}{1 - \theta(\mu)},
\]
which diverges like $1/\Delta(\mu)$ near the fragmentation threshold and satisfies Assumption~3.1.  Under squared loss $L(a,\mu) = (a - R(\mu))^2$, we compute the Bayes estimator $a_n^\star = \mathbb{E}_{\Pi_n}[R(\mu)]$, the baseline posterior risk
$\rho_{0,n} = \mathbb{E}_{\Pi_n}[(a_n^\star - R(\mu))^2]$, and the corresponding Watson--Holmes robust risk
\[
  \rho_{\mathrm{rob},n}(C)
  \;=\; \sup_{\tilde\Pi:\,\mathrm{KL}(\tilde\Pi\|\Pi_n)\le C}
         \mathbb{E}_{\tilde\Pi}\bigl[(a_n^\star - R(\mu))^2\bigr].
\]
Theorem~\ref{thm:critical-exponent} predicts that, in the near--critical regime and for small $C$,
\[
  \rho_{0,n} \;\asymp\; \frac{1}{n\,\Delta(\mu)^4},
  \qquad
  \rho_{\mathrm{rob},n}(C)-\rho_{0,n}
  \;\asymp\; \frac{\sqrt{C}}{n\,\Delta(\mu)^4}.
\]

\paragraph{Numerical implementation and results.}

In experiment~A, we fix $c=3$ and consider a sparse regime $p_n = c/n$ with a weak community signal $\lambda = 0.4$.  We simulate graphs of size $n=400$ under the two--point experiment $H_0$ (sparse ER) versus $H_1$ (balanced two--block SBM with known labels and within/between probabilities $p^{\text{in}}_n = (c+\lambda)/n$ and $p^{\text{out}}_n = (c-\lambda)/n$).  We use $N_{\mathrm{rep}}=1000$ Monte Carlo replicates and a logarithmic grid of radii $C\in[10^{-4},10^{-2}]$.  For each replicate, we compute the exact posterior on $\{H_0,H_1\}$, the Bayes rule under 0--1 loss, and the baseline and robust misclassification probabilities $e_{0,n}(G_n)$ and $e_{\mathrm{rob},n}(C;G_n)$ by entropic tilting of the two posterior weights.  Averaging across replicates yields $R_{0,n}$ and $R_{\mathrm{rob},n}(C)$. By the normalization process described above, an empirical robustness sensitivity curve
\[
  C \;\longmapsto\;
  \frac{R_{\mathrm{rob},n}(C) - R_{0,n}}
       {\sqrt{2\,R_{0,n}(1-R_{0,n})}\,\sqrt{C}}.
\]
Panel~(a) of Figure~\ref{fig:synthetic-small-radius} displays this quantity as a function of $\sqrt{C}$.
Over the range $C\in[10^{-4},10^{-2}]$ the curve is essentially flat, taking values between roughly $3.7$ and $4.2$.
This confirms the predicted $\sqrt{C}$ scaling of the robust misclassification risk with the nearly constant level, providing a finite--sample estimate of the leading $\sqrt{C}$ coefficient for this sparse ER vs.\ SBM testing problem.\\

In experiment~B, we fix $n=5000$ and consider Poisson configuration models with means $1-\Delta$ for $\Delta \in\{0.40,0.30,0.25,0.20,0.17,0.15\}$.  For each $\Delta$, we simulate $N_{\mathrm{rep}}=200$ graphs, approximate the truncated Gamma posterior for the mean by $S_{\text{post}}=2000$ draws and compute
$\rho_{0,n}$ and $\rho_{\mathrm{rob},n}(C)$ over the same grid $C\in[10^{-4},10^{-1}]$.  Panel~(b) of Figure~\ref{fig:synthetic-small-radius} plots the normalized robustness sensitivity curve
\[
  C \;\longmapsto\;
  \frac{\rho_{\mathrm{rob},n}(C)-\rho_{0,n}}{\rho_{0,n}\sqrt{C}}
\]
for $(n,\Delta)=(5000,0.2)$.  For small radii, the curve is again close to flat, taking values between about $2.2$ and $2.9$ over the range of $C$, in good agreement with the theoretical coefficient~2 in the squared--loss expansion of Theorem~3.3 and illustrating the predicted $\sqrt{C}$ inflation of the susceptibility risk at fixed distance to criticality.\\

To probe the exponent~4 in $\Delta^{-1}$, we fix a small radius $C\approx 10^{-3}$ and, for each $\Delta$, consider the log--log plots
\[
  -\log \Delta \;\longmapsto\; \log\bigl(n\,\rho_{0,n}\bigr),
  \qquad
  -\log \Delta \;\longmapsto\;
  \log\Bigl(\tfrac{n\,\bigl(\rho_{\mathrm{rob},n}(C)-\rho_{0,n}\bigr)}
                 {\sqrt{C}}\Bigr).
\]
Figure~\ref{fig:config-slopes} shows the resulting regression lines and the corresponding least--squares slopes are
\[
  \widehat{\kappa}_{\mathrm{base}} \approx 4.49,
  \qquad
  \widehat{\kappa}_{\mathrm{rob}} \approx 4.65,
\]
and are also reported in Table~\ref{tab:config-slopes}.  Both slopes are very close to the theoretical exponent~4, predicted by Theorem~\ref{thm:critical-exponent}. The baseline and robust exponents are numerically indistinguishable at the level of Monte Carlo error.  Together with the small--radius sensitivity curves, these synthetic experiments provide a controlled numerical check of (i) the $\sqrt{C}$ expansion of decision--theoretic robust risks and (ii) the exponent--4 blow--up of percolation--type functionals near the fragmentation threshold.

\begin{figure}[H]
  \centering
  \includegraphics[width=\textwidth]{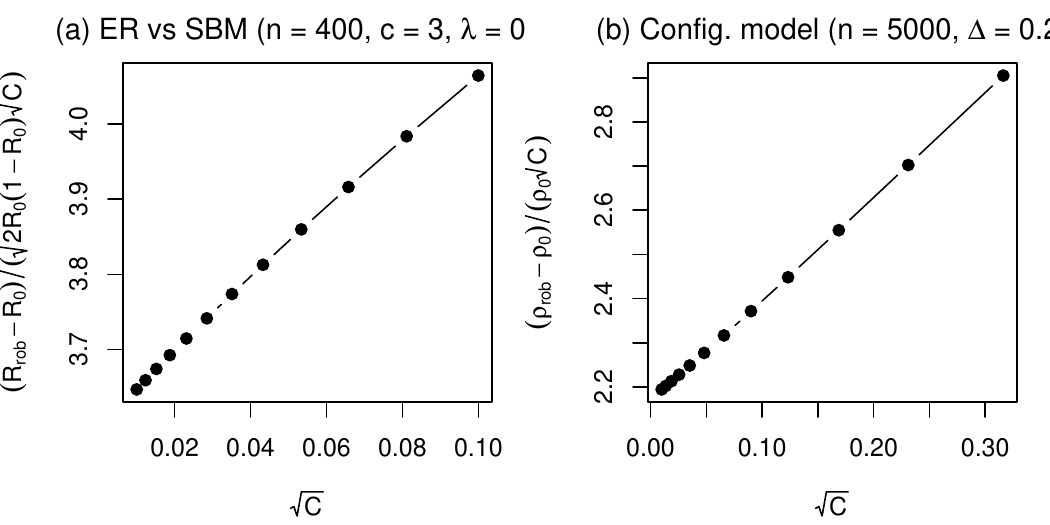}
  \caption{Synthetic small--radius robustness sensitivity curves.
    (a) ER vs.\ two--block SBM test with $n=400$, $c=3$, $\lambda=0.4$:
    normalized excess robust misclassification probability
    $(R_{\mathrm{rob},n}(C)-R_{0,n})/
     (\sqrt{2 R_{0,n}(1-R_{0,n})}\sqrt{C})$
    versus $\sqrt{C}$.
    (b) Configuration model with Poisson degrees of mean $1-\Delta$
    ($n=5000$, $\Delta=0.2$): normalized excess robust susceptibility
    $(\rho_{\mathrm{rob},n}(C)-\rho_{0,n})/(\rho_{0,n}\sqrt{C})$ versus
    $\sqrt{C}$.  Dashed horizontal lines mark reference levels corresponding to the leading $\sqrt{C}$ coefficients
predicted by Theorem~S.2 (panel~a) and Theorem~3.3 (panel~b).}
  \label{fig:synthetic-small-radius}
\end{figure}

\begin{figure}[H]
  \centering
  \includegraphics[width=\textwidth]{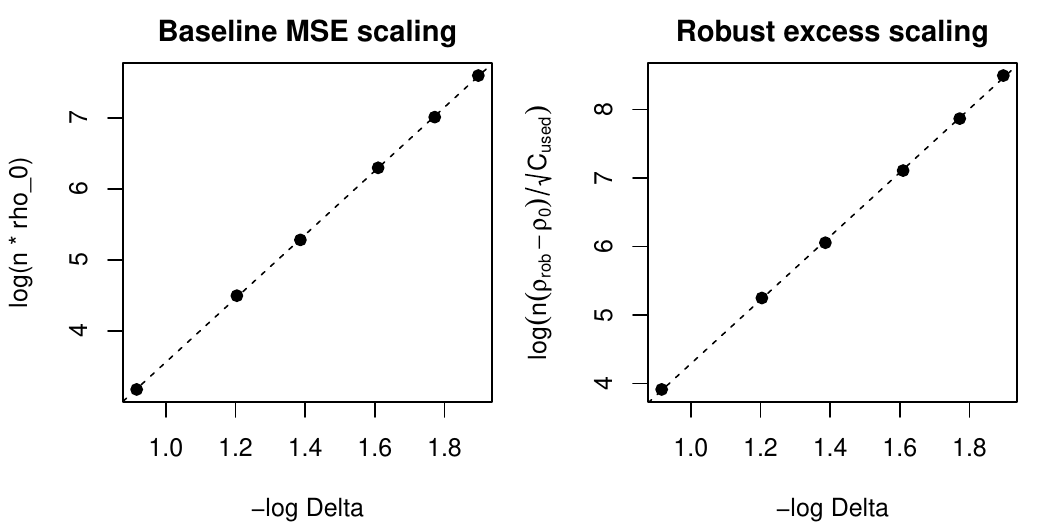}
  \caption{Configuration--model susceptibility: log--log scaling in the
    distance to criticality.  Left: $-\log\Delta \mapsto \log(n\,\rho_{0,n})$
    with fitted least--squares line (dashed); right:
    $-\log\Delta \mapsto
      \log\bigl(n(\rho_{\mathrm{rob},n}(C)-\rho_{0,n})/\sqrt{C}\bigr)$
    at $C\approx 10^{-3}$.  The regression slopes
    $\widehat{\kappa}_{\mathrm{base}} \approx 4.49$ and
    $\widehat{\kappa}_{\mathrm{rob}} \approx 4.65$ are close to the
    theoretical exponent~4 from Theorem~\ref{thm:critical-exponent}.}
  \label{fig:config-slopes}
\end{figure}

\begin{table}[H]
  \centering
  \begin{tabular}{lcc}
    \toprule
    & baseline exponent & robust exponent \\
    \midrule
    configuration model susceptibility ($n=5000$) & $4.49$ & $4.65$ \\
    \bottomrule
  \end{tabular}
  \caption{Estimated slopes of
    $-\log\Delta \mapsto \log(n\,\rho_{0,n})$ and
    $-\log\Delta \mapsto
      \log\bigl(n(\rho_{\mathrm{rob},n}(C)-\rho_{0,n})/\sqrt{C}\bigr)$
    in the configuration--model experiment for
    $\Delta\in\{0.40,0.30,0.25,0.20,0.17,0.15\}$ and
    $C\approx 10^{-3}$.  Theory predicts a common exponent~$4$; the empirical
    estimates $\widehat{\kappa}_{\mathrm{base}}$ and
    $\widehat{\kappa}_{\mathrm{rob}}$ are close to this value.}
  \label{tab:config-slopes}
\end{table}

\paragraph{Synthetic experiment C: Misspecification stress test (DCSBM truth, SBM working model).}
We generated networks from a degree-corrected stochastic block model (DCSBM) with mild community
structure but strong degree heterogeneity, then fit a misspecified plain SBM. We focus on the
leading eigenvalue $\lambda_1$ (and the associated epidemic-threshold proxy), and apply KL-ball
robustification via exponential tilting of posterior draws.\\

To calibrate the threshold decision
$\mathsf{Intervene}\{\mathbb{P}(\lambda_1>\tau) > p_\star\}$ in a nontrivial regime, we set $\tau$
using a pilot procedure that caps $\tau$ to lie within the working posterior support:
$\tau = \min\{q_{0.60}(\lambda_{1,\mathrm{true}}),\ q_{0.98}(\lambda_{1,\mathrm{work}})\}$.
In the pilot, the true exceedance probability was large while the working posterior essentially
ruled it out, indicating severe misspecification.
Because KL-tilting only reweights working-model draws, it cannot create support where the baseline
posterior assigns (near) zero mass; the pilot cap mitigates this degeneracy.

\begin{table}[H]
\centering
\caption{Pilot calibration illustrates misspecification for the event $\{\lambda_1>\tau\}$.}
\label{tab:pilot-calibration}
\begin{tabular}{lcc}
\toprule
Quantity & Value & Interpretation \\
\midrule
$\tau_{\text{truth}} = q_{0.60}(\lambda_{1,\text{true}})$ & 10.274 & target nontrivial truth regime \\
$\tau_{\text{cap}} = q_{0.98}(\lambda_{1,\text{work}})$  & 10.041 & cap to working posterior support \\
Chosen $\tau=\min(\tau_{\text{truth}},\tau_{\text{cap}})$ & 10.041 & threshold used in decision \\
$P_{\text{true}}(\lambda_1>\tau)$ (pilot) & 0.67 & frequent exceedance under DCSBM \\
$P_{\text{work}}(\lambda_1>\tau)$ (pilot) & 0.02 & working SBM nearly rules out exceedance \\
$p_\star=\text{cost}_{\text{int}}/\text{cost}_{\text{out}}$ & 0.20 & decision probability cutoff \\
\bottomrule
\end{tabular}
\end{table}

Figure~\ref{fig:misspec-main} shows that increasing the KL radius $C$ moves the robustified posterior
toward higher values of $\lambda_1$ and substantially reduces mean-squared error (MSE) relative to the
DCSBM truth, consistent with the SBM posterior being biased downward due to ignored degree heterogeneity.
For the threshold policy, robustification reduces regret at moderate radii by hedging against false
negatives induced by the misspecified working posterior.

\begin{figure}[H]
\centering
\begin{tabular}{@{}cc@{}}
  \includegraphics[width=0.48\textwidth]{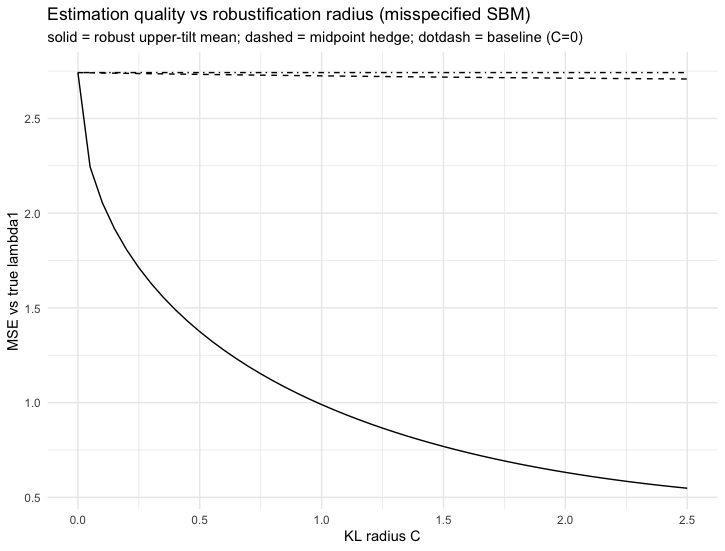} &
  \includegraphics[width=0.48\textwidth]{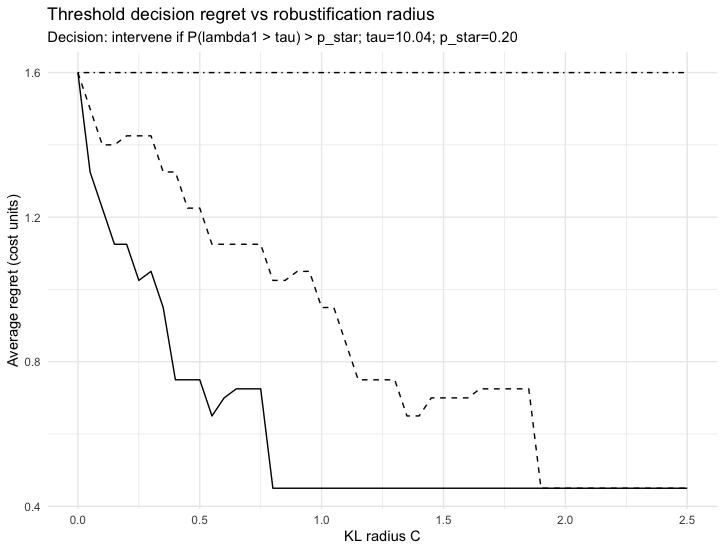}
\end{tabular}
\caption{\textbf{Misspecification stress test (DCSBM truth, SBM working model).}
Robustification is performed by exponential tilting within a KL ball of radius $C$ around the working posterior draws.
\emph{Left:} Estimation performance (MSE) vs.\ KL radius $C$ for $\lambda_1$.
\emph{Right:} Threshold-decision regret vs.\ KL radius $C$; robustification lowers regret by hedging against
false negatives under SBM misspecification.}
\label{fig:misspec-main}
\end{figure}

\paragraph{Synthetic experiment D: Radius paths and error exponents (ER vs.\ SBM).}
We provide additional numerical evidence for the information--theoretic message of Section~4 in the
labeled sparse Erd\H{o}s--R\'enyi versus two-block SBM test: the robustified risk obeys a \emph{local}
small-radius expansion only under genuinely local KL radii, and the \emph{radius path} \(\mathcal{C}_n\) governs
whether robustness preserves or destroys exponential error decay.\\

For each \(n\) and signal strength \(\lambda\), we simulate a labeled sparse graph under \(H_0\) (ER)
and \(H_1\) (balanced two-block SBM with \(p_{\mathrm{in}}=(c+\lambda)/n\) and \(p_{\mathrm{out}}=(c-\lambda)/n\)),
use the Bayes decision rule under equal priors, and record the posterior misclassification probability
\(e_0(G_n)\).
We then robustify the posterior in a KL ball of radius \(C\) and compute the corresponding worst-case
misclassification probability \(e_{\mathrm{rob}}(C;G_n)\) (two-point case; solved exactly via the
Bernoulli KL constraint).\\

\emph{Local small-radius regime (Panel~a).}
Panel~(a) reports the replicate-normalized quantity
\[
\mathbb{E}\!\left[\,
\frac{e_{\mathrm{rob}}(C;G_n)-e_0(G_n)}
{\sqrt{2\,e_0(G_n)\bigl(1-e_0(G_n)\bigr)}\,\sqrt{C}}
\,\right],
\]
which equals \(1+o(1)\) as \(C\downarrow 0\) for each fixed \((n,\lambda)\) under the local
\(\sqrt{C}\)-expansion.
For moderate signal (\(\lambda=0.2\)), the normalization stays close to \(1\) on the smallest radii shown:
for \(n=200\), it ranges from \(1.0006\) to \(1.0087\) over \(\sqrt{C}\in[0.002,0.03]\); for \(n=800\), it ranges
from \(1.0035\) to \(1.0488\); and for \(n=1600\) it remains below \(1.28\) on the displayed grid.
In contrast, for stronger signal and larger \(n\) the same \(\sqrt{C}\)-grid becomes \emph{nonlocal} relative
to the posterior: for \(\lambda=0.4\) the ratio increases from roughly \(1.00\)–\(1.05\) at \(n=200\), to
\(1.02\)–\(1.29\) at \(n=400\), to \(1.63\)–\(6.95\) at \(n=800\), and then explodes to
\(1.2\times 10^3\)–\(1.4\times 10^4\) at \(n=1600\). This is the rare-error regime in which \(e_0(G_n)\) is
already extremely small, so “small radius’’ must shrink with \(n\) (and signal) for the local expansion
to apply at the decision level.\\

\emph{Error exponents and radius paths (Panels~B--C).}
Panel~(B) plots the empirical exponent estimate \(-\log(R)/n\), where \(R=\mathbb{E}[e(G_n)]\), for the
baseline Bayes risk and for an exponentially shrinking radius path \(\mathcal{C}_n=\exp(-2\alpha n)\).
Quantitatively, for \(\lambda=0.2\) the baseline and exponential-radius robust exponents become essentially
indistinguishable at large \(n\): at \(n=6400\) they are \(0.0012021\) (baseline) versus \(0.0012008\) (robust),
and at \(n=12800\) they are \(0.0014373523\) (baseline) versus \(0.0014372875\) (robust).
For \(\lambda=0.1\), the corresponding large-\(n\) values are \(3.568\times 10^{-4}\) (baseline) and
\(3.430\times 10^{-4}\) (robust) at \(n=12800\), i.e.\ a small relative gap of about \(4\%\).
For \(\lambda=0.3\), baseline and robust agree closely through \(n=6400\) (both \(\approx 2.654\times 10^{-3}\)),
while at the largest \(n\) the baseline exponent estimate becomes noticeably more variable (e.g.\
\(0.00448\) at \(n=12800\)), consistent with finite-sample/Monte-Carlo instability once risks are extremely
small; the robust exponent remains stable around \(2.69\times 10^{-3}\).
Across \(\lambda\), the robust exponent does not exceed the baseline exponent, in line with the fact that
robustification cannot improve the information exponent predicted in Section~4.\\

Panel~(C) contrasts several radius scalings at \(\lambda=0.2\). Exponentially small radii preserve exponential
decay: at \(n=12800\), the baseline has \(R\approx 5.40\times 10^{-9}\) and rate \(0.0014873\), while
\(\mathcal{C}_n=\exp(-2\alpha n)\) yields \(R\approx 5.41\times 10^{-9}\) and rate \(0.0014872\) (agreement at four
significant digits in the exponent).
By contrast, polynomial/constant radii yield subexponential behavior and a collapsing per-node exponent:
at \(n=12800\), the constant-radius regime has \(R\approx 1.93\times 10^{-3}\) and rate \(4.88\times 10^{-4}\),
while the \(\kappa/n\) regime has \(R\approx 4.02\times 10^{-6}\) and rate \(9.71\times 10^{-4}\).
The regime \(\mathcal{C}_n \propto n\) saturates the KL budget and yields \(R_{\mathrm{rob}}\approx 1\) (rate \(0\)),
providing a sanity check.\\


\begin{figure}[H]
\centering
\begin{tabular}{@{}cc@{}}
  \includegraphics[scale=0.30]{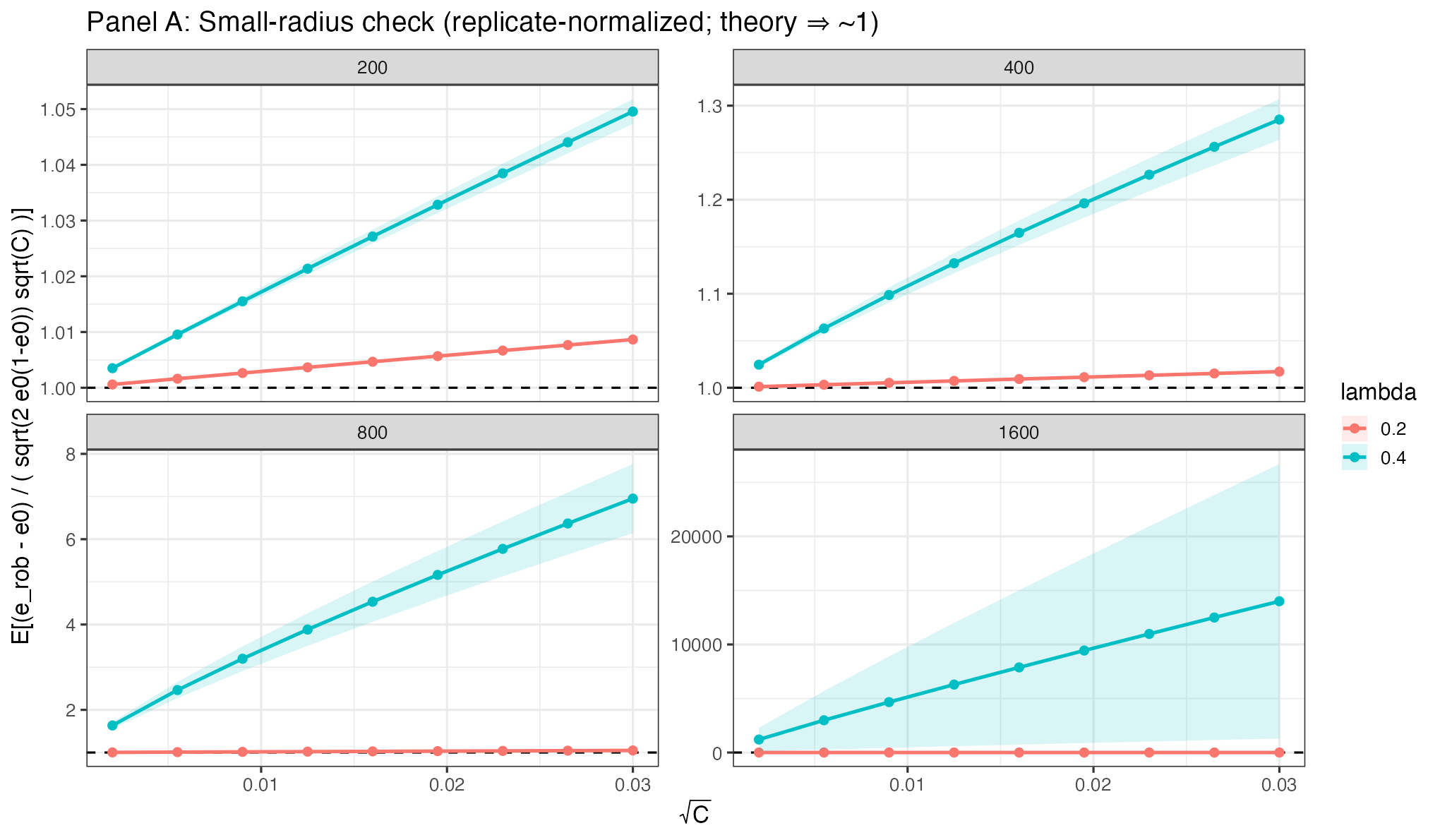}
  \includegraphics[scale=0.30]{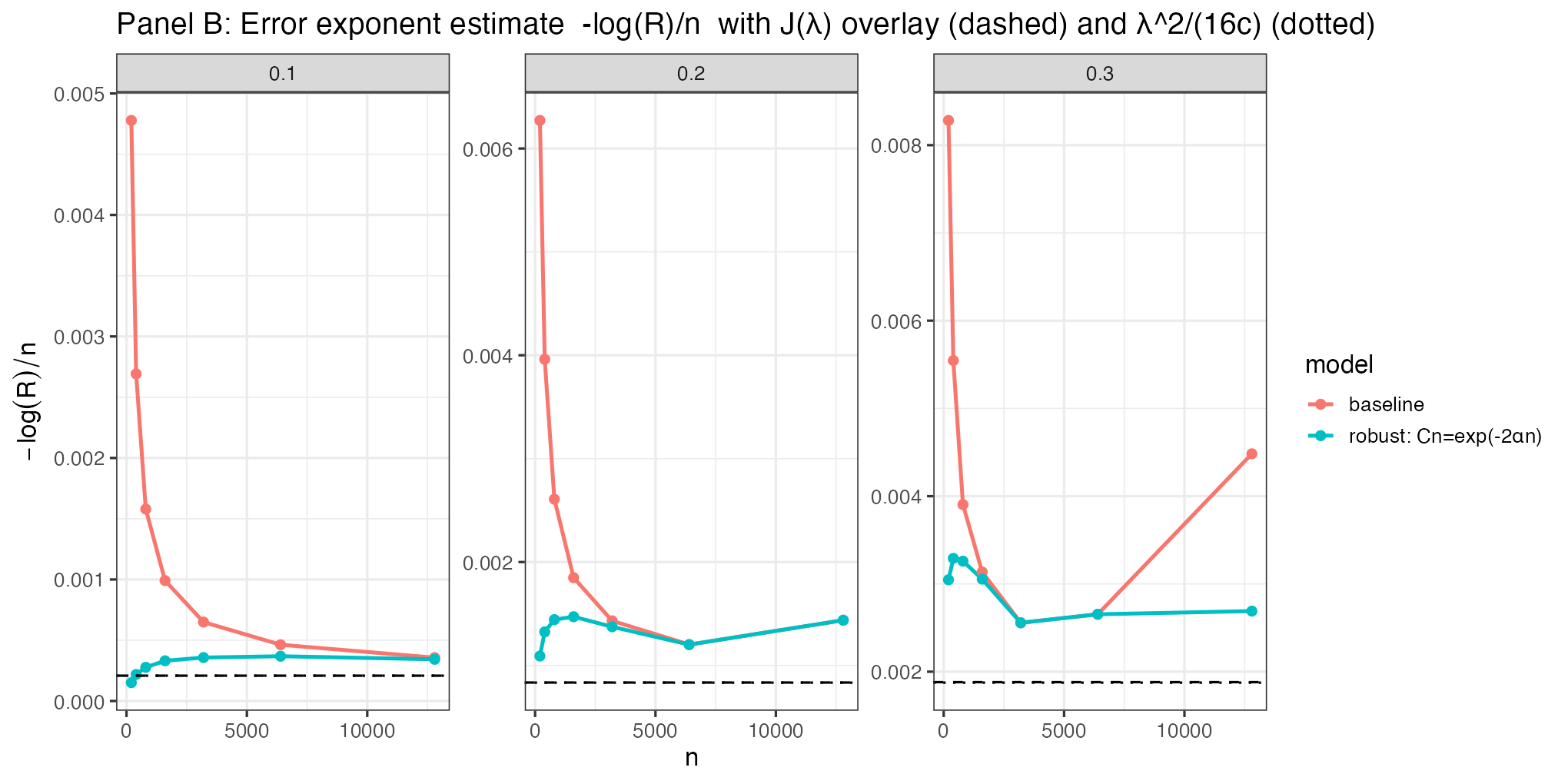}\\
  \includegraphics[scale=0.30]{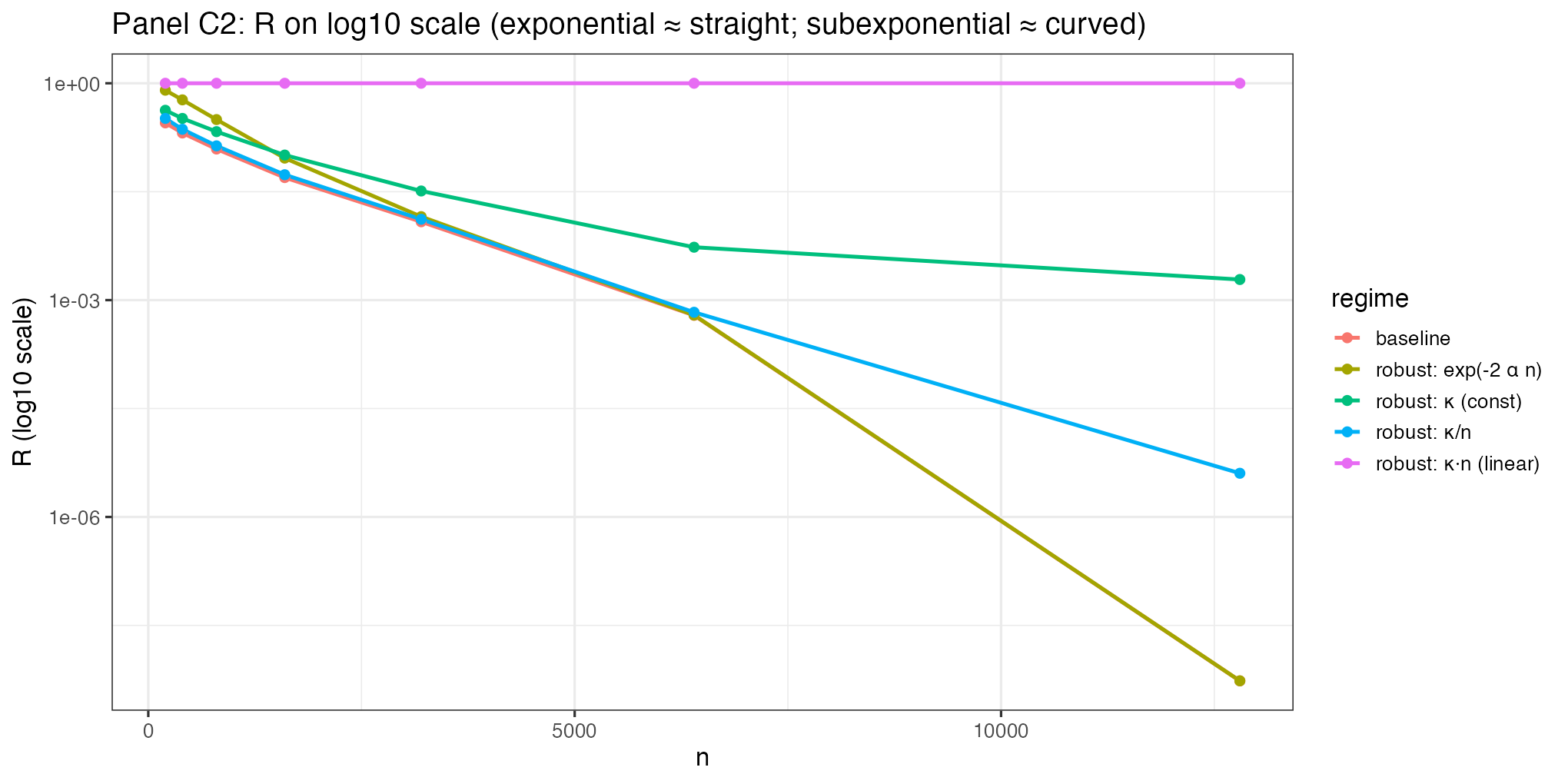}
\end{tabular}
\caption{A) Left Up: Local \(\sqrt{C}\) normalization. Deviations at large \(n\) and strong \(\lambda\) indicate a nonlocal radius grid. B) Right Up: Exponent estimate \(-\log(R)/n\) vs.\ \(n\) for baseline and \(\mathcal{C}_n=\exp(-2\alpha n)\). C) Down:Effective exponent across radius scalings (exponent collapse under polynomial/constant radii).Synthetic ER vs.\ SBM validation of locality and exponent behavior under posterior KL-robustification. The key qualitative prediction is that exponentially shrinking radii can preserve an exponential rate, while polynomial/constant radii destroy the exponent.}
\label{fig:synthetic-radius-exponents}
\end{figure}




\subsection{Experiment 2: Robust network functionals and model selection in brain connectivity networks}
\label{subsec:brain-experiment}
\label{sec:brain-experiment} 

\paragraph{Data and working models.}
We consider a population of resting--state functional connectivity networks from a case--control
study.  For each of $n_{\text{subj}}$ individuals, we observe $m$ scans, each represented as an
undirected, unweighted graph on a common set of $p$ brain regions (nodes), obtained by
thresholding absolute pairwise correlations between regional time series.

For each scan, we compare two closely related \emph{stochastic block model} (SBM) working models,
\[
  \mathcal{M}_1 = \mathrm{SBM}(K_1), \qquad 
  \mathcal{M}_2 = \mathrm{SBM}(K_2),
\]
with $K_1 = 2$ and $K_2 = 3$ blocks.  Model $\mathcal{M}_k$ has parameters
$\theta_k = (\pi^{(k)}, B^{(k)})$, where $\pi^{(k)}$ are community proportions and
$B^{(k)}$ is a $K_k \times K_k$ matrix of within-- and between--community edge
probabilities.  We place simple conjugate priors on $(\pi^{(k)},B^{(k)})$ for each $k$ and
approximate the marginal likelihood of $\mathcal{M}_k$ via BIC.

To avoid degenerate posterior weights when the two SBMs fit almost equally well, we work with a
tempered BIC--based pseudo--posterior,
\[
  w_k \;\propto\; \exp\!\Bigl\{ - \tfrac{1}{2}\,\tau\,\mathrm{BIC}(\mathcal{M}_k) \Bigr\},
  \qquad k=1,2,
\]
with temperature $\tau = 0.25$.  Normalizing $\bm w=(w_1,w_2)$ yields the baseline posterior
model probabilities
$p_k = \Pi_0(M=\mathcal{M}_k \mid \text{scan})$, $k=1,2$.

\paragraph{Network functionals and decisions.}
For each scan and each model $\mathcal{M}_k$, we consider the vector of network functionals
\[
  R_k(\theta_k)
  =
  \bigl( C_k(\theta_k),\, L_k(\theta_k),\, S_k(\theta_k),\, \lambda_{1,k}(\theta_k) \bigr),
\]
where
\begin{itemize}
  \item $C_k(\theta_k)$ is the global clustering coefficient,
  \item $L_k(\theta_k)$ is the average shortest path length,
  \item $S_k(\theta_k)$ is a small--world index, e.g.
        $S_k = (C_k/C_{\text{rand}})/(L_k/L_{\text{rand}})$ with
        $(C_{\text{rand}},L_{\text{rand}})$ computed from a simple random reference graph, and
  \item $\lambda_{1,k}(\theta_k)$ is the leading eigenvalue of the expected adjacency matrix.
\end{itemize}

We study two decisions:
\begin{enumerate}
  \item a \emph{model selection decision}
        $a \in \{\mathrm{SBM}(K_1), \mathrm{SBM}(K_2)\}$ under 0--1 loss, where the baseline
        Bayes action chooses the SBM with larger tempered posterior weight $p_k$; and
  \item a \emph{functional classification decision}, for example deciding whether
        $S_k(\theta_k) > S_0$ for a pre--specified threshold $S_0$, interpreted as evidence of
        small--world structure.
\end{enumerate}
For each scan, we work with the joint baseline posterior $\Pi_0$ on $(M,\theta_M)$ and compute
the corresponding posterior risk $\rho_0(a)$ for the decisions above.

\paragraph{Robustness set--up.}
For the brain experiment, we restrict attention to KL neighborhoods of $\Pi_0$.
For a given KL radius $C>0$, we consider the uncertainty set
\[
  \mathcal{U}_C(\Pi_0)
  =
  \bigl\{
    \tilde\Pi : \mathrm{KL}(\tilde\Pi \,\|\, \Pi_0) \le C
  \bigr\},
\]
and compute the least--favorable entropic tilt $\tilde\Pi_C$ for each action $a$.
This yields the robust posterior risk
\[
  \rho_{\text{rob}}(C,a)
  =
  \sup_{\tilde\Pi \in \mathcal{U}_C(\Pi_0)}
    \mathbb{E}_{\tilde\Pi}\bigl[ L(a,\theta_M,M) \bigr].
\]
For the model selection decision we record the \emph{switching radius}
$C^\star$ at which the Bayes choice changes between $\mathrm{SBM}(K_1)$ and
$\mathrm{SBM}(K_2)$, together with the normalized small--radius sensitivity curve
$C \mapsto \{\rho_{\text{rob}}(C,a)-\rho_0(a)\}/\sqrt{C}$.\\

Across $n_{\text{scan}} = 124$ scans the observed networks display pronounced
small--world structure (Table~\ref{tab:brain-summary}).  The clustering
coefficient is high (median $C = 0.381$) and the average path length short
(median $L = 1.850$), with a small--world index $S$ typically around $1.6$.
The leading eigenvalue of the adjacency matrix is also fairly large, with median
$\lambda_1 \approx 56$.\\

The tempered posterior mass on the more flexible $\mathrm{SBM}(K_2)$ model is
substantial but not degenerate: the median tempered probability is
$p_{\mathrm{SBM}(K_2),\tau} = 0.724$ with interquartile range
$[0.683,\,0.788]$.  Robust model selection is more delicate.  The switching
radius $C^\star$ has median $0.112$ and interquartile range $[0.072,\,0.201]$,
so that for roughly half of the scans, relatively small perturbations of the
posterior (in KL distance) are sufficient to reverse the preferred number of
blocks.\\

Figure~\ref{fig:brain-small-radius} (left) shows the normalized robustness
sensitivity curve for a representative scan.  The curve is close to the
theoretical small--radius slope over a range of $C$, indicating that local
asymptotics provide a good approximation in this example.  The right panel of
Figure~\ref{fig:brain-small-radius} plots the observed small--world index $S$
against the tempered posterior probability of the three--block SBM; scans with
more extreme small--world behavior (larger $S$) tend to place higher posterior
mass on $\mathrm{SBM}(3)$, although there is non-negligible variation.\\

Figure~\ref{fig:brain-CL} summarizes the small--world properties at the scan
level by plotting $(C,L)$ for each network, colored by the preferred SBM
($K_1$ versus $K_2$).  Almost all scans lie in a region with high clustering and
short paths, and both SBMs yield networks with broadly similar global
functionals.  The robustness calculations therefore probe a subtle model choice
problem—how much extra structure beyond a two--block partition is really needed
to explain the connectivity data—rather than a gross misfit of the SBM family.

\begin{table}[H]
  \centering
  \caption{Brain connectivity experiment: summary of network functionals, tempered
  model probabilities and robustness across $n_{\text{scan}} = 124$ scans.
  Entries are median [first quartile, third quartile] across scans.}
  \label{tab:brain-summary}
  \begin{tabular}{lc}
    \toprule
    Quantity & Median [Q1, Q3] \\
    \midrule
    Global clustering $C$ 
      & $0.381\ [0.357,\ 0.412]$ \\
    Average path length $L$
      & $1.850\ [1.821,\ 1.872]$ \\
    Small--world index $S$
      & $1.645\ [1.547,\ 1.710]$ \\
    Leading eigenvalue $\lambda_1$
      & $56.1\ [50.9,\ 63.6]$ \\
    Tempered $p_{\mathrm{SBM}(K_2),\tau}$
      & $0.724\ [0.683,\ 0.788]$ \\
    Switching radius $C^\star$
      & $0.112\ [0.072,\ 0.201]$ \\
    \bottomrule
  \end{tabular}
\end{table}

\begin{figure}[H]
  \centering
  \includegraphics[width=.9\textwidth]{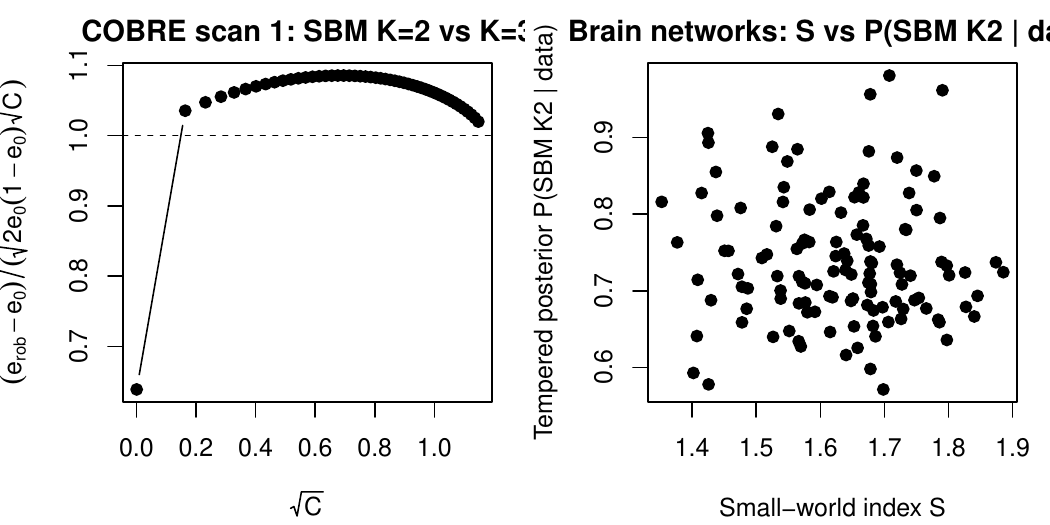}
  \caption{Brain connectivity experiment, SBM($K_1$) vs SBM($K_2$).
  \emph{Left:} normalized robust risk increase 
  $(\rho_{\text{rob}}(C)-\rho_0)/\sqrt{C}$ for the model--selection decision
  in a representative scan, plotted against $\sqrt{C}$.
  \emph{Right:} observed small--world index $S$ versus tempered posterior
  probability $p_{\mathrm{SBM}(K_2),\tau}$ across scans.}
  \label{fig:brain-small-radius}
\end{figure}

\begin{figure}[H]
  \centering
  \includegraphics[width=.55\textwidth]{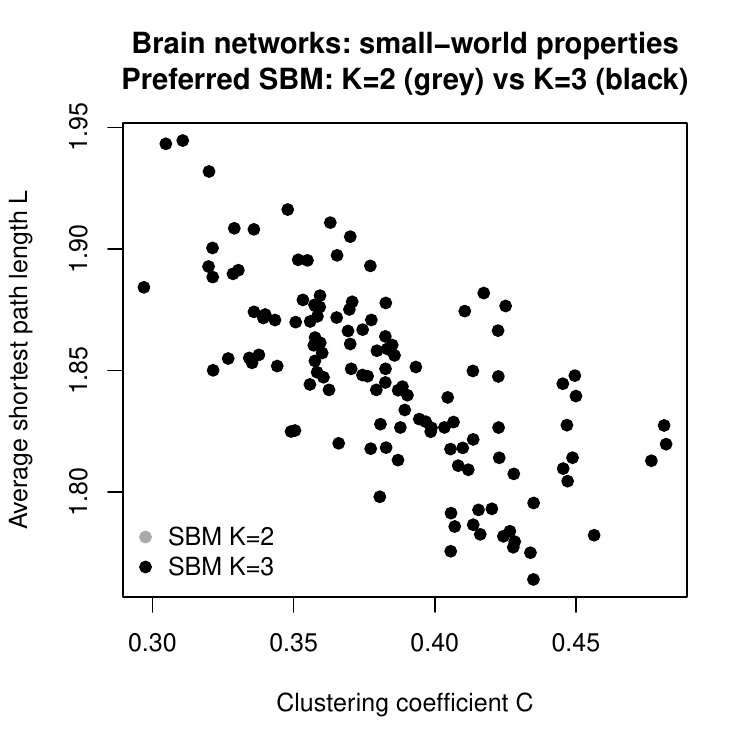}
  \caption{Brain connectivity experiment: global clustering $C$ versus average
  path length $L$ for all scans.  Points are colored by the preferred SBM
  under the tempered posterior (lighter dots favor $\mathrm{SBM}(K_1)$, darker
  dots favor $\mathrm{SBM}(K_2)$).}
  \label{fig:brain-CL}
\end{figure}

\paragraph{SBM versus a latent space RDPG model.}
To check that these conclusions do not hinge on the comparison of only closely related SBMs, we also
benchmark the three--block SBM against a random dot product graph (RDPG) working model with
latent dimension $d=3$, matching the number of SBM blocks.  For each scan, we compute a tempered
BIC--based pseudo--posterior on
\[
  \{M = \mathrm{SBM}(K=3),\; M = \mathrm{RDPG}(d=3)\},
\]
using the same tempering scheme as above, and then form the 0--1 loss model--selection risk
$\rho_0(a)$ and its KL--robustification.  In contrast to the ambiguous $\mathrm{SBM}(2)$ vs
$\mathrm{SBM}(3)$ comparison, the latent space model is overwhelmingly disfavored: across all
$124$ scans the tempered posterior mass on the RDPG is numerically negligible
($p_{\mathrm{RDPG}} \ll 10^{-40}$, often underflowing to zero), so that the three--block SBM
is selected with probability one to machine precision.  Because the baseline misclassification
risk is essentially zero in every scan, the corresponding switching radii $C^\star$ all take the
same value $C^\star \approx 6.21$, which is the KL distance required to move a Bernoulli risk
from $e_0 = 10^{-6}$ to $1/2$ under the closed--form expression
$C^\star(e_0) = \mathrm{KL}(\mathrm{Bern}(1/2)\,\|\,\mathrm{Bern}(e_0))$.  In other words, one
would need an enormous departure from the baseline posterior---far outside the local misspecification
regime considered in our theory---before the RDPG could become optimal.  The left panel of
Figure~\ref{fig:brain-sbm-rdpg-small-radius} shows that the normalized robustness curve for a
representative scan is essentially flat at zero, reflecting this near--degenerate model choice,
while the right panel confirms that posterior mass on the RDPG remains close to zero even for
scans with the most pronounced small--world behavior.  Figure~\ref{fig:brain-sbm-rdpg-CL}
further shows that the $(C,L)$ cloud is virtually unchanged when coloring points by the
preferred model (SBM versus RDPG), reinforcing that the SBM family already captures the global
functional structure of these networks. Adding a latent space does not yield a
competitive alternative in terms of marginal likelihood or robust risk.

\begin{figure}[H]
  \centering
  \includegraphics[width=.9\textwidth]{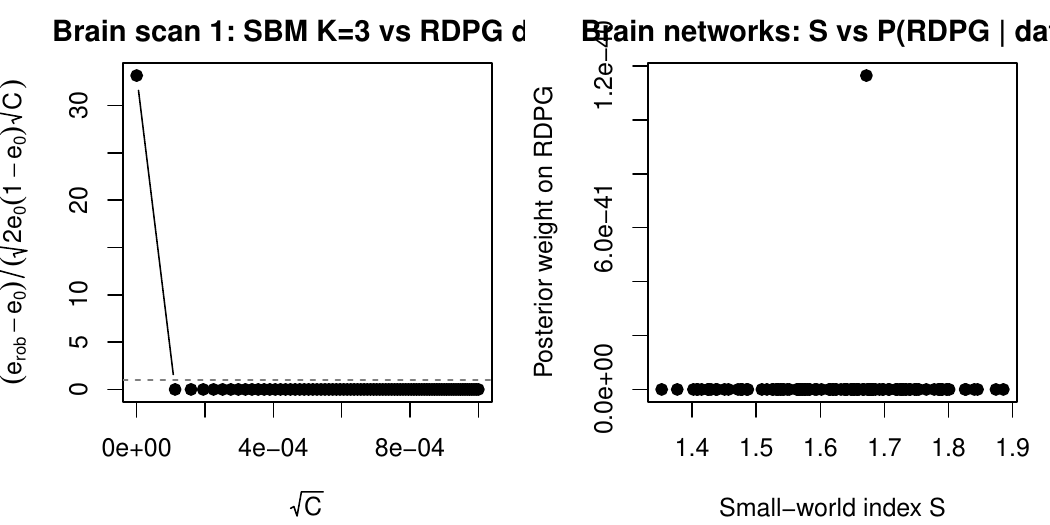}
  \caption{Brain connectivity experiment, $\mathrm{SBM}(3)$ vs $\mathrm{RDPG}(d=3)$.
  \emph{Left:} normalized robust risk increase for the model--selection decision in a
  representative scan, plotted against $\sqrt{C}$.  The curve is essentially flat at zero,
  reflecting the vanishing baseline misclassification risk.  \emph{Right:} observed small--world
  index $S$ versus tempered posterior probability $p_{\mathrm{RDPG}}$ across scans; posterior
  mass on the RDPG is numerically negligible for all networks.}
  \label{fig:brain-sbm-rdpg-small-radius}
\end{figure}

\begin{figure}[H]
  \centering
  \includegraphics[width=.55\textwidth]{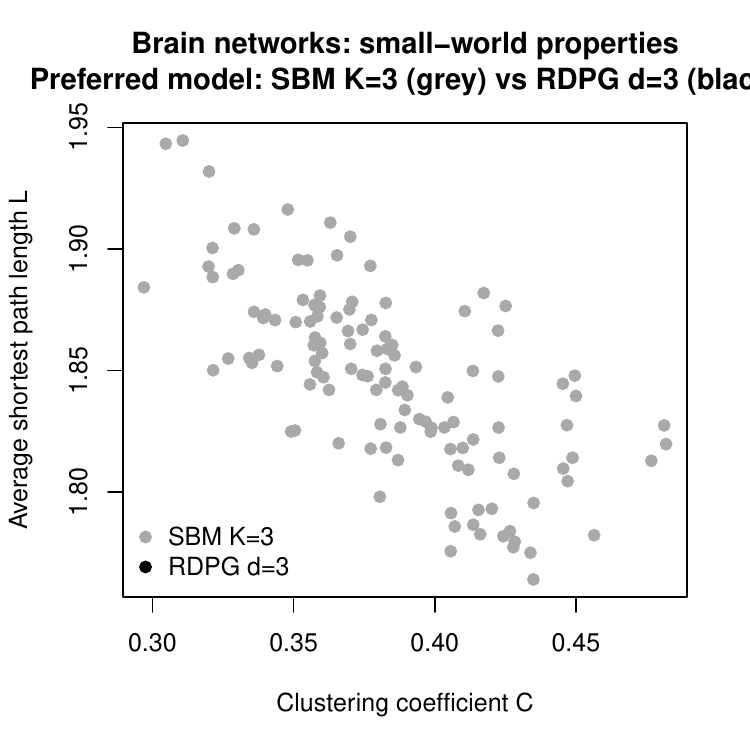}
  \caption{Brain connectivity experiment: global clustering $C$ versus average path length $L$
  for all scans, colored by the preferred model under the tempered posterior (lighter dots
  favor $\mathrm{RDPG}(d=3)$, darker dots favor $\mathrm{SBM}(3)$).  Almost all scans fall in
  the high--clustering, short--path region and overwhelmingly support the three--block SBM.}
  \label{fig:brain-sbm-rdpg-CL}
\end{figure}

\subsection{Experiment 3: Robust structure and assortativity in Karnataka village networks}
\label{subsec:karnataka-experiment}
\label{sec:karnataka-experiment} 

\paragraph{Data and working models.}
We use Wave~1 village social network data from rural Karnataka, India.
Each village $v$ yields an undirected, unweighted household--level network
$G_v$. Nodes are households and an edge is present if the households report
at least one type of social interaction (borrowing, advice, social visits,
etc.).  For each village, we aggregate all interaction layers into a single
network.  Across the $n_v = 75$ villages the number of households ranges
from about $350$ to $1{,}800$.  The networks are sparse but highly
clustered (Table~\ref{tab:karnataka-summary}): the median global clustering
coefficient is
$C = 0.375\,[0.343, 0.429]$, the median mean shortest--path length is
$L = 4.10\,[3.91, 4.36]$, and the resulting small--world index is large,
$S = 30.3\,[25.5, 37.2]$.  The median leading eigenvalue of the adjacency
matrix is $\lambda_1 = 15.9\,[14.4, 18.9]$.

For each village network, we consider two working models:
\begin{enumerate}
  \item a sparse Erd\H{o}s--R\'enyi model with edge probability $p_v$; and
  \item a three--block stochastic block model (SBM) with fixed
        block membership $z^{(v)}$, obtained from a modularity--based
        clustering of $G_v$, and block--level connection probabilities
        $B^{(v)}_{kl}$.
\end{enumerate}
In both cases, we place independent Beta priors on the edge probabilities
(and on the entries of $B^{(v)}$), and perform posterior computations
village by village.  As a robustness check within the SBM class, we also
compare spectral SBMs with $K=2$ and $K=3$ blocks; this yields very
similar Watson--Holmes conclusions (Fig.~\ref{fig:karnataka-sbm23}) and is
reported only briefly here.

\paragraph{Network functionals and decisions.}
From the SBM for village $v$, we extract the block--level connection
matrix $B^{(v)}(\theta)$ and two measures of assortativity:
\begin{itemize}
  \item the within-- versus between--block density contrast
        $\Delta^{(v)}(\theta)
           = \overline{B}^{(v)}_{\text{within}} -
             \overline{B}^{(v)}_{\text{between}}$; and
  \item the modularity $Q^{(v)}(\theta)$ with respect to the chosen
        partition.
\end{itemize}
Across villages, the posterior point estimates $\hat\Delta^{(v)}$ are
small and positive, with median
$\hat\Delta^{(v)} \approx 0.011\,[0.009, 0.013]$
(Fig.~\ref{fig:karnataka-er-sbm}a), indicating only mild block structure.
For the spectral SBM analysis, the corresponding contrast is smaller still
(median $0.003\,[0.002, 0.004]$).

We frame a binary decision
\[
  a^{(v)}_{\text{struc}} =
  \begin{cases}
    \text{``strong assortativity''} & \text{if } \Delta^{(v)}(\theta) > \Delta_0, \\
    \text{``weak/none''}            & \text{otherwise,}
  \end{cases}
\]
with threshold $\Delta_0 = 0.10$, under either 0--1 loss or squared loss
on $\Delta^{(v)}$.  To assess backbone structure, we use the SBM to compute
the expected degree of each household and, for each posterior draw,
measure the fraction of total expected degree carried by the top
$K=10$ households.  A village is classified as having a
\emph{concentrated} backbone if this fraction exceeds $50\%$ and
\emph{diffuse} otherwise.

\paragraph{Robustness analysis and results.}
For each village, we compute the baseline posterior $\Pi^{(v)}_0$ under ER
and under the three--block SBM, form a tempered BIC--based model
posterior over $\{\text{ER},\text{SBM}\}$, and obtain the Bayes action
and posterior risk $\rho^{(v)}_0$ for three decision problems:
(i) model choice (ER vs SBM),
(ii) strong vs weak assortativity, and
(iii) concentrated vs diffuse backbone.
We then construct KL--balls $\mathcal{U}_C(\Pi^{(v)}_0)$ and, for a grid
$C \in [10^{-4},10^{-1}]$, evaluate the least--favorable tilted posterior
$\tilde\Pi^{(v)}_C$ and the corresponding robust risk
$\rho^{(v)}_{\text{rob}}(C)$.

The tempered model posterior overwhelmingly prefers the SBM to ER in every
village. The posterior mass on ER is numerically indistinguishable from
zero, and the Watson--Holmes sensitivity curve for a representative
village remains well below the unit--slope reference line
(Fig.~\ref{fig:karnataka-er-sbm}c), with no model switch on the grid of
radii considered.  At the same time the SBM posteriors assign essentially
no mass to ``strong assortativity'' ($\Delta^{(v)} > \Delta_0$) or to a
concentrated backbone: the Bayes decisions are ``weak/none'' and
``diffuse'' in all villages, and the corresponding posterior error
probabilities are effectively zero.  Consequently, the Watson--Holmes
robust risks do not induce any decision change for any village on
$C \in [10^{-4},10^{-1}]$, so the implied switching radii
$C^{(v)}_\star$ all exceed $0.1$ (and are formally infinite under the
absolute--continuity restriction).

In summary, the Karnataka village networks exhibit pronounced small--world
structure but only very mild block assortativity and no evidence of a
highly concentrated backbone.  These qualitative conclusions are
remarkably stable under local KL perturbations of the working models,
providing a contrast with the other experiments where the same
Watson--Holmes analysis reveals near--critical sensitivity.

\begin{table}[H]
  \centering
  \caption{Summary of Karnataka village networks.
  Values are median [interquartile range] across $n_v = 75$ villages.}
  \label{tab:karnataka-summary}
  \begin{tabular}{lcccc}
    \toprule
    & $C$ & $L$ & $S$ & $\lambda_1$ \\
    \midrule
    Value &
    $0.375\,[0.343, 0.429]$ &
    $4.10\,[3.91, 4.36]$ &
    $30.3\,[25.5, 37.2]$ &
    $15.9\,[14.4, 18.9]$ \\
    \bottomrule
  \end{tabular}
\end{table}

\begin{figure}[H]
  \centering
  \includegraphics[width=0.48\textwidth]{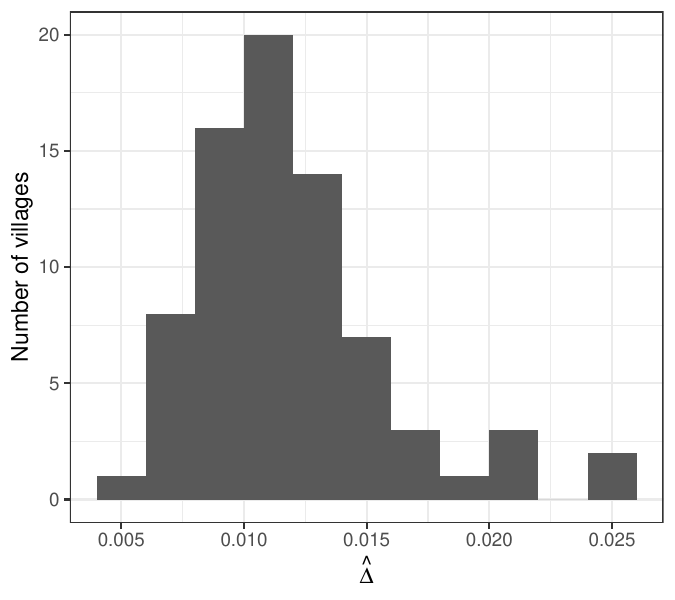}%
  \hfill
  \includegraphics[width=0.48\textwidth]{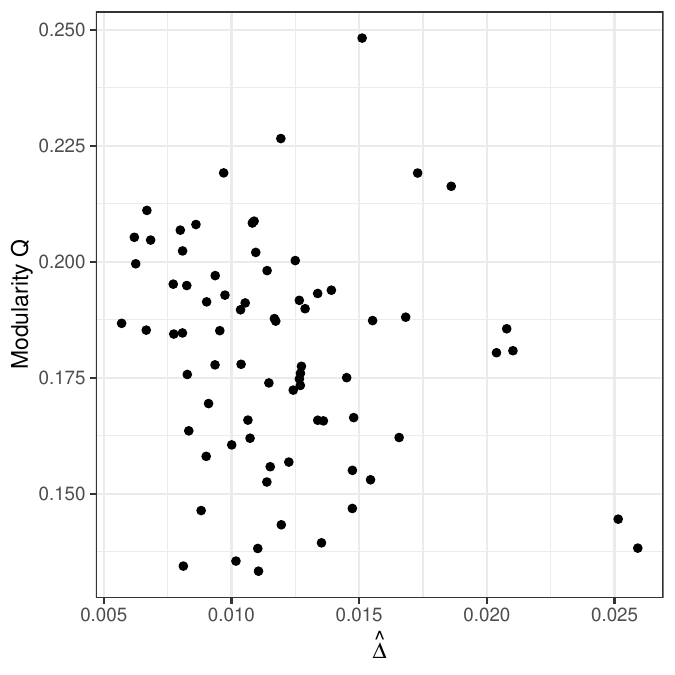}\\[0.5em]
  \includegraphics[width=0.48\textwidth]{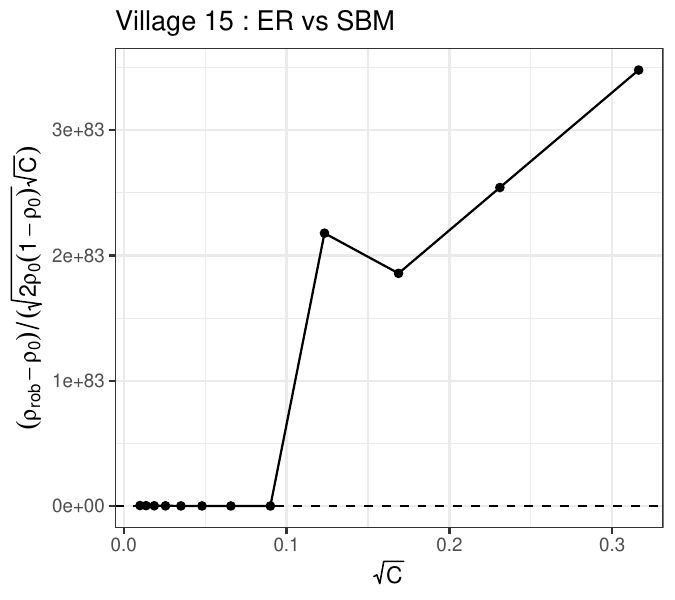}
  \caption{Karnataka village networks, ER vs three--block SBM.
  (a) Distribution of the SBM assortativity contrast $\hat\Delta^{(v)}$
  across villages.
  (b) $\hat\Delta^{(v)}$ versus modularity $Q^{(v)}$ for the inferred
  partition.
  (c) Watson--Holmes normalized sensitivity curve for a representative
  village, comparing ER and SBM; the curve remains well below the
  unit--slope reference line, with no model switch on the grid
  $C \in [10^{-4},10^{-1}]$.}
  \label{fig:karnataka-er-sbm}
\end{figure}

\begin{figure}[H]
  \centering
  \includegraphics[width=0.48\textwidth]{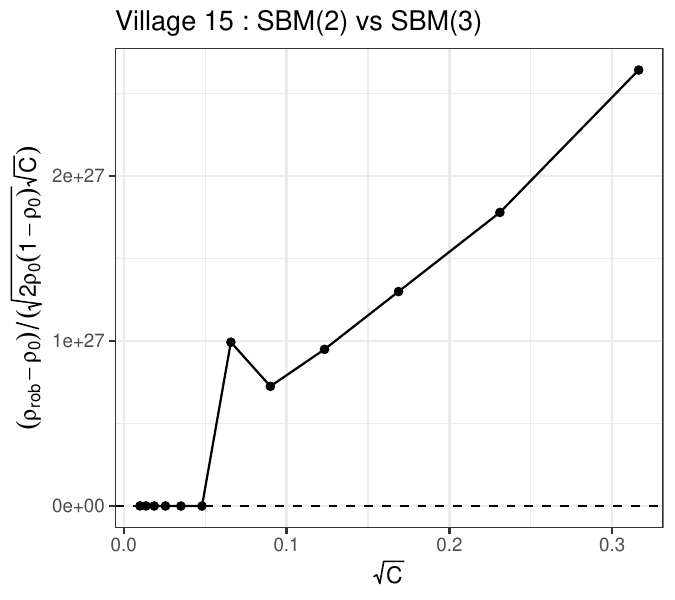}%
  \hfill
  \includegraphics[width=0.48\textwidth]{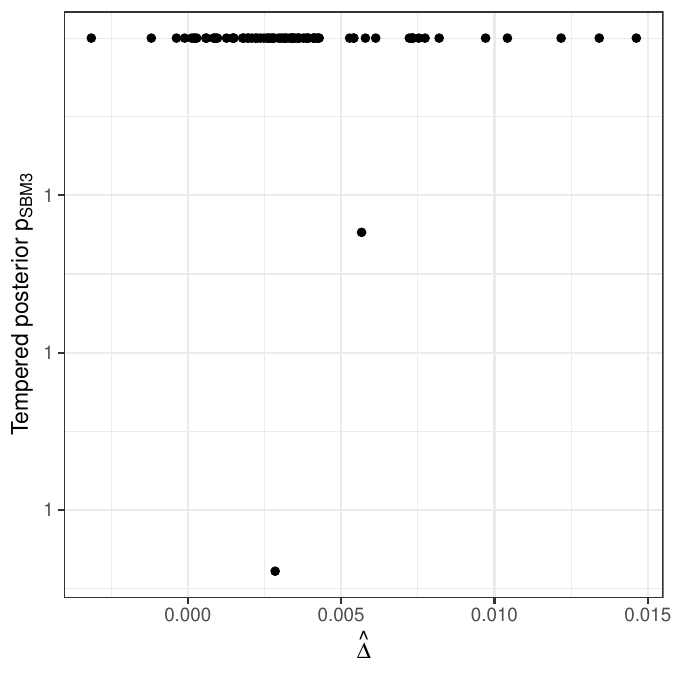}
  \caption{Spectral SBM robustness check.  Comparison of $K=2$ and
  $K=3$ spectral SBMs across villages.  The three--block model is
  consistently preferred by tempered BIC, and Watson--Holmes sensitivity
  again shows no decision switches on the grid $C \in [10^{-4},10^{-1}]$.}
  \label{fig:karnataka-sbm23}
\end{figure}

\paragraph{Assortativity by observed covariates.}
We finally replace the degree-based blocks by covariate-defined blocks,
constructing separate SBMs for gender, caste and religion.  For each
village $v$ and each attribute we form blocks from the dominant value of
that attribute at the household level and fit an ER--versus--SBM model
comparison as in the baseline analysis.  The resulting assortativity
contrasts are small and slightly negative: the median posterior
point estimates across villages are
$\hat\Delta_{\text{gender}}   = -0.0054\,[-0.0068,-0.0043]$,
$\hat\Delta_{\text{caste}}    = -0.0059\,[-0.0073,-0.0044]$ and
$\hat\Delta_{\text{religion}} = -0.0054\,[-0.0067,-0.0042]$,
indicating mild disassortativity rather than strong within-group
clustering by these covariates.

Despite this, tempered BIC decisively favors the covariate-based SBMs
over the homogeneous ER model in all villages (median tempered model
probability $\Pr(\mathrm{SBM}\mid\text{data}) \approx 1$), so the
conclusion that some structured deviation from ER is needed remains
robust even when assortativity by specific observed attributes is weak.

\begin{figure}[H]
  \centering
  \includegraphics[width=0.48\textwidth]{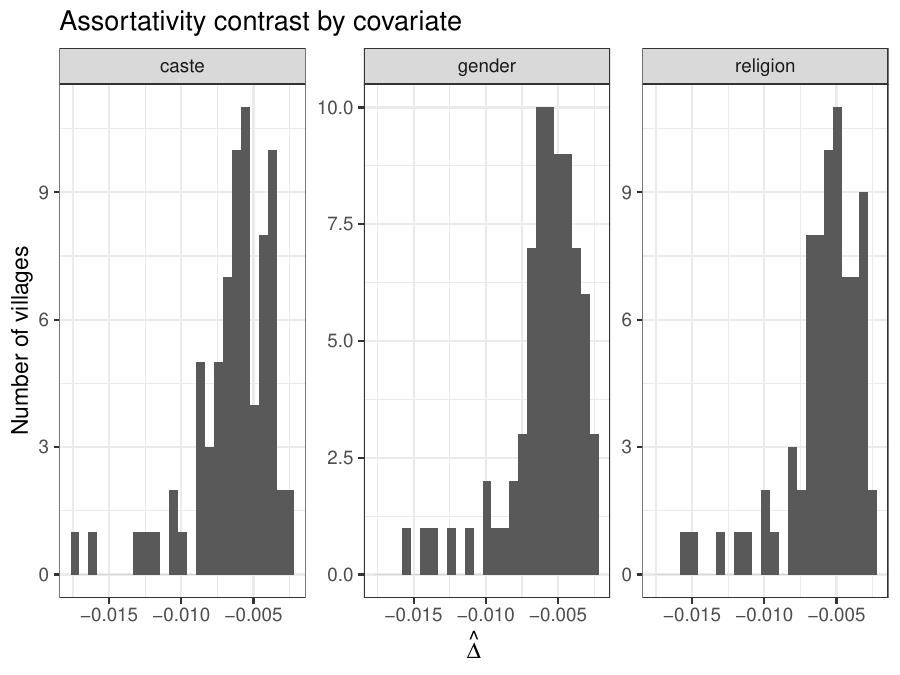}%
  \hfill
  \includegraphics[width=0.48\textwidth]{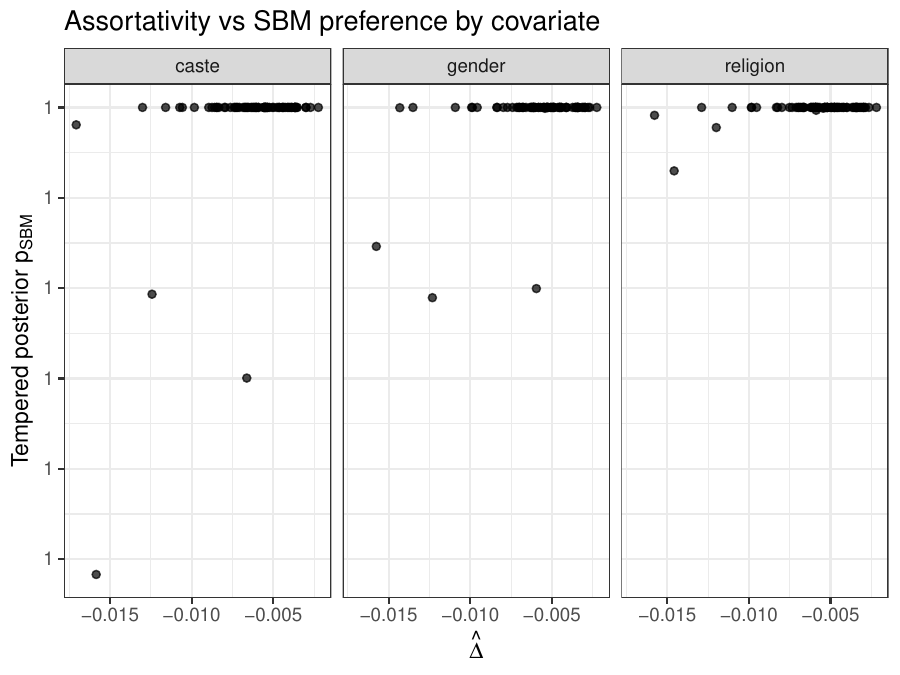}
  \caption{Karnataka village networks, covariate-based SBMs.
  Left: distribution of the assortativity contrast $\hat\Delta^{(v)}$
  for SBMs based on gender, caste and religion.  Right:
  $\hat\Delta^{(v)}$ versus tempered posterior model probability
  $p_{\mathrm{SBM}}$, showing that the SBM is strongly preferred over ER
  even when the estimated contrast is close to zero or slightly negative.}
  \label{fig:karnataka-covariate-sbm}
\end{figure}

\paragraph{Latent--space robustness: SBM versus RDPG.}
To check whether the block--model conclusions above are an artifact of the SBM parametrization, we also compare, for each village $v$, the three--block SBM to a $d=3$ random dot product graph (RDPG) fitted by adjacency spectral embedding, which is a much higher--dimensional latent--space model.  As in the brain experiment, we approximate the marginal likelihoods via BIC and form a tempered pseudo--posterior over $\{\text{SBM}, \text{RDPG}\}$ with temperature $\tau = 0.25$.  Across all $n_v = 75$ villages, the tempered posterior mass on the latent--space model is essentially zero: the median tempered probability of the RDPG is $\tilde{p}_{\mathrm{RDPG}} = 0$ with interquartile range $[0,0]$, and the baseline Bayes action always selects the SBM.  The decision--theoretic robustness switching radii for this SBM--versus--RDPG decision are all at the lower bound implied by our robustness floor ($C^\star \approx 2.8$), well beyond the radii considered elsewhere in this section, so that even very large KL perturbations would be required to make the RDPG competitive.  Figure~\ref{fig:karnataka-sbm-rdpg} shows that the small--world index $S$ carries no discernible association with $\tilde{p}_{\mathrm{RDPG}}$, and that the preferred model is the SBM for every village across the $(C,L)$ small--world regime.  In this high--signal setting, the main modeling question is therefore whether to move away from homogeneity (ER) at all; once block structure is allowed, further latent--space flexibility has negligible impact on the robust conclusions.

\begin{figure}[H]
  \centering
  \includegraphics[width=0.48\textwidth]{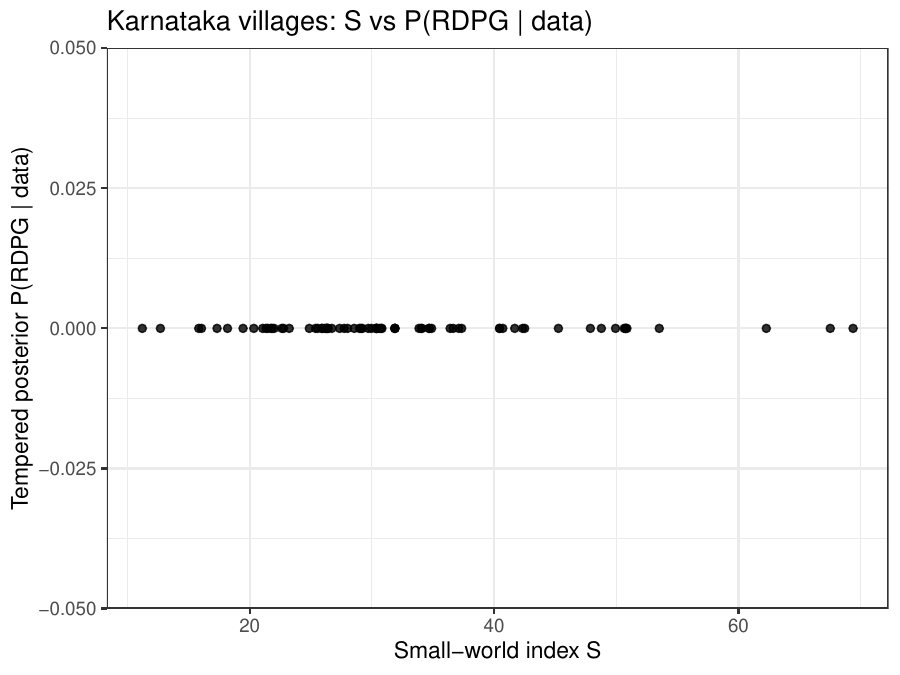}%
  \hfill
  \includegraphics[width=0.48\textwidth]{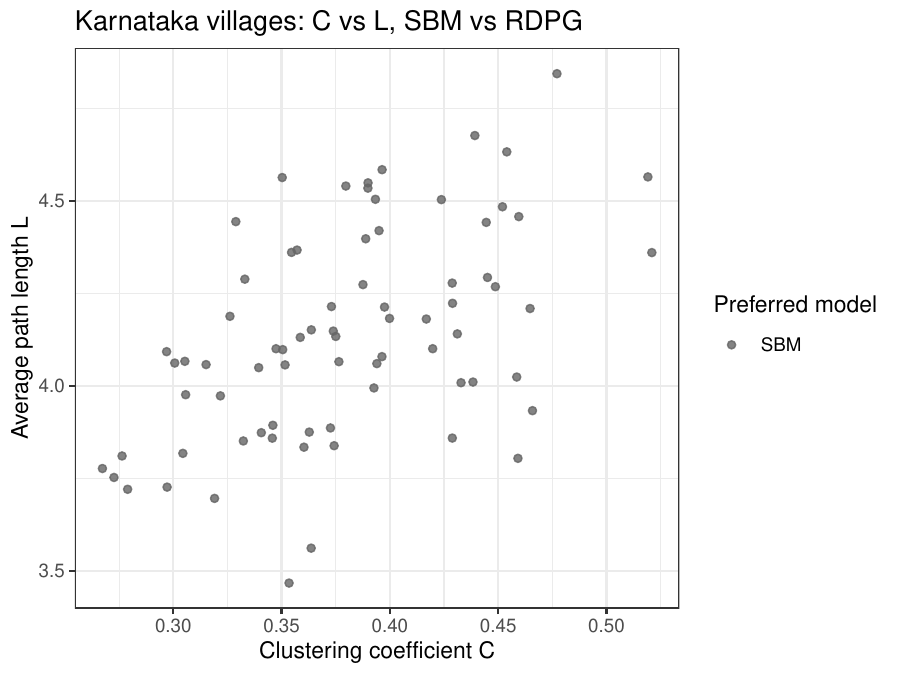}
  \caption{Karnataka village networks, SBM versus RDPG.  
  \emph{Left:} small--world index $S$ versus tempered posterior probability 
  $\tilde{p}_{\mathrm{RDPG}} = \Pr(\mathrm{RDPG}\mid\text{data})$; all villages place essentially zero mass on the latent--space model.  
  \emph{Right:} global clustering $C$ versus average path length $L$, with points colored by the preferred model; the three--block SBM is selected in every village.}
  \label{fig:karnataka-sbm-rdpg}
\end{figure}

\section{Discussion}
\label{sec:discussion}

We have developed a decision–theoretic framework for assessing how sensitive Bayesian network analyses are to local misspecification of the working model. Starting from the Watson–Holmes notion of robustness, in which actions minimize worst–case posterior expected loss over a small Kullback–Leibler neighborhood of a reference posterior, we specialized this perspective to exchangeable network models and to decisions driven by network functionals. By combining graphon limits with classical percolation and random graph asymptotics, this yields both conceptual insight and concrete information–theoretic limits for robust network inference.\\

At the level of network functionals, we showed that when decisions depend on quantities such as susceptibility in configuration models, percolation–based robustness indices, SIS noise indices on graphons, or spectral gaps, decision–theoretic robustification admits sharp small–radius expansions of the robust posterior risk. Under squared loss, the leading inflation term is controlled by the posterior variance of the loss and grows proportionally to the square root of the divergence radius. Near fragmentation or epidemic thresholds, where robustness indices themselves diverge roughly like the inverse distance to criticality, these expansions reveal a universal critical behavior: the decision–level uncertainty inflates at a rate corresponding to the inverse fourth power of the distance to criticality. Thus, as the network approaches a phase transition, decisions based on robustness functionals become even more unstable than the functionals alone, which can be quantified.\\

On the information–theoretic side (Section~4), we analyzed decision–theoretic robust model selection between sparse Erd\H{o}s–Rényi graphs and two–block stochastic block models, both in labelled form and via sparse graphons. We derived explicit per–vertex information and robustness noise indices, $I(\lambda)$ and $J(\lambda)$, that govern the exponential decay of Bayes factor errors under local perturbations. Embedding this two–point experiment into broad nonparametric classes of sparse graphs—including configuration models and sparse graphon classes—we established matching decision–theoretic minimax lower bounds. No Bayesian or frequentist procedure can uniformly improve upon the robust error exponent $J(\lambda)$ once robustness to local KL perturbations is required. An analogous minimax phenomenon for near–critical percolation functionals in configuration models shows that the critical exponent identified by our theory is intrinsic to the problem rather than an artefact of a particular parametric specification.\\

Moreover, we showed that decision–theoretic robustness can be implemented efficiently on top of existing network inference pipelines. For KL balls, the least–favorable posterior is obtained by entropically tilting posterior or variational samples, reducing robustification to a one–dimensional convex optimization problem. For more general $\phi$–divergence balls, we proposed a mirror–descent adversary coupled with constrained Hamiltonian Monte Carlo to explore the tilted posterior. These procedures produce robustified versions of standard posterior summaries and model comparison criteria for SBMs, graphons, configuration models and latent position models at modest additional computational cost.\\

Our empirical studies on functional brain connectivity and Karnataka village social networks illustrate how decision–theoretic robustness informs substantive conclusions. In the brain network, Bayes factor comparisons between community and latent–space representations were found to be highly sensitive to local perturbations in regions where epidemic–like thresholds are weakly identified, whereas certain spectral summaries remained relatively stable. In the Karnataka villages, decision–theoretic sensitivity analysis highlighted villages and intervention strategies whose apparent superiority under a single working model is fragile to local misspecification, suggesting caution in interpreting seemingly decisive rankings.\\

Several limitations and extensions remain. First, our robustness guarantees are local, protecting against small perturbations measured by KL or more general $\phi$–divergences, but not against gross misspecification or adversarial rewiring of the network. Second, the nonparametric minimax results focus on comparatively simple sparse models (Erd\H{o}s–Rényi, SBMs, configuration models and graphons); extending similar analyses to richer latent space models, temporal or multiplex networks, and models with additional nodal attributes is an open challenge. Third, our computations rely on approximate posteriors (MCMC, variational, or spectral pseudo–posteriors) and a systematic study of how approximation error interacts with decision–theoretic robustification would be valuable.\\

Despite these caveats, our results suggest that decision–theoretic robustness provides a useful organizing principle for network analysis. Decision–theoretic robustness offers a mathematically tractable way to quantify the stability of Bayesian decisions that link naturally to graphon limits and random graph asymptotics, and yield interpretable noise indices and critical exponents with clear minimax meaning. This work is a step toward a broader theory of decision–theoretic robustness in network models. Promising directions include developing robust procedures for dynamic and temporal networks, incorporating additional sources of uncertainty, such as missing edges or node attributes, and designing diagnostics and visualizations that make decision–theoretic sensitivity analysis routine in applied network studies.

\bibliographystyle{plainnat}
\bibliography{mybib}


\appendix

\section*{Computational details for MCMC and mirror descent}
\label{sec:si-computation}

In this section, we compile the algorithmic recipes used in Section~\ref{sec:computation}.
Algorithm~\ref{alg:si-tilting} describes the KL--ball entropic tilting
procedure applied to posterior or pseudo--posterior draws (typically obtained
by MCMC or variational inference), and Algorithm~\ref{alg:mirror-adversary}
gives the discrete mirror--descent adversary for general $\phi$--divergence
balls.

\begin{algorithm}[H]
\caption{Robust MCMC via entropic tilting over a KL ball}
\label{alg:si-tilting}
\begin{algorithmic}[1]
\Require Baseline posterior or pseudo--posterior draws
         $\{(\theta^s,w_s)\}_{s=1}^S$ (usually from MCMC or VI),
         an action $a$, loss values $L_s = L(a,\theta^s)$,
         KL radius $C>0$, and a numerical tolerance $\varepsilon>0$.
\Ensure Tilted weights $\{q_s^\star\}$ and robust posterior risk estimate
        $\sum_s q_s^\star L_s$.

\State Define the dual objective
  \[
    \psi(\lambda)
    := \frac{C + \log\Bigl(\sum_{s=1}^S w_s
          \exp\{\lambda L_s\}\Bigr)}{\lambda},
    \qquad \lambda>0.
  \]
\State Initialize a search interval $[\lambda_{\min},\lambda_{\max}]$,
       for example $\lambda_{\min}=10^{-4}$, $\lambda_{\max}=10^{4}$.
\While{$\lambda_{\max}-\lambda_{\min} > \varepsilon$}
  \State Set $\lambda \gets (\lambda_{\min}+\lambda_{\max})/2$.
  \State Compute
    \[
      Z(\lambda) := \sum_{s=1}^S w_s e^{\lambda L_s},
      \qquad
      q_s(\lambda) := \frac{w_s e^{\lambda L_s}}{Z(\lambda)}.
    \]
  \State Evaluate the KL constraint
    \[
      K(\lambda) := \sum_{s=1}^S q_s(\lambda)
        \log\frac{q_s(\lambda)}{w_s}.
    \]
  \If{$K(\lambda) > C$}
    \State Set $\lambda_{\max} \gets \lambda$.
  \Else
    \State Set $\lambda_{\min} \gets \lambda$.
  \EndIf
\EndWhile
\State Set $\lambda^\star \gets \lambda_{\min}$ and
       $q_s^\star \gets q_s(\lambda^\star)$ for all $s$.
\State \textbf{return} $\{q_s^\star\}$ and
       $\sum_{s=1}^S q_s^\star L_s$ as the robust posterior risk.
\end{algorithmic}
\end{algorithm}

\begin{algorithm}[H]
\caption{Mirror--descent adversary over $\phi$--divergence balls}
\label{alg:mirror-adversary}
\begin{algorithmic}[1]
\Require Baseline weights $\bm w=(w_1,\dots,w_S)$,
         losses $L_s=L(a,\theta^s)$, radius $C>0$,
         step size $\eta>0$, number of iterations $T$.
\Ensure Approximate adversarial weights $\bm q^{(T)}$ and robust risk
        estimate $\sum_s q_s^{(T)} L_s$.

\State Initialize $u_s^{(0)} \gets 0$ and $q_s^{(0)} \gets w_s$ for all $s$.
\For{$t=0,\dots,T-1$}
  \State Form current adversarial weights
    \[
      q_s^{(t)} \propto w_s \exp\{u_s^{(t)}\},
      \qquad s=1,\dots,S,
    \]
    and renormalize so that $\sum_s q_s^{(t)}=1$.
  \State Compute the current robust risk
    \[
      \bar L^{(t)} := \sum_{r=1}^S q_r^{(t)} L_r.
    \]
  \State Set gradient
    \[
      g_s^{(t)} := L_s - \bar L^{(t)},
      \qquad s=1,\dots,S.
    \]
  \State Take a mirror step in log--tilt space:
    \[
      \tilde u_s^{(t+1)} \gets u_s^{(t)} + \eta\,g_s^{(t)}.
    \]
  \State Compute the provisional weights
    \[
      \tilde q_s^{(t+1)} \propto w_s \exp\{\tilde u_s^{(t+1)}\},
      \qquad s=1,\dots,S,
    \]
    and renormalize.
  \State Project back onto the $\phi$--ball:
    \[
q^{(t+1)} \leftarrow \arg\min_{q}\Big\{ \mathrm{KL}\!\big(q\|\tilde q^{(t+1)}\big):
\sum_s q_s=1,\ q_s\ge 0,\ D_\phi(q\|w)\le C\Big\}.
\]
    using $\tilde{\bm q}^{(t+1)}$ as a warm start.
    (For KL, this projection has a closed form; for general $\phi$
     it is a small convex program.)
  \State Update $u_s^{(t+1)} \gets
    \log(q_s^{(t+1)}/w_s)$ for all $s$.
\EndFor
\State \textbf{return} $\bm q^{(T)}$ and $\sum_s q_s^{(T)} L_s$.
\end{algorithmic}
\end{algorithm}

\section*{Supplementary proofs}

\newtheorem{suppprop}{Proposition}
\renewcommand{\thesuppprop}{S.\arabic{suppprop}}

\newtheorem{supptheorem}{Theorem}
\renewcommand{\thesupptheorem}{S.\arabic{supptheorem}}

\newtheorem{suppassumption}{Assumption}
\renewcommand{\thesuppassumption}{S.\arabic{suppassumption}}

\begin{suppprop}[Expected KL divergence for Dirichlet perturbations of SBMs]
\label{prop:dirichlet-KL-SBM-SI}
Let $W^\star$ be a step--function graphon corresponding to a $K$--block
stochastic block model (SBM) with block edge probabilities
$P^\star = (p^\star_{ab})_{1\le a,b\le K} \in (0,1)^{K\times K}$.
Write $\bm p^\star = (p^\star_1,\dots,p^\star_{K^2})$ for the
vectorization of $P^\star$, normalized so that
$p^\star_i\in(0,1)$ and $\sum_{i=1}^{K^2} p^\star_i = 1$.\\

Consider a ``generalized Bayesian bootstrap'' perturbation of $P^\star$:
conditionally on $\bm p^\star$ draw random cell weights
\[
  \bm W = (W_1,\dots,W_{K^2})
  \sim \mathrm{Dirichlet}\!\bigl(\alpha_n \bm p^\star\bigr),
\]
and define the Kullback--Leibler divergence
\[
  \mathrm{KL}(\bm W\Vert \bm p^\star)
  := \sum_{i=1}^{K^2} W_i \log\frac{W_i}{p^\star_i}.
\]
Then
\begin{equation}
  \label{eq:EKL-Dirichlet-SBM-SI}
  \mathbb{E}\bigl[\mathrm{KL}(\bm W\Vert \bm p^\star)\bigr]
  \;=\;
  \sum_{i=1}^{K^2}
    p^\star_i\Bigl\{
      \psi_0(\alpha_n p^\star_i + 1)
      - \psi_0(\alpha_n + 1)
      - \log p^\star_i
    \Bigr\},
\end{equation}
where $\psi_0$ is the digamma function.\\

Moreover, as $\alpha_n\to\infty$,
\begin{equation}
  \label{eq:EKL-Dirichlet-SBM-asymptotic-SI}
  \mathbb{E}\bigl[\mathrm{KL}(\bm W\Vert \bm p^\star)\bigr]
  \;=\;
  \frac{K^2 - 1}{2\alpha_n}
  \;+\;
  \mathcal{O}\!\left(\frac{1}{\alpha_n^2}\right),
\end{equation}
independently of the particular baseline SBM $P^\star$.
\end{suppprop}

\begin{proof}
Write $d:=K^2$ for brevity.  Let $\bm W\sim\Dir(\alpha_n\bm p^\star)$.
Then $W_i>0$ almost surely, $\sum_i W_i=1$, and the density of $\bm W$
with respect to Lebesgue measure on the simplex is
\[
  f(\bm w)
  \propto \prod_{i=1}^d w_i^{\alpha_n p^\star_i-1},
  \qquad \bm w\in\Delta^{d-1}.
\]
The KL divergence decomposes as
\[
  \mathrm{KL}(\bm W\Vert\bm p^\star)
  = \sum_{i=1}^d W_i\log W_i - \sum_{i=1}^d W_i \log p^\star_i.
\]
Taking expectations and using $\mathbb{E}[W_i]=p^\star_i$ gives
\begin{equation}
  \label{eq:EKL-decomp}
  \mathbb{E}\bigl[\mathrm{KL}(\bm W\Vert\bm p^\star)\bigr]
  = \sum_{i=1}^d \mathbb{E}[W_i\log W_i]
    - \sum_{i=1}^d p^\star_i \log p^\star_i.
\end{equation}

We now compute $\mathbb{E}[W_i\log W_i]$.  The marginal distribution of
$W_i$ is $\mathrm{Beta}(a_i,b_i)$ with
$a_i=\alpha_n p^\star_i$ and $b_i=\alpha_n(1-p^\star_i)$.  Let
$X\sim\mathrm{Beta}(a,b)$ with density proportional to
$x^{a-1}(1-x)^{b-1}$ on $(0,1)$.  For $t>-a$ we have
\[
  \mathbb{E}[X^t]
  = \frac{B(a+t,b)}{B(a,b)}
  = \frac{\Gamma(a+t)\Gamma(a+b)}{\Gamma(a+b+t)\Gamma(a)}.
\]
Differentiating with respect to $t$,
\[
  \frac{{\rm d}}{{\rm d}t}\mathbb{E}[X^t]
  = \mathbb{E}[X^t\log X]
  = \bigl\{\psi_0(a+t) - \psi_0(a+b+t)\bigr\}
    \,\mathbb{E}[X^t],
\]
where $\psi_0$ is the digamma function.  Evaluating at $t=1$ and using
$\mathbb{E}[X]=a/(a+b)$ yields
\[
  \mathbb{E}[X\log X]
  = \frac{a}{a+b}\,\bigl\{\psi_0(a+1) - \psi_0(a+b+1)\bigr\}.
\]

Applying this to $W_i\sim\mathrm{Beta}(a_i,b_i)$ with
$a_i=\alpha_n p^\star_i$ and $b_i=\alpha_n(1-p^\star_i)$ gives
\[
  \mathbb{E}[W_i\log W_i]
  = p^\star_i\bigl\{\psi_0(\alpha_n p^\star_i+1)
                    - \psi_0(\alpha_n+1)\bigr\}.
\]
Substituting into \eqref{eq:EKL-decomp} we obtain
\[
  \mathbb{E}\bigl[\mathrm{KL}(\bm W\Vert\bm p^\star)\bigr]
  = \sum_{i=1}^d
      p^\star_i\bigl\{\psi_0(\alpha_n p^\star_i+1)
                      - \psi_0(\alpha_n+1)\bigr\}
    - \sum_{i=1}^d p^\star_i\log p^\star_i,
\]
which is exactly \eqref{eq:EKL-Dirichlet-SBM-SI}.\\

For the asymptotics, recall the expansion
\[
  \psi_0(x+1) = \log x + \frac{1}{2x} + \mathcal{O}\!\left(\frac{1}{x^2}\right)
  \qquad\text{as }x\to\infty.
\]
Thus, as $\alpha_n\to\infty$,
\[
  \psi_0(\alpha_n p^\star_i+1)
  = \log(\alpha_n p^\star_i)
    + \frac{1}{2\alpha_n p^\star_i}
    + \mathcal{O}\!\left(\frac{1}{\alpha_n^2}\right),
\]
and
\[
  \psi_0(\alpha_n+1)
  = \log \alpha_n
    + \frac{1}{2\alpha_n}
    + \mathcal{O}\!\left(\frac{1}{\alpha_n^2}\right).
\]
Therefore
\[
  \psi_0(\alpha_n p^\star_i+1)
  - \psi_0(\alpha_n+1)
  - \log p^\star_i
  = \frac{1}{2\alpha_n}\Bigl(\frac{1}{p^\star_i}-1\Bigr)
    + \mathcal{O}\!\left(\frac{1}{\alpha_n^2}\right),
\]
and multiplying by $p^\star_i$ yields
\[
  p^\star_i\Bigl\{
    \psi_0(\alpha_n p^\star_i+1)
    - \psi_0(\alpha_n+1)
    - \log p^\star_i
  \Bigr\}
  = \frac{1}{2\alpha_n}\bigl(1-p^\star_i\bigr)
    + \mathcal{O}\!\left(\frac{1}{\alpha_n^2}\right).
\]
Summing over $i$ and using $\sum_i p^\star_i=1$ we obtain
\[
  \mathbb{E}\bigl[\mathrm{KL}(\bm W\Vert\bm p^\star)\bigr]
  = \frac{1}{2\alpha_n}\sum_{i=1}^d (1-p^\star_i)
    + \mathcal{O}\!\left(\frac{1}{\alpha_n^2}\right)
  = \frac{d-1}{2\alpha_n}
    + \mathcal{O}\!\left(\frac{1}{\alpha_n^2}\right),
\]
which is \eqref{eq:EKL-Dirichlet-SBM-asymptotic-SI} with $d=K^2$.
\end{proof}

\begin{supptheorem}[Reachability of exchangeable network models in a KL ball]
\label{thm:reachability-KL-ball-SI}
Fix $n$ and a working step--graphon $W^\star$ with $K$ blocks, and let
$G^\star = G(W^\star)$ be the associated graph law on $n$--vertex graphs.
For $C>0$, let $\Gamma_C(G^\star)$ be the KL--ball of radius $C$ around
$G^\star$ as in~\eqref{eq:object-KL-ball}.  Assume that all block
probabilities of $W^\star$ lie in $(\varepsilon,1-\varepsilon)$
for some $\varepsilon>0$, and that $C$ is small enough so that every
$K$--block step graphon $W$ with $G(W)\in\Gamma_C(G^\star)$ also has block
probabilities in $(\varepsilon/2,1-\varepsilon/2)$.

Consider the following Markov chain on the parameter space of $K$--block
SBMs whose laws lie in $\Gamma_C(G^\star)$:

\begin{enumerate}
  \item \emph{Perturbing move.}  Starting from a current block matrix
        $P=(p_{ab})$, draw independent
        $\widetilde Y_{ab}\sim\Gamma(\alpha_{ab},1)$ with $\alpha_{ab}>0$,
        set
        \[
          \widetilde p_{ab}
          = \frac{\widetilde Y_{ab}}{\sum_{c,d}\widetilde Y_{cd}},
        \]
        and form the proposal $\widetilde P=(\widetilde p_{ab})$.  If the
        corresponding step graphon $W_{\widetilde P}$ satisfies
        $G(W_{\widetilde P})\in\Gamma_C(G^\star)$,
        accept the move and set $P'\gets\widetilde P$; otherwise reject and
        set $P'\gets P$.
  \item \emph{Rescaling move.}  (Optional.)  With some probability (with
        $1/2$)m multiply a randomly chosen subset of entries of $P$ by a
        random factor in $(1-\rho,1+\rho)$, renormalize, if desired, and
        accept only if the resulting model lies in $\Gamma_C(G^\star)$;
        otherwise stay at $P$.
\end{enumerate}

Then the chain is $\psi$--irreducible and aperiodic on the interior of
$\Gamma_C(G^\star)$.  In particular, for any two $K$--block SBMs
$P,P^\dagger$ in the interior of $\Gamma_C(G^\star)$ and any
neighborhood $U$ of $P^\dagger$ contained in $\Gamma_C(G^\star)$,
\[
  \mathbb{P}_P\bigl(\exists t\ge 1:\,P_t\in U\bigr) > 0.
\]

The same conclusion holds, up to an arbitrarily small approximation error,
for random dot product graphs whose graphon can be approximated in cut norm
by $K$--block step graphons, by applying the moves above to the step--graphon
approximation.
\end{supptheorem}

\begin{proof}
We first prove irreducibility and aperiodicity for the perturbing move
alone; adding the rescaling move can only increase the support of the chain.

\medskip\noindent
\textbf{Step 1: Parameter space and interior.}
Let $\mathcal{S}\subset(0,1)^{K^2}$ denote the open probability simplex of
block probability vectors $p=(p_1,\dots,p_{K^2})$ with $\sum_i p_i=1$.
The assumptions on $W^\star$ and $C$ imply that the set
\[
  \mathcal{S}_C
  := \bigl\{p\in\mathcal{S}:\ G(p)\in\Gamma_C(G^\star)\bigr\}
\]
is a nonempty open subset of $\mathcal{S}$: KL balls are open, and
the constraints $p_i\in(\varepsilon/2,1-\varepsilon/2)$ exclude the boundary
of the simplex.  We refer to $\mathcal{S}_C$ as the interior of the KL ball.

\medskip\noindent
\textbf{Step 2: Full support of Dirichlet perturbations.}
Given a current state $p\in\mathcal{S}_C$, the perturbing move draws
independent $\widetilde Y_{ab}\sim\Gamma(\alpha_{ab},1)$ and sets
$\widetilde p_{ab} = \widetilde Y_{ab}/\sum_{c,d}\widetilde Y_{cd}$.  The
vector $\widetilde p$ has a Dirichlet distribution with strictly positive
parameters $(\alpha_{ab})$, and hence a density with respect to Lebesgue
measure on $\mathcal{S}$ of the form
\[
  f_p(\widetilde p)
  \propto \prod_{a,b} \widetilde p_{ab}^{\alpha_{ab}-1},
  \qquad \widetilde p\in\mathcal{S}.
\]
This density is continuous and strictly positive on all of $\mathcal{S}$.
In particular, for any Borel set $B\subset\mathcal{S}$ with positive
Lebesgue measure,
\[
  \mathbb{P}_p(\widetilde p\in B) > 0.
\]

\medskip\noindent
\textbf{Step 3: Irreducibility on $\mathcal{S}_C$.}
Let $P$ and $P^\dagger$ be two step--graphons in the interior of the KL
ball, with associated block vectors $p,p^\dagger\in\mathcal{S}_C$, and let
$U$ be any open neighborhood of $p^\dagger$ contained in $\mathcal{S}_C$.
Because $\mathcal{S}_C$ is open, such a neighborhood exists and has
positive Lebesgue measure.  Since the Dirichlet density $f_p$ is strictly
positive on all of $\mathcal{S}$, we have
\[
  \mathbb{P}_p\bigl(\widetilde p\in U\bigr)>0.
\]
Whenever $\widetilde p\in U\subset\mathcal{S}_C$, the perturbing move is
accepted, so the one--step transition kernel satisfies
\[
  K(p,U)
  := \mathbb{P}_p(P_1\in U)
  \;\ge\;
  \mathbb{P}_p(\widetilde p\in U)
  \;>\; 0.
\]
Thus any open subset $U$ of $\mathcal{S}_C$ can be reached from any
starting point $p\in\mathcal{S}_C$ in a single step with positive
probability.  This implies $\psi$--irreducibility of the chain on
$\mathcal{S}_C$, with respect to Lebesgue measure restricted to
$\mathcal{S}_C$.

\medskip\noindent
\textbf{Step 4: Aperiodicity.}
For aperiodicity it is enough to show that there exists a nonnull set
$A\subset\mathcal{S}_C$ such that the chain has a positive probability of
remaining in $A$ in one step.  Fix $p\in\mathcal{S}_C$ and let
$B\subset\mathcal{S}\setminus\mathcal{S}_C$ be any measurable set of
positive Lebesgue measure contained in the complement of the KL ball.  Since
$f_p$ is strictly positive on $\mathcal{S}$, we have
$\mathbb{P}_p(\widetilde p\in B)>0$.  Whenever $\widetilde p\in B$ the move
is rejected and the chain stays at $p$.  Thus
\[
  K(p,\{p\})
  = \mathbb{P}_p(P_1=p)
  \;\ge\;
  \mathbb{P}_p(\widetilde p\in B)
  \;>\; 0.
\]
So every interior point $p\in\mathcal{S}_C$ has a self--transition
probability strictly larger than zero, which implies that the chain is
aperiodic on $\mathcal{S}_C$.

\medskip\noindent
\textbf{Step 5: Rescaling move and RDPGs.}
Including the rescaling move (Step~2 in the theorem statement) can only
increase the set of reachable points and does not affect irreducibility or
aperiodicity established above.\\

For random dot product graphs whose graphon $W$ can be approximated in cut
norm by $K$--block step graphons $\widetilde W_K$, we can apply the same
Markov chain to the block models $\widetilde W_K$.  Given any two RDPG
graphons $W,W^\dagger$ whose associated laws $G(W)$ and $G(W^\dagger)$ lie
in $\Gamma_C(G^\star)$, choose $K$ large enough that $W$ and $W^\dagger$ are
approximated in cut norm by step--graphons in $\mathcal{S}_C$ to within any
prescribed tolerance.  By the SBM case above, the chain on step--graphons
can reach an arbitrarily small neighborhood of the step approximation of
$W^\dagger$ starting from that of $W$ with positive probability; the
corresponding RDPG graphons are then reachable up to the chosen
approximation error.

This completes the proof.
\end{proof}

\begin{supptheorem}[Sharp small--KL expansion]\label{thm:sharp-small-KL}
Let $(\mathcal{X},\mathcal{A},P)$ be a probability space, and let
$f:\mathcal{X}\to\mathbb{R}$ be a measurable function with
\[
  m := \mathbb{E}_P[f],\qquad
  \sigma^2 := \Var_P(f) \in [0,\infty),
\]
such that the centred moment generating function
\[
  M(\lambda)
  := \mathbb{E}_P\bigl[e^{\lambda(f-m)}\bigr]
\]
is finite for all $\lambda$ in a neighborhood of $0$.
For $C>0$ define
\[
  \mathcal{U}_C(P)
  := \bigl\{Q:\ Q\ll P,\ \KL(Q\Vert P)\le C\bigr\},
  \qquad
  S_f(C)
  := \sup_{Q\in\mathcal{U}_C(P)} \int f\,\mathrm{d}Q,
\]
where
$\KL(Q\Vert P) := \int \log\!\bigl(\frac{\mathrm{d}Q}{\mathrm{d}P}\bigr)\,\mathrm{d}Q$.
Then, as $C\downarrow 0$,
\begin{equation}
  S_f(C)
  =
  m + \sqrt{2\,\Var_P(f)}\,\sqrt{C}
    + o\!\bigl(\sqrt{C}\bigr).
  \label{eq:sharp-small-KL-expansion}
\end{equation}

Moreover, if $\Var_P(f)>0$, the coefficient $\sqrt{2\,\Var_P(f)}$ is sharp in the sense that
\[
  \lim_{C\downarrow 0}
  \frac{S_f(C)-m}{\sqrt{2\,\Var_P(f)}\,\sqrt{C}}
  = 1.
\]
In particular, for any $k<\sqrt{2\,\Var_P(f)}$ there exists a sequence $C_j\downarrow 0$ such that
\[
  S_f(C_j)
  > m + k\,\sqrt{C_j}
\]
for all $j$ large enough.
\end{supptheorem}

\begin{proof}
If $f$ is almost surely constant under $P$, say $f\equiv m$, then
$S_f(C)=m$ for all $C\ge 0$ and \eqref{eq:sharp-small-KL-expansion}
holds trivially with $\sigma^2=0$.  Hence we assume $\sigma^2>0$.

Write $\tilde f:=f-m$, so that $\mathbb{E}_P[\tilde f]=0$ and
$\Var_P(\tilde f)=\sigma^2$.  Let
\[
  M(\lambda)
  := \mathbb{E}_P\bigl[e^{\lambda \tilde f}\bigr],
  \qquad
  \Lambda(\lambda)
  := \log M(\lambda),
\]
be the moment generating function and cumulant generating function of
$\tilde f$ under $P$.  By assumption, $M(\lambda)<\infty$ for $\lambda$
in a neighborhood of $0$, and $\Lambda$ is analytic there with Taylor
expansion
\[
  \Lambda(\lambda)
  = \tfrac{1}{2}\sigma^2\lambda^2 + \mathcal{O}(|\lambda|^3),
  \qquad \lambda\to 0.
\]

\medskip\noindent
\emph{Dual representation.}
For each $\lambda>0$ and any $Q\ll P$, the Donsker--Varadhan
variational inequality gives
\[
  \lambda\int f\,\mathrm{d}Q - \KL(Q\Vert P)
  \;\le\;
  \log \mathbb{E}_P\!\bigl[e^{\lambda f}\bigr].
\]
Rearranging,
\[
  \int f\,\mathrm{d}Q
  \;\le\;
  \frac{1}{\lambda}\Bigl\{
    \log\mathbb{E}_P\!\bigl[e^{\lambda f}\bigr]
    + \KL(Q\Vert P)
  \Bigr\}.
\]
If $\KL(Q\Vert P)\le C$, this yields
\[
  \int f\,\mathrm{d}Q
  \;\le\;
  m + \frac{1}{\lambda}\bigl\{\Lambda(\lambda) + C\bigr\},
\]
because
$\log\mathbb{E}_P[e^{\lambda f}]
 = \lambda m + \Lambda(\lambda)$.
Taking the supremum over $Q\in\mathcal{U}_C(P)$ and then the infimum
over $\lambda>0$ shows that
\begin{equation}
  S_f(C)
  \;\le\;
  m + \inf_{\lambda>0} g(\lambda,C),
  \qquad
  g(\lambda,C)
  := \frac{\Lambda(\lambda)+C}{\lambda}.
  \label{eq:Sf-upper-dual}
\end{equation}

\medskip\noindent
\emph{Exponential tilting and lower bound.}
For the lower bound we consider the exponentially tilted measures
\[
  \frac{\mathrm{d}Q_\lambda}{\mathrm{d}P}
  := \frac{e^{\lambda\tilde f}}{M(\lambda)},\qquad \lambda>0.
\]
Then $Q_\lambda\ll P$, and standard properties of exponential tilting
give
\[
  \mathbb{E}_{Q_\lambda}[\tilde f]
  = \Lambda'(\lambda),
  \qquad
  \KL(Q_\lambda\Vert P)
  = \lambda\Lambda'(\lambda) - \Lambda(\lambda).
\]
In particular,
\begin{equation}
  \int f\,\mathrm{d}Q_\lambda
  = m + \Lambda'(\lambda).
  \label{eq:Sf-lower-tilt-mean}
\end{equation}

\medskip\noindent
\emph{Asymptotics of the tilt.}
From the Taylor expansion of $\Lambda$ we obtain
\[
  \Lambda(\lambda)
  = \tfrac{1}{2}\sigma^2\lambda^2 + \mathcal{O}(\lambda^3),
  \qquad
  \Lambda'(\lambda)
  = \sigma^2\lambda + \mathcal{O}(\lambda^2),
  \qquad \lambda\to 0.
\]
Hence
\begin{equation}
  C(\lambda)
  := \KL(Q_\lambda\Vert P)
  = \lambda\Lambda'(\lambda) - \Lambda(\lambda)
  = \tfrac{1}{2}\sigma^2\lambda^2 + \mathcal{O}(\lambda^3),
  \quad \lambda\to 0,
  \label{eq:C-lambda-expansion}
\end{equation}
and
\begin{equation}
  \int f\,\mathrm{d}Q_\lambda
  = m + \sigma^2\lambda + \mathcal{O}(\lambda^2),
  \qquad \lambda\to 0.
  \label{eq:EQlambda-expansion}
\end{equation}
Since $\sigma^2>0$ and $\Lambda''(\lambda) = \Var_{Q_\lambda}(\tilde f)>0$
for all sufficiently small $\lambda>0$, the function
$\lambda\mapsto C(\lambda)$ is strictly increasing on $(0,\lambda_0)$
for some $\lambda_0>0$, with $C(\lambda)\downarrow 0$ as
$\lambda\downarrow 0$.  Therefore it is invertible on this interval,
and we may define for small $C>0$ the inverse
$\lambda(C)\in(0,\lambda_0)$ such that $C(\lambda(C))=C$.\\

From \eqref{eq:C-lambda-expansion} we find
\[
  C
  = \tfrac{1}{2}\sigma^2\lambda(C)^2 + \mathcal{O}(\lambda(C)^3)
  \quad\Rightarrow\quad
  \lambda(C)
  = \frac{\sqrt{2C}}{\sigma} + \mathcal{O}(C),
  \qquad C\downarrow 0,
\]
using $\lambda(C)=\mathcal{O}(\sqrt{C})$.
Substituting into \eqref{eq:EQlambda-expansion} gives
\begin{equation}
  \int f\,\mathrm{d}Q_{\lambda(C)}
  = m + \sqrt{2\sigma^2 C} + \mathcal{O}(C),
  \qquad C\downarrow 0.
  \label{eq:lower-bound-asymptotic}
\end{equation}

\medskip\noindent
\emph{Lower bound on $S_f(C)$.}
For each small $C>0$, $Q_{\lambda(C)}\in\mathcal{U}_C(P)$ by
construction, so
\[
  S_f(C)
  \;\ge\;
  \int f\,\mathrm{d}Q_{\lambda(C)}
  = m + \sqrt{2\sigma^2 C} + \mathcal{O}(C),
\]
which implies
\begin{equation}
  \liminf_{C\downarrow 0}
  \frac{S_f(C)-m}{\sqrt{C}}
  \;\ge\; \sqrt{2\sigma^2}.
  \label{eq:lower-liminf}
\end{equation}

\medskip\noindent
\emph{Upper bound on $S_f(C)$.}
Using \eqref{eq:Sf-upper-dual} we have, for any $\lambda>0$,
\[
  S_f(C)-m
  \;\le\;
  g(\lambda,C)
  = \frac{\Lambda(\lambda)+C}{\lambda}.
\]
Take $\lambda = t\sqrt{C}$ with $t>0$ fixed and $C$ small.  Then, as
$C\downarrow 0$,
\[
  \Lambda(\lambda)
  = \tfrac{1}{2}\sigma^2 t^2 C + \mathcal{O}(C^{3/2}),
\]
and hence
\[
  g(t\sqrt{C},C)
  = \frac{\tfrac{1}{2}\sigma^2 t^2 C + C + \mathcal{O}(C^{3/2})}
         {t\sqrt{C}}
  = \Bigl(\tfrac{1}{2}\sigma^2 t + \frac{1}{t}\Bigr)\sqrt{C}
    + \mathcal{O}(C).
\]
Let
\[
  F(t) := \tfrac{1}{2}\sigma^2 t + \frac{1}{t},\qquad t>0.
\]
A simple calculus check shows that $F$ attains its unique minimum at
$t_\star = \sqrt{2}/\sigma$, with
\[
  F(t_\star)
  = \sqrt{2\sigma^2}.
\]
Therefore, for any $\varepsilon>0$ we can choose $t_\varepsilon$ close
enough to $t_\star$ such that
$F(t_\varepsilon)\le\sqrt{2\sigma^2}+\varepsilon$.
For this choice,
\[
  g(t_\varepsilon\sqrt{C},C)
  = F(t_\varepsilon)\sqrt{C} + \mathcal{O}(C)
  \le \bigl(\sqrt{2\sigma^2}+\varepsilon\bigr)\sqrt{C}
       + \mathcal{O}(C).
\]
Taking the infimum over $\lambda>0$ and then letting $C\downarrow 0$
gives
\begin{equation}
  \limsup_{C\downarrow 0}
  \frac{S_f(C)-m}{\sqrt{C}}
  \;\le\; \sqrt{2\sigma^2} + \varepsilon.
\end{equation}
Since $\varepsilon>0$ is arbitrary,
\begin{equation}
  \limsup_{C\downarrow 0}
  \frac{S_f(C)-m}{\sqrt{C}}
  \;\le\; \sqrt{2\sigma^2}.
  \label{eq:upper-limsup}
\end{equation}

\medskip\noindent
Combining \eqref{eq:lower-liminf} and \eqref{eq:upper-limsup} yields
\[
  \lim_{C\downarrow 0}
  \frac{S_f(C)-m}{\sqrt{C}}
  = \sqrt{2\sigma^2},
\]
which is exactly \eqref{eq:sharp-small-KL-expansion}.  The sharpness
statement follows immediately from the existence of this limit: if
$k<\sqrt{2\sigma^2}$, then for all sufficiently small $C$ we must have
$(S_f(C)-m)/\sqrt{C} > k$, and hence $S_f(C)>m+k\sqrt{C}$.

This completes the proof.
\end{proof}

\subsection*{Proof of Theorem~\ref{thm:WH-critical-robustness}}

\begin{proof}[Proof of Theorem~\ref{thm:WH-critical-robustness}]
Write $R_0:=R(\theta_0)$ and $\Delta_0:=\Delta(\theta_0)$.  By
Assumption~\ref{ass:critical-R} and the chain rule,
\[
  \nabla_\theta R(\theta_0)
  = c_0\,\Delta_0^{-2}\,\nabla_\theta\rho(\theta_0)
    + \nabla_\theta H(\theta_0).
\]
By item~(i) in the theorem assumptions, we have
\[
  G := \nabla_\theta R(\theta_0),
  \qquad
  \|G\|\asymp\Delta_0^{-2}
  \quad\text{as }\Delta_0\downarrow 0.
\]

By Assumption~\ref{ass:local-BvM}, under $\Pi_{0,n}$ we may write
\[
  \theta
  = \theta_0 + r_n Z_n,
  \qquad
  Z_n\Rightarrow Z\sim N(0,\Sigma)
\]
in $P_{\theta_0}^{(n)}$--probability, and item~(ii) ensures that
$\|Z_n\|$ has uniformly bounded second and fourth moments (again in
$P_{\theta_0}^{(n)}$--probability).  A Taylor expansion of $R$ around
$\theta_0$ on a neighborhood where $R$ is $C^2$ yields
\[
  R(\theta)
  = R_0 + r_n G^\top Z_n + R_n^{\mathrm{rem}},
\]
with
\[
  R_n^{\mathrm{rem}} = O\bigl(r_n^2\|Z_n\|^2\bigr)
\]
uniformly on that neighborhood.

\medskip\noindent
\emph{Baseline posterior risk.}
The Bayes estimator is
\[
  a_n^\star
  = \int R(\theta)\,\Pi_{0,n}(\mathrm{d}\theta)
  = R_0 + r_n G^\top\mathbb{E}_{\Pi_{0,n}}[Z_n]
    + \mathbb{E}_{\Pi_{0,n}}\bigl[R_n^{\mathrm{rem}}\bigr].
\]
By the bound on $R_n^{\mathrm{rem}}$ and the uniform moment bounds on
$Z_n$,
\[
  \mathbb{E}_{\Pi_{0,n}}\bigl[R_n^{\mathrm{rem}}\bigr]
  = O_{P_{\theta_0}^{(n)}}\bigl(r_n^2\bigr)
  = o_{P_{\theta_0}^{(n)}}(r_n),
\]
so
\[
  a_n^\star
  = R_0 + r_n G^\top\mathbb{E}_{\Pi_{0,n}}[Z_n]
    + o_{P_{\theta_0}^{(n)}}(r_n).
\]
Hence the centred fluctuation of $R(\theta)$ under $\Pi_{0,n}$ can be
written as
\[
  R(\theta)-a_n^\star
  = r_n G^\top\bigl(Z_n-\mathbb{E}_{\Pi_{0,n}}[Z_n]\bigr)
    + \bigl(R_n^{\mathrm{rem}}
           -\mathbb{E}_{\Pi_{0,n}}[R_n^{\mathrm{rem}}]\bigr)
    + o_{P_{\theta_0}^{(n)}}(r_n).
\]
The remainder term satisfies
\[
  R_n^{\mathrm{rem}}
  -\mathbb{E}_{\Pi_{0,n}}[R_n^{\mathrm{rem}}]
  = O_{P_{\theta_0}^{(n)}}\bigl(r_n^2\|Z_n\|^2\bigr),
\]
so that
\[
  \Var_{\Pi_{0,n}}\bigl(R_n^{\mathrm{rem}}\bigr)
  = O_{P_{\theta_0}^{(n)}}\bigl(r_n^4\bigr)
  = o_{P_{\theta_0}^{(n)}}\bigl(r_n^2\|G\|^2\bigr),
\]
because $\|G\|\asymp\Delta_0^{-2}\to\infty$ and $r_n\to 0$,
$\Delta_0\to 0$.  Denoting
\[
  Z_n^0 := Z_n - \mathbb{E}_{\Pi_{0,n}}[Z_n],
\]
we therefore obtain
\[
  R(\theta)-a_n^\star
  = r_n G^\top Z_n^0
    + o_{P_{\theta_0}^{(n)}}\bigl(r_n\|G\|\bigr),
\]
and hence
\[
  \rho_{0,n}
  = \Var_{\Pi_{0,n}}(R(\theta))
  = r_n^2 G^\top \Sigma_n G
    + o_{P_{\theta_0}^{(n)}}\bigl(r_n^2\|G\|^2\bigr),
\]
where $\Sigma_n:=\Var_{\Pi_{0,n}}(Z_n)\to\Sigma$ in
$P_{\theta_0}^{(n)}$--probability by the local BvM assumption and the
moment bounds.

Write
\[
  W_0 := G^\top\Sigma G.
\]
By the nondegeneracy of $\Sigma$ there exist constants
$0<c_1\le c_2<\infty$ such that
\[
  c_1\|G\|^2 \le W_0 \le c_2\|G\|^2.
\]
Together with $\|G\|\asymp\Delta_0^{-2}$ this implies
$W_0\asymp\Delta_0^{-4}$, so there exists a constant $V_0\in(0,\infty)$
(depending on the local geometry along the path $\theta_0$) such that
\[
  W_0
  = \frac{V_0}{\Delta_0^4}\,\bigl(1+o(1)\bigr)
  \quad\text{as }\Delta_0\downarrow 0.
\]
Combining this with the approximation for $\rho_{0,n}$ yields
\[
  \rho_{0,n}
  =
  \frac{V_0\,r_n^2}{\Delta_0^4}\,
  \bigl(1+o_{P_{\theta_0}^{(n)}}(1)\bigr),
\]
which is \eqref{eq:WH-R-critical-rho0-main}.

\medskip\noindent
\emph{Law of the squared loss.}
Define
\[
  L_n(\theta)
  := \bigl(a_n^\star-R(\theta)\bigr)^2.
\]
By definition of $a_n^\star$,
\[
  \rho_{0,n}
  = \Var_{\Pi_{0,n}}(R(\theta))
  = \mathbb{E}_{\Pi_{0,n}}\bigl[L_n(\theta)\bigr].
\]

From the expansion above we have, under $\Pi_{0,n}$,
\[
  a_n^\star-R(\theta)
  = -r_n G^\top Z_n^0
    + o_{P_{\theta_0}^{(n)}}\bigl(r_n\|G\|\bigr),
\]
so
\[
  L_n(\theta)
  = \bigl(a_n^\star-R(\theta)\bigr)^2
  = r_n^2 \bigl(G^\top Z_n^0\bigr)^2
    + o_{P_{\theta_0}^{(n)}}\bigl(r_n^2\|G\|^2\bigr).
\]
Let $\tau_n^2:=\rho_{0,n}$ and define the normalized loss
\[
  Y_n
  := -\tau_n^{-1}\bigl(R(\theta)-a_n^\star\bigr),
  \qquad
  L_n(\theta) = \tau_n^2 Y_n^2.
\]
By the local BvM assumption and the uniform moment bounds,
$G^\top Z_n^0$ is asymptotically normal with variance $W_0$, and the
remainder is negligible at scale $\tau_n\asymp r_n\|G\|$.  A
delta--method / continuous mapping argument thus gives
\[
  Y_n \Rightarrow Y\sim N(0,1)
\]
under $\Pi_{0,n}$ in $P_{\theta_0}^{(n)}$--probability, and the moment
bounds upgrade this weak convergence to convergence of moments up to
order~$4$.  In particular,
\[
  \mathbb{E}_{\Pi_{0,n}}[Y_n^2] = 1
  \quad\text{and}\quad
  \Var_{\Pi_{0,n}}(Y_n^2)
  \to \Var(Y^2) = 2
\]
in $P_{\theta_0}^{(n)}$--probability.  Since $L_n=\tau_n^2 Y_n^2$ and
$\tau_n^2=\rho_{0,n}$, we obtain
\[
  \mathbb{E}_{\Pi_{0,n}}[L_n]
  = \rho_{0,n}
  \quad\text{and}\quad
  \Var_{\Pi_{0,n}}(L_n)
  = 2\rho_{0,n}^2\bigl(1+o_{P_{\theta_0}^{(n)}}(1)\bigr).
\]

Moreover, for each fixed $n$ the (centered) moment generating function of
$L_n$ under $\Pi_{0,n}$ is finite in a neighborhood of the origin; this
holds, for example, if $R$ has at most polynomial growth and
$\Pi_{0,n}$ has sub-Gaussian tails locally around $\theta_0$.
Together with item~(iii) in the theorem assumptions, this implies that
the assumptions of Theorem~\ref{thm:sharp-small-KL} are satisfied with
$P=\Pi_{0,n}$ and $f=L_n$, and that the $o(\sqrt C)$ remainder in that
theorem can be taken uniformly over the family of normalized losses
$L_n/\rho_{0,n}$ for small $C$.

\medskip\noindent
\emph{Apply the sharp small--KL expansion.}
By Theorem~\ref{thm:sharp-small-KL}, for each fixed $n$ and all
sufficiently small $C>0$,
\[
  \sup_{\widetilde\Pi:\,\KL(\widetilde\Pi\Vert\Pi_{0,n})\le C}
    \int L_n(\theta)\,\widetilde\Pi(\mathrm{d}\theta)
  =
  \mathbb{E}_{\Pi_{0,n}}[L_n]
  + \sqrt{2\,\Var_{\Pi_{0,n}}(L_n)}\,\sqrt{C}
  + o(\sqrt{C}),
\]
where the $o(\sqrt{C})$ term tends to $0$ as $C\downarrow 0$ and, by
the uniform exponential--moment bounds in item~(iii), can be chosen
uniformly in $n$ for $C$ in a sufficiently small interval $(0,C_0]$.

By definition,
\[
  \rho_{\mathrm{rob},n}(C)
  :=
  \sup_{\widetilde\Pi:\,\KL(\widetilde\Pi\Vert\Pi_{0,n})\le C}
    \int L_n(\theta)\,\widetilde\Pi(\mathrm{d}\theta),
\]
so substituting
$\mathbb{E}_{\Pi_{0,n}}[L_n]=\rho_{0,n}$ and
$\Var_{\Pi_{0,n}}(L_n)=2\rho_{0,n}^2(1+o_{P_{\theta_0}^{(n)}}(1))$ gives,
for all sufficiently small $C$,
\[
  \rho_{\mathrm{rob},n}(C)
  =
  \rho_{0,n}
  + 2\,\rho_{0,n}\sqrt{C}
  + o_{P_{\theta_0}^{(n)}}\bigl(\rho_{0,n}\sqrt{C}\bigr),
\]
where the $o_{P_{\theta_0}^{(n)}}(\rho_{0,n}\sqrt{C})$ term is uniform
in $n$ for $C\in(0,C_0]$.  Taking $C=\mathcal{C}_n\downarrow 0$ yields
\eqref{eq:WH-R-critical-rob-main} and the convergence
\[
  \frac{
    \rho_{\mathrm{rob},n}(\mathcal{C}_n)-\rho_{0,n}
  }{
    \rho_{0,n}\sqrt{\mathcal{C}_n}
  }
  \xrightarrow[n\to\infty]{P_{\theta_0}^{(n)}} 2,
\]
which proves the sharp inflation statement (part~(2)).

\medskip\noindent
\emph{Sharpness.}
From Theorem~\ref{thm:sharp-small-KL} applied to $f=L_n$ we have, for
each fixed $n$,
\[
  \rho_{\mathrm{rob},n}(C)
  =
  \rho_{0,n}
  + \sqrt{2\,\Var_{\Pi_{0,n}}(L_n)}\,\sqrt{C}
  + o\bigl(\sqrt{C}\bigr)
  \quad\text{as }C\downarrow 0.
\]
Dividing by $\rho_{0,n}\sqrt{C}$ and using
$\Var_{\Pi_{0,n}}(L_n)=2\rho_{0,n}^2(1+o_{P_{\theta_0}^{(n)}}(1))$ we
obtain
\[
  \frac{
    \rho_{\mathrm{rob},n}(C)-\rho_{0,n}
  }{
    \rho_{0,n}\sqrt{C}
  }
  =
  \frac{\sqrt{2\,\Var_{\Pi_{0,n}}(L_n)}}{\rho_{0,n}}
  + o_{P_{\theta_0}^{(n)}}(1)
  \xrightarrow[n\to\infty]{P_{\theta_0}^{(n)}} 2
\]
for every fixed $C>0$ small enough.  Equivalently, for any $\varepsilon>0$
there exist $C_0>0$ and $n_0$ such that, for all $n\ge n_0$ and all
$C\in(0,C_0]$,
\[
  P_{\theta_0}^{(n)}
  \left(
    \frac{
      \rho_{\mathrm{rob},n}(C)-\rho_{0,n}
    }{
      \rho_{0,n}\sqrt{C}
    }
    \ge 2-\varepsilon
  \right)
  \to 1.
\]
Now fix $k<2$ and choose $\varepsilon\in(0,2-k)$.  By the previous
display, there exist $C_0>0$ and $n_0$ such that, for all $n\ge n_0$ and
all $C\in(0,C_0]$, the event
\[
  \rho_{\mathrm{rob},n}(C)
  >
  \rho_{0,n} + k\,\rho_{0,n}\sqrt{C}
\]
has $P_{\theta_0}^{(n)}$--probability tending to $1$.  In particular,
we may pick any deterministic sequence $\mathcal{C}_n\downarrow 0$ with
$\mathcal{C}_n\le C_0$ for all $n$; for that sequence we obtain
\[
  \rho_{\mathrm{rob},n}(\mathcal{C}_n)
  >
  \rho_{0,n} + k\,\rho_{0,n}\sqrt{\mathcal{C}_n}
\]
for all sufficiently large $n$ with $P_{\theta_0}^{(n)}$--probability
tending to~$1$.  This proves the sharpness claim in part~(3), and shows
that the coefficient $2$ in \eqref{eq:WH-R-critical-rob-main} is
asymptotically optimal.

This completes the proof.
\end{proof}

\subsection*{Proof of Lemma~\ref{lem:ER-SBM-KL}}
\label{subsec:ER-SBM-labelled-appendix}

\begin{proof}
Under both $P_0^{(n)}$ and $P_1^{(n)}$, edges are independent; only the
Bernoulli parameters differ. Hence
\[
  D_n
  = \KL\bigl(P_1^{(n)}\Vert P_0^{(n)}\bigr)
  = \sum_{1\le i<j\le n}
      \KL\bigl(
        \mathrm{Bern}(p_{1,ij})\Vert\mathrm{Bern}(p_n)
      \bigr),
\]
where $p_{1,ij}=p_n^{\mathrm{in}}$ if $\sigma(i)=\sigma(j)$ and
$p_{1,ij}=p_n^{\mathrm{out}}$ otherwise, with $\sigma$ the community
assignment. Grouping within-- and between--block pairs gives
\[
  D_n
  =
  N_n^{\mathrm{in}}\,
    {\rm KL}\bigl(\mathrm{Bern}(p_n^{\mathrm{in}})\Vert\mathrm{Bern}(p_n)\bigr)
  +
  N_n^{\mathrm{out}}\,
    {\rm KL}\bigl(\mathrm{Bern}(p_n^{\mathrm{out}})\Vert\mathrm{Bern}(p_n)\bigr).
\]

For the asymptotics, write $r_n:=p_n=c/n$ and, for a generic
$\delta\in\{\lambda,-\lambda\}$,
\[
  q_n := r_n + \frac{\delta}{n}
  = \frac{c+\delta}{n}.
\]
For a single edge with parameter $q_n$ under $P_1^{(n)}$ and $r_n$ under
$P_0^{(n)}$, the KL divergence is
\[
  K_n(\delta)
  := \KL\bigl(\mathrm{Bern}(q_n)\Vert\mathrm{Bern}(r_n)\bigr)
  = q_n\log\frac{q_n}{r_n}
    + (1-q_n)\log\frac{1-q_n}{1-r_n}.
\]

The first term is
\[
  q_n\log\frac{q_n}{r_n}
  = \frac{c+\delta}{n}
    \log\frac{(c+\delta)/n}{c/n}
  = \frac{c+\delta}{n}\log\frac{c+\delta}{c}.
\]
For the second term, use $\log(1-x)=-x-x^2/2+O(x^3)$ as $x\to 0$:
\[
  \log(1-q_n) = -\frac{c+\delta}{n} - \frac{(c+\delta)^2}{2n^2}
                + O\!\left(\frac{1}{n^3}\right),
  \qquad
  \log(1-r_n) = -\frac{c}{n} - \frac{c^2}{2n^2}
                + O\!\left(\frac{1}{n^3}\right),
\]
so
\[
  \log\frac{1-q_n}{1-r_n}
  = -\frac{\delta}{n}
    - \frac{(c+\delta)^2-c^2}{2n^2}
    + O\!\left(\frac{1}{n^3}\right)
  = -\frac{\delta}{n} + O\!\left(\frac{1}{n^2}\right).
\]
Multiplying by $1-q_n = 1 + O(1/n)$ gives
\[
  (1-q_n)\log\frac{1-q_n}{1-r_n}
  = -\frac{\delta}{n} + O\!\left(\frac{1}{n^2}\right).
\]
Hence
\[
  K_n(\delta)
  = \frac{c+\delta}{n}\log\frac{c+\delta}{c}
    - \frac{\delta}{n}
    + O\!\left(\frac{1}{n^2}\right),
\]
and thus
\[
  n\,K_n(\delta)
  = (c+\delta)\log\frac{c+\delta}{c}
    - \delta
    + O\!\left(\frac{1}{n}\right).
\]

Now
\[
  D_n
  = N_n^{\mathrm{in}} K_n(\lambda) + N_n^{\mathrm{out}} K_n(-\lambda),
\]
so
\[
  \frac{D_n}{n}
  = \frac{N_n^{\mathrm{in}}}{n}K_n(\lambda)
    + \frac{N_n^{\mathrm{out}}}{n}K_n(-\lambda).
\]
In the symmetric two--block SBM with equal block sizes,
\[
  N_n^{\mathrm{in}} = \frac{n^2}{4} + O(n),
  \qquad
  N_n^{\mathrm{out}} = \frac{n^2}{4} + O(n),
\]
so
\[
  \frac{N_n^{\mathrm{in}}}{n} = \frac{n}{4} + O(1),
  \qquad
  \frac{N_n^{\mathrm{out}}}{n} = \frac{n}{4} + O(1).
\]
Therefore
\begin{align*}
  \frac{D_n}{n}
  &= \left(\frac{n}{4}+O(1)\right)K_n(\lambda)
     + \left(\frac{n}{4}+O(1)\right)K_n(-\lambda)\\
  &= \frac{1}{4}\Bigl\{nK_n(\lambda) + nK_n(-\lambda)\Bigr\}
     + O\bigl(K_n(\lambda)+K_n(-\lambda)\bigr).
\end{align*}
Using $K_n(\delta)=O(1/n)$ and the expansion for $nK_n(\delta)$ above,
we obtain
\begin{align*}
  \frac{D_n}{n}
  &= \frac{1}{4}\Bigl[
       (c+\lambda)\log\frac{c+\lambda}{c}
       - \lambda
       + (c-\lambda)\log\frac{c-\lambda}{c}
       + \lambda
     \Bigr]
     + O\!\left(\frac{1}{n}\right)\\
  &= \frac{1}{4}\Bigl[
       (c+\lambda)\log\frac{c+\lambda}{c}
       + (c-\lambda)\log\frac{c-\lambda}{c}
     \Bigr]
     + O\!\left(\frac{1}{n}\right).
\end{align*}
This shows that
\[
  \frac{D_n}{n} \xrightarrow[n\to\infty]{} I(\lambda)
  := \frac{1}{4}\Bigl[
       (c+\lambda)\log\frac{c+\lambda}{c}
       + (c-\lambda)\log\frac{c-\lambda}{c}
     \Bigr],
\]
and hence $D_n = I(\lambda)\,n + o(n)$.

Finally, a Taylor expansion of $I(\lambda)$ in $\lambda$ around $0$ gives
\[
  I(\lambda)
  = \frac{\lambda^2}{4c}
    + O\!\left(\frac{\lambda^4}{c^3}\right),
  \qquad \lambda\to 0,
\]
as claimed.
\end{proof}
\subsection*{Proof of Lemma~\ref{lem:ER-SBM-Chernoff}}

\begin{proof}
Because edges are independent under both $P_0^{(n)}$ and $P_1^{(n)}$,
\[
  \sum_A P_0^{(n)}(A)^{1-t}P_1^{(n)}(A)^t
  = \prod_{e}
      \sum_{a\in\{0,1\}}
        P_{0,e}(a)^{1-t}P_{1,e}(a)^t,
\]
where $P_{0,e}=\mathrm{Bern}(p_n)$ and
$P_{1,e}=\mathrm{Bern}(p_n^{\mathrm{in}})$ or
$\mathrm{Bern}(p_n^{\mathrm{out}})$ according to whether $e$ is a within-- or
between--block edge.  Thus
\[
  \mathcal{C}_n
  = \sup_{0\le t\le 1}
      \Bigl(
        N_n^{\mathrm{in}}\phi_n^{+}(t)
        +
        N_n^{\mathrm{out}}\phi_n^{-}(t)
      \Bigr),
\]
where, for $r,q\in(0,1)$,
\[
  \phi(r,q;t)
  := -\log\bigl(
        r^{1-t}q^t + (1-r)^{1-t}(1-q)^t
      \bigr),
\]
and
\[
  \phi_n^{+}(t) := \phi\bigl(p_n,p_n^{\mathrm{in}};t\bigr),
  \qquad
  \phi_n^{-}(t) := \phi\bigl(p_n,p_n^{\mathrm{out}};t\bigr).
\]

Set $u:=1/n$ and note $p_n = cu$, $p_n^{\mathrm{in}}=(c+\lambda)u$,
$p_n^{\mathrm{out}}=(c-\lambda)u$.  For fixed $c,\lambda,t$, expand
$\phi_n^{\pm}(t)$ as $u\to 0$.  Write
\[
  r := cu,
  \qquad
  q_\pm := (c\pm\lambda)u,
\]
so
\[
  \phi(r,q_\pm;t)
  = -\log\Bigl(
        r^{1-t}q_\pm^t
        + (1-r)^{1-t}(1-q_\pm)^t
      \Bigr).
\]
Let
\[
  S_\pm
  := r^{1-t}q_\pm^t + (1-r)^{1-t}(1-q_\pm)^t.
\]
Factor the second term:
\[
  S_\pm
  = (1-r)^{1-t}(1-q_\pm)^t\,
    \Bigl[
      1 + \frac{r^{1-t}q_\pm^t}{(1-r)^{1-t}(1-q_\pm)^t}
    \Bigr],
\]
so
\[
  \phi(r,q_\pm;t)
  = -(1-t)\log(1-r) - t\log(1-q_\pm)
    - \log\Bigl(1 + \frac{r^{1-t}q_\pm^t}{(1-r)^{1-t}(1-q_\pm)^t}\Bigr).
\]

Using $\log(1-x)=-x-x^2/2+O(x^3)$ and $r,q_\pm=O(u)$,
\[
  -(1-t)\log(1-r) - t\log(1-q_\pm)
  = (1-t)r + tq_\pm + O(u^2)
  = uc + ut(\pm\lambda) + O(u^2).
\]
Moreover,
\[
  r^{1-t}q_\pm^t
  = c^{1-t}(c\pm\lambda)^t u + O(u^2),
\]
and $(1-r)^{1-t}(1-q_\pm)^t = 1 + O(u)$, so
\[
  \frac{r^{1-t}q_\pm^t}{(1-r)^{1-t}(1-q_\pm)^t}
  = c^{1-t}(c\pm\lambda)^t u + O(u^2).
\]
Therefore
\[
  -\log\Bigl(1 + \frac{r^{1-t}q_\pm^t}{(1-r)^{1-t}(1-q_\pm)^t}\Bigr)
  = -c^{1-t}(c\pm\lambda)^t u + O(u^2).
\]

Combining,
\[
  \phi_n^{\pm}(t)
  = u\Bigl[
        c \pm t\lambda - c^{1-t}(c\pm\lambda)^t
      \Bigr]
    + O(u^2),
\]
hence
\[
  n\,\phi_n^{\pm}(t)
  = c \pm t\lambda - c^{1-t}(c\pm\lambda)^t
    + O\!\left(\frac{1}{n}\right).
\]

Using the block counts above,
\[
  \frac{N_n^{\mathrm{in}}}{n}
  = \frac{n-2}{4}
  = \frac{n}{4}+O(1),
  \qquad
  \frac{N_n^{\mathrm{out}}}{n}
  = \frac{n}{4},
\]
so
\[
  \frac{\mathcal{C}_n}{n}
  = \sup_{0\le t\le 1}
      \left\{
        \frac{N_n^{\mathrm{in}}}{n}\phi_n^{+}(t)
        + \frac{N_n^{\mathrm{out}}}{n}\phi_n^{-}(t)
      \right\}
  = \sup_{0\le t\le 1}
      \left\{
        \frac{1}{4}\bigl[n\phi_n^{+}(t) + n\phi_n^{-}(t)\bigr]
        + O\!\left(\frac{1}{n}\right)
      \right\}.
\]
Using the expansions for $n\phi_n^{\pm}(t)$ and letting $n\to\infty$,
\begin{align*}
  \lim_{n\to\infty}\frac{\mathcal{C}_n}{n}
  &= \sup_{0\le t\le 1}
     \frac{1}{4}\Bigl[
       \bigl(c + t\lambda - c^{1-t}(c+\lambda)^t\bigr)
       + \bigl(c - t\lambda - c^{1-t}(c-\lambda)^t\bigr)
     \Bigr]\\
  &= \sup_{0\le t\le 1}
     \frac{1}{4}\Bigl[
       2c - c^{1-t}\bigl((c+\lambda)^t + (c-\lambda)^t\bigr)
     \Bigr]
  = J(\lambda),
\end{align*}
which proves the asserted limit.

For the small--signal expansion, expand $(c\pm\lambda)^t$ in $\lambda$:
\[
  (c\pm\lambda)^t
  = c^t\Bigl(
        1 \pm t\frac{\lambda}{c}
        + \frac{t(t-1)}{2}\frac{\lambda^2}{c^2}
        + O\!\left(\frac{\lambda^3}{c^3}\right)
      \Bigr),
  \qquad \lambda\to 0.
\]
Summing,
\[
  (c+\lambda)^t + (c-\lambda)^t
  = 2c^t\Bigl(
      1 + \frac{t(t-1)}{2}\frac{\lambda^2}{c^2}
      + O\!\left(\frac{\lambda^4}{c^4}\right)
    \Bigr),
\]
so
\[
  c^{1-t}\bigl[(c+\lambda)^t + (c-\lambda)^t\bigr]
  = 2c\Bigl(
      1 + \frac{t(t-1)}{2}\frac{\lambda^2}{c^2}
      + O\!\left(\frac{\lambda^4}{c^4}\right)
    \Bigr).
\]
Therefore
\[
  J(\lambda;t)
  := \frac{1}{4}\Bigl\{
        2c - c^{1-t}\bigl[(c+\lambda)^t + (c-\lambda)^t\bigr]
      \Bigr\}
  = -\frac{t(t-1)}{4c}\,\lambda^2
    + O\!\left(\frac{\lambda^4}{c^3}\right),
\]
uniformly for $t$ in compact subsets of $(0,1)$.  Since
$-t(t-1)\ge 0$ on $[0,1]$ and attains its maximum $1/4$ at $t=1/2$,
\[
  J(\lambda)
  = \sup_{0\le t\le 1} J(\lambda;t)
  = \frac{\lambda^2}{16c}
    + O\!\left(\frac{\lambda^4}{c^3}\right),
  \qquad \lambda\to 0.
\]
This also yields the stated asymptotic relation
$J(\lambda) = \tfrac14 I(\lambda) + O(\lambda^4/c^3)$.
\end{proof}

\subsection*{Proof of Lemma~\ref{lem:graphon-information}}

\begin{proof}
Under the labelled SBM with parameters $(\pi,P_0)$ or $(\pi,P_\lambda)$ we may realise
$G_n$ as
\[
  Z_1,\dots,Z_n \overset{\text{i.i.d.}}{\sim}\pi,
  \qquad
  A_{ij}\mid Z_{1:n} \sim \mathrm{Bern}\bigl(P_\lambda(Z_i,Z_j)\bigr),
  \quad 1\le i<j\le n,
\]
independently across unordered pairs $(i,j)$.  The resulting law on graphs is
$\mathbb{P}^{(n)}_\lambda$, and similarly $\mathbb{P}^{(n)}_0$ for $(\pi,P_0)$.

For the step--graphons $W_0,W_\lambda$, the graphon sampling scheme is
\[
  U_1,\dots,U_n \overset{\text{i.i.d.}}{\sim}\mathrm{Unif}[0,1],
  \qquad
  A_{ij}\mid U_{1:n} \sim \mathrm{Bern}\bigl(W_\lambda(U_i,U_j)\bigr),
  \quad 1\le i<j\le n,
\]
again independently across edges.  Partition $[0,1]$ into $K$ subintervals of
lengths $\pi_k$ and define $Z_i$ to be the block index of $U_i$, i.e.\ the unique
$k$ such that $U_i$ lies in block $k$.  Then
$Z_1,\dots,Z_n\overset{\text{i.i.d.}}{\sim}\pi$, and
\[
  W_\lambda(U_i,U_j) = P_\lambda(Z_i,Z_j)
  \qquad\text{for all }1\le i<j\le n.
\]
Thus, conditional on $(Z_i)_{i=1}^n$, the adjacency matrix under the graphon model
has the same distribution as under the labelled SBM, and hence the marginal laws on
$G_n$ coincide:
\[
  \widetilde{\mathbb{P}}^{(n)}_{W_\lambda}
  = \mathbb{P}^{(n)}_\lambda,
  \qquad
  \widetilde{\mathbb{P}}^{(n)}_{W_0}
  = \mathbb{P}^{(n)}_0.
\]

Since the pairs of laws coincide exactly at finite $n$, any functional of the pair
is identical in both representations.  In particular,
\[
  D\bigl(\mathbb{P}^{(n)}_\lambda\big\Vert\mathbb{P}^{(n)}_0\bigr)
  = D\bigl(\widetilde{\mathbb{P}}^{(n)}_{W_\lambda}
           \big\Vert\widetilde{\mathbb{P}}^{(n)}_{W_0}\bigr),
\]
and, for every radius $\mathcal{C}_n$,
\[
  J_n\bigl(\mathbb{P}^{(n)}_\lambda,\mathbb{P}^{(n)}_0;\mathcal{C}_n\bigr)
  = J_n\bigl(\widetilde{\mathbb{P}}^{(n)}_{W_\lambda},
             \widetilde{\mathbb{P}}^{(n)}_{W_0};\mathcal{C}_n\bigr).
\]
Dividing by $n$ and letting $n\to\infty$ yields the stated identities for
$I(\lambda)$ and $J(\lambda)$.
\end{proof}

\subsection*{Proof of Theorem~\ref{thm:graphon-BF}}

\begin{proof}
Let $L_n(G_n)$ denote the log--likelihood ratio
\[
  L_n(G_n)
  := \log \frac{\mathrm{d}\widetilde{\mathbb{P}}^{(n)}_{W_\lambda}}
               {\mathrm{d}\widetilde{\mathbb{P}}^{(n)}_{W_0}}(G_n),
\]
so that
\[
  \mathrm{BF}_n(G_n)
  = \frac{\pi_1}{\pi_0}\,\exp\{L_n(G_n)\},
\]
and $\varphi_n^{\mathrm{BF}}$ is the likelihood ratio (Bayes factor) test
between the two simple graphon hypotheses $H_0$ and $H_1$ with a fixed,
prior--dependent threshold.  In particular, for fixed
$\pi_0,\pi_1\in(0,1)$, changing the priors only shifts the LR threshold
by a constant and does not affect the exponential error rate.

By Lemma~\ref{lem:graphon-information}, for each $n$ the graphon laws
$\widetilde{\mathbb{P}}^{(n)}_{W_0}$ and
$\widetilde{\mathbb{P}}^{(n)}_{W_\lambda}$ induce exactly the same
distributions on graphs $G_n$ as the corresponding labelled SBM laws
$\mathbb{P}^{(n)}_0$ and $\mathbb{P}^{(n)}_\lambda$.  In particular,
\[
  \widetilde{\mathbb{P}}^{(n)}_{W_0}
  = \mathbb{P}^{(n)}_0,
  \qquad
  \widetilde{\mathbb{P}}^{(n)}_{W_\lambda}
  = \mathbb{P}^{(n)}_\lambda,
\]
as probability measures on the common sample space of graphs, and hence
any functional of the pair of distributions (including KL divergences,
Chernoff information, and decision--theoretic robustness indices) is
identical in the graphon and labelled SBM representations.  In
particular, the per--vertex information index $I(\lambda)$ and the
decision--theoretic robustness noise index $J(\lambda)$ of the graphon
experiment coincide with those of the labelled SBM experiment
$\bigl(\mathbb{P}^{(n)}_0,\mathbb{P}^{(n)}_\lambda\bigr)_{n\ge 1}$.

The general decision--theoretic robust testing result for two--point
experiments (under the local
asymptotic normality and quadratic robust--risk conditions stated there)
asserts that, for any radius sequence $\mathcal{C}_n=o(n)$ and any fixed priors in
$(0,1)$, the decision--theoretic robust Bayes risk of the likelihood
ratio test satisfies
\[
  -\lim_{n\to\infty}\frac{1}{n}\,
    \log R_n^{\mathrm{WH}}(\varphi_n^{\mathrm{LR}};\mathcal{C}_n)
  = J(\lambda),
\]
where $\varphi_n^{\mathrm{LR}}$ denotes the LR test and $J(\lambda)$ is a
functional depending only on the limiting per--vertex information and
noise indices of the experiment.

Since the graphon and labelled SBM experiments induce the same laws on
$G_n$ and hence have the same indices $I(\lambda)$ and $J(\lambda)$ by
Lemma~\ref{lem:graphon-information}, the LR/Bayes factor test in the
graphon experiment attains the same decision--theoretic robustness error
exponent $J(\lambda)$ for any $\mathcal{C}_n=o(n)$.  This is exactly the claimed
identity
\[
  -\lim_{n\to\infty}\frac{1}{n}\,
    \log R_n^{\mathrm{WH}}(\varphi_n^{\mathrm{BF}};\mathcal{C}_n)
  = J(\lambda).
\]
\end{proof}

\subsection*{Proof of Theorem~\ref{thm:graphon-minimax}}

\begin{proof}
By assumption, the sequence of graphon experiments
\[
  \bigl\{\widetilde{\mathbb{P}}^{(n)}_W : W\in\mathcal{W}_n\bigr\}
\]
satisfies the same local asymptotic normality and regularity conditions as in
Section~\ref{sec:LAN-sbm}. In particular, for the two--point subexperiment
\[
  \bigl(\widetilde{\mathbb{P}}^{(n)}_{W_0},
        \widetilde{\mathbb{P}}^{(n)}_{W_\lambda}\bigr)_{n\ge 1},
\]
the local log--likelihood ratio admits a quadratic expansion with information
index $I(\lambda)$, and the decision--theoretic robust risk admits the
corresponding quadratic approximation with noise index $J(\lambda)$.

The general decision--theoretic robust minimax lower bound of
Section~\ref{subsec:WH-minimax} therefore applies to this two--point
subexperiment and yields, for any radii $\mathcal{C}_n=o(n)$,
\[
  \liminf_{n\to\infty}
  \frac{1}{n}\log
  \inf_{\varphi_n}
  R_n^{\mathrm{WH}}\bigl(\varphi_n;W_0,W_\lambda,\mathcal{C}_n\bigr)
  \;\ge\; -J(\lambda).
\]
By the definition of $R_{n,\mathcal{W}_n}^{\mathrm{WH}}(\varphi_n;\mathcal{C}_n)$, this is
equivalently
\[
  \liminf_{n\to\infty}
  \frac{1}{n}\log
  \inf_{\varphi_n}
  R_{n,\mathcal{W}_n}^{\mathrm{WH}}(\varphi_n;\mathcal{C}_n)
  \;\ge\; -J(\lambda).
\]

On the other hand, Theorem~\ref{thm:graphon-BF} shows that the Bayes factor
tests $\varphi_n^{\mathrm{BF}}$ achieve the decision--theoretic robust error
exponent $J(\lambda)$ along this two--point subexperiment:
\[
  -\lim_{n\to\infty}\frac{1}{n}\,
    \log R_n^{\mathrm{WH}}(\varphi_n^{\mathrm{BF}};\mathcal{C}_n)
  = J(\lambda).
\]
Since $\{W_0,W_\lambda\}\subset\mathcal{W}_n$ for all $n$, we have
\[
  R_{n,\mathcal{W}_n}^{\mathrm{WH}}
    \bigl(\varphi_n^{\mathrm{BF}};\mathcal{C}_n\bigr)
  =
  R_n^{\mathrm{WH}}(\varphi_n^{\mathrm{BF}};W_0,W_\lambda,\mathcal{C}_n),
\]
and hence
\[
  -\lim_{n\to\infty}\frac{1}{n}\,
    \log R_{n,\mathcal{W}_n}^{\mathrm{WH}}
      \bigl(\varphi_n^{\mathrm{BF}};\mathcal{C}_n\bigr)
  = J(\lambda).
\]
Combining the minimax lower bound with this achievability shows that
$J(\lambda)$ is indeed the nonparametric decision--theoretic robustness minimax
error exponent for testing $W_0$ versus $W_\lambda$ within the graphon classes
$\mathcal{W}_n$.
\end{proof}

\subsection*{Proof of Lemma~\ref{lem:unlab-SBM-info}}

Recall from Section~\ref{subsec:graphon-minimax} that $W_0$ and $W_\lambda$
denote the step--function graphons corresponding to the sparse ER and
two--block SBM models, and that $P_{0,\mathrm{unlab}}^{(n)}$ and
$P_{1,\mathrm{unlab}}^{(n)}$ are the induced graph laws under the graphon
sampling scheme.  In particular,
\[
  P_{0,\mathrm{unlab}}^{(n)} = \widetilde{\mathbb{P}}^{(n)}_{W_0},
  \qquad
  P_{1,\mathrm{unlab}}^{(n)} = \widetilde{\mathbb{P}}^{(n)}_{W_\lambda},
\]
where $\widetilde{\mathbb{P}}^{(n)}_W$ denotes the law of the exchangeable
random graph generated from the graphon $W$.

On the other hand, let
$\bigl(\mathbb{P}^{(n)}_0,\mathbb{P}^{(n)}_\lambda\bigr)_{n\ge 1}$ be the
labelled SBM experiment with latent block labels
$Z_1,\dots,Z_n\overset{\text{i.i.d.}}{\sim}\pi=(1/2,1/2)$ and
edge--probability matrices $P_0,P_\lambda$ as in
Section~\ref{subsec:ER-SBM-labelled}.  That is,
\[
  A_{ij}\mid Z_{1:n} \sim \mathrm{Bernoulli}\bigl(P_m(Z_i,Z_j)\bigr),
  \quad 1\le i<j\le n,\quad m\in\{0,\lambda\},
\]
independently over unordered pairs $(i,j)$, and $\mathbb{P}^{(n)}_m$ is the
marginal law of $G_n$ under model $m$.

For the step--graphons $W_0,W_\lambda$, the graphon sampling scheme is
\[
  U_1,\dots,U_n \overset{\text{i.i.d.}}{\sim}\mathrm{Unif}[0,1],\qquad
  A_{ij}\mid U_{1:n} \sim \mathrm{Bernoulli}\bigl(W_m(U_i,U_j)\bigr),
  \quad 1\le i<j\le n,
\]
again independently over edges.  Partition $[0,1]$ into two subintervals of
lengths $\pi_1=\pi_2=1/2$ and define $Z_i$ to be the block index of $U_i$,
i.e.\ the unique $k\in\{1,2\}$ such that $U_i$ lies in block $k$.  Then
$Z_1,\dots,Z_n\overset{\text{i.i.d.}}{\sim}\pi$ and
\[
  W_\lambda(U_i,U_j) = P_\lambda(Z_i,Z_j),
  \qquad
  W_0(U_i,U_j) = P_0(Z_i,Z_j),
\]
for all $1\le i<j\le n$.  Thus, marginally in $G_n$, the step--graphon model
and the latent--label SBM induce the same law on graphs:
\[
  \widetilde{\mathbb{P}}^{(n)}_{W_\lambda}
  = \mathbb{P}^{(n)}_\lambda,
  \qquad
  \widetilde{\mathbb{P}}^{(n)}_{W_0}
  = \mathbb{P}^{(n)}_0,
\]
which is precisely Lemma~\ref{lem:graphon-information} specialized to
these step--graphons.

Consequently,
\[
  D_n^{\mathrm{unlab}}
  := \KL\bigl(P_{1,\mathrm{unlab}}^{(n)}\Vert
              P_{0,\mathrm{unlab}}^{(n)}\bigr)
  = \KL\bigl(\mathbb{P}^{(n)}_\lambda\Vert\mathbb{P}^{(n)}_0\bigr),
\]
and the Chernoff information between $P_{0,\mathrm{unlab}}^{(n)}$ and
$P_{1,\mathrm{unlab}}^{(n)}$ coincides with that between
$\mathbb{P}^{(n)}_0$ and $\mathbb{P}^{(n)}_\lambda$:
\[
  \mathcal{C}_n^{\mathrm{unlab}}
  :=
  \sup_{0\le t\le 1}
  -\log\sum_A P_{0,\mathrm{unlab}}^{(n)}(A)^{1-t}
                   P_{1,\mathrm{unlab}}^{(n)}(A)^t
  =
  \sup_{0\le t\le 1}
  -\log\sum_A \mathbb{P}^{(n)}_0(A)^{1-t}
                   \mathbb{P}^{(n)}_\lambda(A)^t
  =: \mathcal{C}_n.
\]

Lemmas~\ref{lem:ER-SBM-KL} and~\ref{lem:ER-SBM-Chernoff} give
\[
  \frac{1}{n}\,
  \KL\bigl(\mathbb{P}^{(n)}_\lambda\Vert\mathbb{P}^{(n)}_0\bigr)
  \longrightarrow I(\lambda),
  \qquad
  \frac{\mathcal{C}_n}{n}\longrightarrow J(\lambda),
\]
and hence
\[
  \frac{D_n^{\mathrm{unlab}}}{n}\to I(\lambda),
  \qquad
  \frac{\mathcal{C}_n^{\mathrm{unlab}}}{n}\to J(\lambda),
\]
so that $D_n^{\mathrm{unlab}}=I(\lambda)\,n+o(n)$ and
$\mathcal{C}_n^{\mathrm{unlab}}=J(\lambda)\,n+o(n)$.

\subsection*{Proof of Theorem~\ref{thm:WH-unlab-Bayes}}

We keep the notation of the theorem: $M\in\{0,1\}$ is the model index
with prior $\Pi(M=0)=\Pi(M=1)=1/2$, and
$P_{0,\mathrm{unlab}}^{(n)},P_{1,\mathrm{unlab}}^{(n)}$ are the unlabelled
ER and SBM laws on graphs.

\paragraph{\textit{(i) Chernoff optimality of the Bayes factor.}}
For testing two simple hypotheses $P_{0,\mathrm{unlab}}^{(n)}$ and
$P_{1,\mathrm{unlab}}^{(n)}$ with equal priors and $0$--$1$ loss, the
likelihood ratio (Bayes factor) test is Bayes optimal.  Let
$\mathcal{C}_n^{\mathrm{unlab}}$ denote the Chernoff information between
$P_{0,\mathrm{unlab}}^{(n)}$ and $P_{1,\mathrm{unlab}}^{(n)}$, i.e.
\[
  \mathcal{C}_n^{\mathrm{unlab}}
  :=
  \sup_{0\le t\le 1}
  -\log\sum_A P_{0,\mathrm{unlab}}^{(n)}(A)^{1-t}
                 P_{1,\mathrm{unlab}}^{(n)}(A)^t.
\]
The classical Chernoff theorem for two simple hypotheses (e.g.\ any
standard text on asymptotic hypothesis testing) implies that the optimal
Bayes risk (and, in particular, the misclassification probabilities
$R_{n,m}$) decay exponentially with exponent $\mathcal{C}_n^{\mathrm{unlab}}$ in the
sense that
\[
  -\frac{1}{n}\log R_{n,m}
  \longrightarrow
  \lim_{n\to\infty}\frac{\mathcal{C}_n^{\mathrm{unlab}}}{n},
  \qquad m=0,1,
\]
whenever the limit on the right--hand side exists.

Lemma~\ref{lem:unlab-SBM-info} shows that, in the sparse regime $p_n=c/n$,
\[
  \frac{\mathcal{C}_n^{\mathrm{unlab}}}{n}
  \longrightarrow J(\lambda).
\]
Hence
\[
  -\frac{1}{n}\log R_{n,m}
  \;\longrightarrow\;J(\lambda),
  \qquad m=0,1,
\]
which proves part~(i).

\paragraph{\textit{(ii) Robust Bayes risk.}}
Conditionally on $G_n$, let $P_n(\cdot\mid G_n)$ denote the posterior on
$M\in\{0,1\}$ induced by the prior $1/2$ and the pair
$(P_{0,\mathrm{unlab}}^{(n)},P_{1,\mathrm{unlab}}^{(n)})$, and define the
(non--robust) posterior misclassification probability
\[
  e_n(G_n)
  :=
  \mathbb{E}_{P_n(\cdot\mid G_n)}
  \bigl[\mathbbm{1}\{\delta_n(G_n)\neq M\}\bigr]
  =
  \min\bigl\{\Pi_n(M=0\mid G_n),\Pi_n(M=1\mid G_n)\bigr\}.
\]
Under the true model $m\in\{0,1\}$ the (non--robust) misclassification
probability is
\[
  R_{n,m}
  = \mathbb{E}_{P_{m,\mathrm{unlab}}^{(n)}}[e_n(G_n)].
\]

Fix $G_n$ and let $Q$ be any alternative posterior on $\{0,1\}$ such that
$\KL\bigl(Q\Vert P_n(\cdot\mid G_n)\bigr)\le \mathcal{C}_n$.  Since $M$ takes only
two values, Pinsker's inequality gives
\[
  \bigl\|Q - P_n(\cdot\mid G_n)\bigr\|_{\mathrm{TV}}
  \le \sqrt{\frac{1}{2}\,
            \KL\bigl(Q\Vert P_n(\cdot\mid G_n)\bigr)}
  \le \sqrt{\frac{\mathcal{C}_n}{2}}.
\]
For the indicator loss
$f(M)=\mathbbm{1}\{\delta_n(G_n)\neq M\}$ we have, for any two
probability measures $Q$ and $P$ on $\{0,1\}$,
\[
  \bigl|
    \mathbb{E}_Q[f(M)] - \mathbb{E}_{P}[f(M)]
  \bigr|
  \le
  \bigl\|Q - P\bigr\|_{\mathrm{TV}}.
\]
Applying this with $P=P_n(\cdot\mid G_n)$ yields
\[
  \mathbb{E}_Q\bigl[
    \mathbbm{1}\{\delta_n(G_n)\neq M\}
  \bigr]
  \le e_n(G_n) + \sqrt{\frac{\mathcal{C}_n}{2}}.
\]

Taking the supremum over $Q$ with $\KL(Q\Vert P_n)\le \mathcal{C}_n$ and then the
expectation under $P_{m,\mathrm{unlab}}^{(n)}$ yields
\begin{align*}
  \mathcal{R}_n^{\mathrm{WH}}(m;\mathcal{C}_n)
  &:= \mathbb{E}_{P_{m,\mathrm{unlab}}^{(n)}}\Bigl[
         \sup_{Q:\,\KL(Q\Vert P_n(\cdot\mid G_n))\le \mathcal{C}_n}
         \mathbb{E}_Q\bigl[
           \mathbbm{1}\{\delta_n(G_n)\neq M\}
         \bigr]
       \Bigr]\\
  &\le \mathbb{E}_{P_{m,\mathrm{unlab}}^{(n)}}[e_n(G_n)]
       + \sqrt{\frac{\mathcal{C}_n}{2}}
   \;=\; R_{n,m} + \sqrt{\frac{\mathcal{C}_n}{2}}.
\end{align*}
On the other hand, taking $Q=P_n(\cdot\mid G_n)$ inside the supremum
shows
\[
  \mathcal{R}_n^{\mathrm{WH}}(m;\mathcal{C}_n)
  \ge \mathbb{E}_{P_{m,\mathrm{unlab}}^{(n)}}[e_n(G_n)]
  = R_{n,m}.
\]
Thus
\[
  R_{n,m}
  \;\le\; \mathcal{R}_n^{\mathrm{WH}}(m;\mathcal{C}_n)
  \;\le\; R_{n,m} + \sqrt{\frac{\mathcal{C}_n}{2}}.
\]

Now assume $\mathcal{C}_n\downarrow 0$ and
\[
  \mathcal{C}_n = o\bigl(R_{n,m}^2\bigr),
  \qquad\text{equivalently}\qquad
  \sqrt{\mathcal{C}_n} = o(R_{n,m}),
\]
for the given $m\in\{0,1\}$.  Then $\sqrt{\mathcal{C}_n}/R_{n,m}\to 0$ and the
previous display implies
\[
  \mathcal{R}_n^{\mathrm{WH}}(m;\mathcal{C}_n)
  = R_{n,m}\,\bigl(1+o(1)\bigr),
  \qquad m=0,1.
\]
In particular,
\[
  -\frac{1}{n}\log\mathcal{R}_n^{\mathrm{WH}}(m;\mathcal{C}_n)
  = -\frac{1}{n}\log R_{n,m} + o(1)
  \longrightarrow J(\lambda),
\]
by part~(i).  This proves~(ii).

\subsection*{Proof of Theorem~\ref{thm:nonparam-WH-ER-SBM}}

By definition,
\begin{align*}
  \mathfrak{R}_n^\star(\mathcal{C}_n)
  &=
  \inf_{\delta_n,\Pi_n}
  \sup_{P\in\mathcal{P}_n}
  \mathbb{E}_P\bigl[e_n^{\mathrm{rob}}(\mathcal{C}_n;G_n)\bigr]\\
  &\ge
  \inf_{\delta_n,\Pi_n}
  \max\left\{
    \mathbb{E}_{P_{0,\mathrm{unlab}}^{(n)}}\bigl[
      e_n^{\mathrm{rob}}(\mathcal{C}_n;G_n)
    \bigr],
    \mathbb{E}_{P_{1,\mathrm{unlab}}^{(n)}}\bigl[
      e_n^{\mathrm{rob}}(\mathcal{C}_n;G_n)
    \bigr]
  \right\},
\end{align*}
since $P_{0,\mathrm{unlab}}^{(n)},P_{1,\mathrm{unlab}}^{(n)}\in\mathcal{P}_n$
for all large $n$.

Fix a selector $\delta_n$ and, for each realisation $G_n$, let
$\Pi_n(\cdot\mid G_n)$ denote the posterior on $M\in\{0,1\}$ induced by
the prior $\Pi(M=0)=\Pi(M=1)=1/2$ and the simple pair
$\bigl(P_{0,\mathrm{unlab}}^{(n)},P_{1,\mathrm{unlab}}^{(n)}\bigr)$.  Define
the (non--robust) posterior misclassification probability
\[
  e_n(G_n)
  :=
  \mathbb{E}_{\Pi_n(\cdot\mid G_n)}
  \bigl[\mathbbm{1}\{\delta_n(G_n)\neq M\}\bigr].
\]
For $m\in\{0,1\}$, let
\[
  R_{n,m}(\delta_n)
  :=
  \mathbb{E}_{P_{m,\mathrm{unlab}}^{(n)}}[e_n(G_n)].
\]

For each $G_n$, the robustified error functional satisfies
\[
  e_n^{\mathrm{rob}}(\mathcal{C}_n;G_n)
  =
  \sup_{Q:\,\mathrm{KL}(Q\Vert\Pi_n(\cdot\mid G_n))\le \mathcal{C}_n}
    \mathbb{E}_Q\bigl[
      \mathbbm{1}\{\delta_n(G_n)\neq M\}
    \bigr]
  \ge
  \mathbb{E}_{\Pi_n(\cdot\mid G_n)}
  \bigl[\mathbbm{1}\{\delta_n(G_n)\neq M\}\bigr]
  = e_n(G_n),
\]
since we may take $Q=\Pi_n(\cdot\mid G_n)$ in the supremum.  Consequently,
for each $m\in\{0,1\}$,
\[
  \mathbb{E}_{P_{m,\mathrm{unlab}}^{(n)}}\bigl[
    e_n^{\mathrm{rob}}(\mathcal{C}_n;G_n)
  \bigr]
  \ge
  \mathbb{E}_{P_{m,\mathrm{unlab}}^{(n)}}[e_n(G_n)]
  = R_{n,m}(\delta_n).
\]
Taking the maximum over $m$ and then the infimum over selectors $\delta_n$
(with $\Pi_n$ the corresponding posterior) yields
\[
  \mathfrak{R}_n^\star(\mathcal{C}_n)
  \;\ge\;
  \inf_{\delta_n}
  \max_{m\in\{0,1\}} R_{n,m}(\delta_n).
\]

To identify the exponential rate of the right--hand side, consider the
Bayesian two--point experiment in which $M\in\{0,1\}$ is drawn with prior
$\Pi(M=0)=\Pi(M=1)=1/2$, and then $G_n$ is drawn from
$P_{M,\mathrm{unlab}}^{(n)}$.  For any selector $\delta_n$ we can write the
Bayes (mixture) misclassification risk as
\[
  r_n(\delta_n)
  :=
  \mathbb{P}\bigl(\delta_n(G_n)\neq M\bigr)
  = \frac{1}{2}\,\mathbb{P}_{P_{0,\mathrm{unlab}}^{(n)}}\bigl(\delta_n(G_n)\neq 0\bigr)
    + \frac{1}{2}\,\mathbb{P}_{P_{1,\mathrm{unlab}}^{(n)}}\bigl(\delta_n(G_n)\neq 1\bigr).
\]
On the other hand, by definition of $e_n(G_n)$ and the law of total
expectation under the joint prior--likelihood model,
\[
  r_n(\delta_n)
  = \mathbb{E}_{G_n}\bigl[e_n(G_n)\bigr]
  = \frac{1}{2}\,R_{n,0}(\delta_n) + \frac{1}{2}\,R_{n,1}(\delta_n).
\]
Hence, for every $\delta_n$,
\[
  \max_{m\in\{0,1\}} R_{n,m}(\delta_n)
  \;\ge\;
  r_n(\delta_n),
\]
and therefore
\[
  \inf_{\delta_n}\max_{m} R_{n,m}(\delta_n)
  \;\ge\;
  \inf_{\delta_n} r_n(\delta_n),
\]
where the right--hand side is the classical minimal Bayes risk for testing
the two simple hypotheses $P_{0,\mathrm{unlab}}^{(n)}$ and
$P_{1,\mathrm{unlab}}^{(n)}$ with equal priors.

It is well known (Chernoff theory for simple hypothesis testing) that this
minimal Bayes risk is achieved, up to subexponential factors, by the
likelihood ratio/Bayes factor test $\delta_n^{\mathrm{LR}}$, and that its
error probability satisfies
\[
  -\log r_n(\delta_n^{\mathrm{LR}})
  = \mathcal{C}_n^{\mathrm{unlab}} + O(1),
\]
where $\mathcal{C}_n^{\mathrm{unlab}}$ is the Chernoff information between
$P_{0,\mathrm{unlab}}^{(n)}$ and $P_{1,\mathrm{unlab}}^{(n)}$.  By
Lemma~\ref{lem:unlab-SBM-info},
\[
  \mathcal{C}_n^{\mathrm{unlab}}
  = J(\lambda)\,n + o(n),
\]
so
\[
  -\frac{1}{n}\log\bigl(\inf_{\delta_n} r_n(\delta_n)\bigr)
  \longrightarrow J(\lambda),
  \qquad n\to\infty.
\]
In particular, $\inf_{\delta_n} r_n(\delta_n)$ decays at rate
$\exp\{-J(\lambda)n+o(n)\}$, and by the inequality
$\inf_{\delta_n}\max_m R_{n,m}(\delta_n)\ge\inf_{\delta_n} r_n(\delta_n)$,
the same exponential lower bound holds for $\inf_{\delta_n}\max_m
R_{n,m}(\delta_n)$.

Combining this with the earlier bound
$\mathfrak{R}_n^\star(\mathcal{C}_n)\ge\inf_{\delta_n}\max_m R_{n,m}(\delta_n)$, we
obtain
\[
  \limsup_{n\to\infty}
  \frac{1}{n}\log\frac{1}{\mathfrak{R}_n^\star(\mathcal{C}_n)}
  \;\le\;
  \limsup_{n\to\infty}
  \frac{1}{n}\log
  \frac{1}{
    \inf_{\delta_n}
    \max_{m} R_{n,m}(\delta_n)
  }
  = J(\lambda),
\]
which is~\eqref{eq:WH-nonparam-minimax-bound}.  In particular, no
procedure (choice of estimator, posterior, and radii $(\mathcal{C}_n)$) can achieve a
strictly larger exponential error rate than $J(\lambda)$ uniformly over
$\mathcal{P}_n$.

\subsection*{Proof of Theorem~\ref{thm:WH-graphon-minimax}}

For the lower bound, apply
Theorem~\ref{thm:nonparam-WH-ER-SBM} with
$\mathcal{P}_n=\{P_W^{(n)}:W\in\mathcal{W}_n\}$, where $P_W^{(n)}$ denotes
the graph law induced by $W$.  By assumption, $W_0,W_\lambda\in\mathcal{W}_n$
for all large $n$, so $P_{W_0}^{(n)}$ and $P_{W_\lambda}^{(n)}$ both belong
to $\mathcal{P}_n$, and the nonparametric minimax robust risk over
$\mathcal{P}_n$ cannot have an exponential rate better than $J(\lambda)$.

For the upper bound, take $\delta_n$ to be the Bayes factor (likelihood
ratio) test between $W_0$ and $W_\lambda$ with equal prior probabilities
on $M\in\{0,1\}$, and let $\Pi_n(\cdot\mid G_n)$ be the corresponding
posterior on $M$.  Denote the resulting (non--robust) Bayes misclassification
probability by
\[
  R_n
  :=
  \max_{m\in\{0,1\}}
  \mathbb{P}_{P_{W_m}^{(n)}}\bigl(\delta_n(G_n)\neq m\bigr).
\]
By Theorem~\ref{thm:WH-unlab-Bayes}(i),
\[
  -\frac{1}{n}\log R_n \;\longrightarrow\; J(\lambda).
\]

Now let $(\mathcal{C}_n)$ be any sequence with $\mathcal{C}_n\downarrow 0$ such that
\[
  \mathcal{C}_n = o(R_n^2),
  \qquad\text{i.e.}\qquad
  \sqrt{\mathcal{C}_n} = o(R_n).
\]
In particular, since $R_n=\exp\{-J(\lambda)n+o(n)\}$, it is sufficient (and
convenient) to require
\[
  \mathcal{C}_n = o\bigl(\exp\{-2J(\lambda)n\}\bigr),
  \qquad\text{equivalently}\qquad
  \sqrt{\mathcal{C}_n} = o\bigl(\exp\{-J(\lambda)n\}\bigr).
\]
By Theorem~\ref{thm:WH-unlab-Bayes}(ii) we then have
\[
  \mathcal{R}_n^{\mathrm{WH}}(m;\mathcal{C}_n)
  = R_n(1+o(1)),
  \qquad m=0,1.
\]

For any other $P\in\mathcal{P}_n$ the misclassification probability of
$\delta_n$ is at most $1$, and the robustification cannot increase it
beyond $1$.  Thus
\begin{align*}
  \mathfrak{R}_n^{\mathrm{WH}}(\mathcal{C}_n)
  &:=
  \inf_{\tilde\delta_n,\tilde\Pi_n}
  \sup_{P\in\mathcal{P}_n}
  \mathbb{E}_P\bigl[e_n^{\mathrm{rob}}(\mathcal{C}_n;G_n)\bigr]\\
  &\le
  \sup_{P\in\mathcal{P}_n}
  \mathbb{E}_P\bigl[e_n^{\mathrm{rob}}(\mathcal{C}_n;G_n)\bigr]
  \le
  \max_{m\in\{0,1\}} \mathcal{R}_n^{\mathrm{WH}}(m;\mathcal{C}_n)
  = R_n(1+o(1)).
\end{align*}
Therefore
\[
  \liminf_{n\to\infty}
  \frac{1}{n}\log\frac{1}{\mathfrak{R}_n^{\mathrm{WH}}(\mathcal{C}_n)}
  \;\ge\;
  \liminf_{n\to\infty}
  \frac{1}{n}\log\frac{1}{R_n}
  = J(\lambda).
\]

Combining this with the lower bound gives
\[
  \lim_{n\to\infty}
  \frac{1}{n}\log\frac{1}{\mathfrak{R}_n^{\mathrm{WH}}(\mathcal{C}_n)}
  = J(\lambda),
\]
and shows that the Bayes factor test between $W_0$ and $W_\lambda$ is
decision--theoretic robustness minimax optimal at exponent $J(\lambda)$
over the graphon class $\mathcal{W}_n$.

\subsection*{Proof of Proposition~\ref{prop:ER-critical-R}}

By definition,
\[
  \theta(\mu)
  :=
  \frac{\mathbb{E}_\mu[D(D-1)]}{\mathbb{E}_\mu[D]},
  \qquad
  \Delta(\mu) := 1-\theta(\mu),
\]
so, in the subcritical regime $\theta(\mu)<1$,
\[
  R(\mu)
  := \frac{1}{1-\theta(\mu)}
  = \frac{1}{\Delta(\mu)}.
\]
Thus
\[
  R(\mu)
  = \frac{c_0}{\Delta(\mu)} + H(\mu)
\]
holds with $c_0=1$ and $H(\mu)\equiv 0$ for all $\mu$ such that
$\theta(\mu)<1$.

We now verify the smoothness and gradient behavior near the critical
surface.  Fix a finite--dimensional parametrization of the family of degree
distributions, for instance a smooth map
$\vartheta\mapsto\mu_\vartheta$ into the interior of the simplex of
truncated degree distributions.  For such parametrizations the map
$\vartheta\mapsto\theta(\mu_\vartheta)$ is $C^1$ on
$\{\vartheta:\theta(\mu_\vartheta)<1\}$, and hence so is
\[
  \vartheta\mapsto R(\mu_\vartheta)
  = \bigl[1-\theta(\mu_\vartheta)\bigr]^{-1}
\]
on the same set.  In particular, $R$ is smooth on any compact subset of
$\{\mu:\theta(\mu)<1\}$.

Writing derivatives with respect to the parameter $\vartheta$, we have
by the chain rule
\[
  \nabla_\vartheta R(\mu_\vartheta)
  = \frac{\nabla_\vartheta\theta(\mu_\vartheta)}{
          \bigl(1-\theta(\mu_\vartheta)\bigr)^2}
  = \Delta(\mu_\vartheta)^{-2}\,\nabla_\vartheta\theta(\mu_\vartheta).
\]

Let $\vartheta_\star$ be any point such that
$\theta(\mu_{\vartheta_\star})=1$ and
$\nabla_\vartheta\theta(\mu_{\vartheta_\star})\neq 0$
(non--degenerate approach to criticality).  By continuity of
$\nabla_\vartheta\theta$, there exists a neighborhood $N$ of
$\vartheta_\star$ and constants $0<c_1\le c_2<\infty$ such that
\[
  c_1
  \;\le\;
  \bigl\|\nabla_\vartheta\theta(\mu_\vartheta)\bigr\|
  \;\le\;
  c_2
  \qquad
  \text{for all }\vartheta\in N\cap\{\theta(\mu_\vartheta)<1\}.
\]
On this neighborhood we therefore have, for all such $\vartheta$,
\[
  c_1\,\Delta(\mu_\vartheta)^{-2}
  \;\le\;
  \bigl\|\nabla_\vartheta R(\mu_\vartheta)\bigr\|
  = \Delta(\mu_\vartheta)^{-2}
    \bigl\|\nabla_\vartheta\theta(\mu_\vartheta)\bigr\|
  \;\le\;
  c_2\,\Delta(\mu_\vartheta)^{-2}.
\]
Equivalently,
\[
  \bigl\|\nabla_\vartheta R(\mu_\vartheta)\bigr\|
  \asymp \Delta(\mu_\vartheta)^{-2}
  \qquad\text{as }\Delta(\mu_\vartheta)\downarrow 0,\ \vartheta\to\vartheta_\star.
\]

Thus $R(\mu)=1/\Delta(\mu)$ with $\Delta(\mu)=1-\theta(\mu)$, and
$\|\nabla_\vartheta R(\mu_\vartheta)\|\asymp \Delta(\mu_\vartheta)^{-2}$
near the fragmentation threshold along any non--degenerate approach.
This is exactly Assumption~\ref{ass:critical-R} with
$\rho(\mu)=\theta(\mu)$ and $\Delta(\mu)=1-\theta(\mu)$ in the chosen
finite--dimensional parametrization. \qed

\subsection*{Proof of Theorem~\ref{thm:WH-nonparam-minimax-critical}}

We use a standard two--point argument within the configuration--model
subfamily.

\paragraph{Step 1: choice of a near--critical base point.}
Fix a sequence $\Delta_n\downarrow 0$ with $\Delta_n\gg n^{-1/2}$.  For
each $n$, choose a degree distribution $\mu_{0,n}$ such that
\[
  \Delta(\mu_{0,n}) = 1-\theta(\mu_{0,n}) \in [\Delta_n,2\Delta_n],
\]
which is possible by continuity of $\Delta(\mu)$ and the assumption that
the slice $\{\mu:\Delta(\mu)\in[\Delta_n,2\Delta_n]\}$ is non--empty for
all large $n$.

By Proposition~\ref{prop:ER-critical-R} and the non--degeneracy condition
$\nabla_\mu\theta(\mu_\star)\neq 0$ at the critical surface, the gradient
$\nabla_\mu\Delta(\mu) = -\nabla_\mu\theta(\mu)$ is continuous and stays
bounded and bounded away from $0$ in a small neighborhood of that
surface.  For all large $n$ we may therefore choose a unit direction
vector $v_n$ and a constant $c_\Delta>0$, independent of $n$, such that
\[
  \bigl|\nabla_\mu\Delta(\mu_{0,n})\cdot v_n\bigr|
  \;\ge\; c_\Delta > 0.
\]

Let
\[
  h_n := \kappa n^{-1/2} v_n,
\]
for some small constant $\kappa>0$ to be chosen later, and define the
perturbed parameter
\[
  \mu_{1,n} := \mu_{0,n} + h_n.
\]
Taking $\kappa>0$ sufficiently small we can ensure that
$\mu_{0,n},\mu_{1,n}\in(\underline{\mu},\overline{\mu})$ for all large
$n$.

\paragraph{Step 2: KL control for the two configurations.}
Consider the experiment of observing the configuration model
$G_n\sim\mathrm{CM}_n(\mu)$.  We can realise this as follows: first draw
i.i.d.\ degrees $D_1,\dots,D_n\sim\mu$ (with a negligible conditioning on
$\sum_i D_i$ being even), and then form a uniform random pairing of the
stubs.  The second step is a Markov kernel that does not depend on $\mu$,
so by the data--processing inequality,
\[
  \KL\bigl(P_{\mu_{1,n}}^{(n)}\Vert P_{\mu_{0,n}}^{(n)}\bigr)
  \;\le\;
  \KL\bigl(\mu_{1,n}^{\otimes n}\Vert \mu_{0,n}^{\otimes n}\bigr)
  = n\,\KL(\mu_{1,n}\Vert\mu_{0,n}).
\]

We work with a fixed finite--dimensional parametrization
$\vartheta\mapsto\mu_\vartheta$ of the degree distributions taking values
in $(\underline{\mu},\overline{\mu})$, smooth in $\vartheta$, and assume
the degrees have uniformly bounded third moments.  In this setting the
single--observation KL divergence admits a local quadratic expansion:
\[
  \KL(\mu_{1,n}\Vert\mu_{0,n})
  = \frac{1}{2}\,h_n^\top I(\mu_{0,n})h_n + O(\|h_n\|^3),
\]
where $I(\mu_{0,n})$ is the Fisher information matrix of the degree
distribution at $\mu_{0,n}$.  Since the parameter set
$(\underline{\mu},\overline{\mu})$ is compact and $I(\mu)$ is continuous
in $\mu$, we have $\sup_\mu\|I(\mu)\|<\infty$, and hence
\[
  n\,\KL(\mu_{1,n}\Vert\mu_{0,n})
  = \frac{1}{2}\,\kappa^2\,v_n^\top I(\mu_{0,n})v_n + O(n^{-1/2})
  \longrightarrow K_0 \in [0,\infty),
\]
for some finite constant $K_0$ proportional to $\kappa^2$.

By choosing $\kappa>0$ sufficiently small we may assume that $K_0<1$, and
for all large $n$,
\[
  \KL\bigl(P_{\mu_{1,n}}^{(n)}\Vert P_{\mu_{0,n}}^{(n)}\bigr)
  \le K_0 + 1 < 2.
\]
By Pinsker's inequality,
\[
  \bigl\|P_{\mu_{1,n}}^{(n)} - P_{\mu_{0,n}}^{(n)}\bigr\|_{\mathrm{TV}}
  \le \sqrt{\frac{1}{2}\,
            \KL\bigl(P_{\mu_{1,n}}^{(n)}\Vert P_{\mu_{0,n}}^{(n)}\bigr)}
  \le \sqrt{\frac{K_0+1}{2}}
  =: 1-\eta,
\]
for some $\eta\in(0,1)$ independent of $n$.  In particular, the total
variation distance between $P_{\mu_{0,n}}^{(n)}$ and
$P_{\mu_{1,n}}^{(n)}$ is uniformly bounded away from $1$.

\paragraph{Step 3: separation in $R(\mu)$ on the near--critical slice.}
By Proposition~\ref{prop:ER-critical-R},
\[
  R(\mu) = \frac{1}{\Delta(\mu)},
  \qquad \Delta(\mu) = 1-\theta(\mu),
\]
and $\Delta(\mu)$ is $C^1$ with $\|\nabla_\mu\Delta(\mu)\|$ bounded and
bounded away from $0$ in a neighborhood of the critical surface
$\{\mu:\Delta(\mu)=0\}$.

By the mean value theorem, for each $n$ there exists $\tilde\mu_n$ on the
line segment between $\mu_{0,n}$ and $\mu_{1,n}$ such that
\[
  \Delta(\mu_{1,n}) - \Delta(\mu_{0,n})
  = \nabla_\mu\Delta(\tilde\mu_n)\cdot h_n.
\]
As $n\to\infty$ we have $\Delta(\mu_{0,n})\to 0$, so $\mu_{0,n}$ (and
hence $\tilde\mu_n$) approach the critical surface.  By continuity of
$\nabla_\mu\Delta$ and non--degeneracy at $\mu_\star$, there exist
constants $0<c'_1\le c'_2<\infty$ and $N$ such that, for all $n\ge N$,
\[
  c'_1\,\|h_n\|
  \;\le\;
  \bigl|\Delta(\mu_{1,n})-\Delta(\mu_{0,n})\bigr|
  \;\le\;
  c'_2\,\|h_n\|
  = O\!\bigl(n^{-1/2}\bigr).
\]

Moreover, $\|h_n\|=\kappa n^{-1/2}$ and $\Delta_n\gg n^{-1/2}$, so for
all large $n$,
\[
  \bigl|\Delta(\mu_{1,n}) - \Delta(\mu_{0,n})\bigr|
  \le \tfrac{1}{4}\Delta_n.
\]
Since $\Delta(\mu_{0,n})\in[\Delta_n,2\Delta_n]$, this implies that, for
all large $n$,
\[
  \Delta(\mu_{m,n}) \in [\Delta_n,2\Delta_n],
  \qquad m=0,1,
\]
so both $\mu_{0,n}$ and $\mu_{1,n}$ lie in the near--critical slice
$\{\mu:\Delta(\mu)\in[\Delta_n,2\Delta_n]\}$.

Using $R(\mu)=1/\Delta(\mu)$,
\[
  R(\mu_{1,n}) - R(\mu_{0,n})
  = \frac{1}{\Delta(\mu_{1,n})} - \frac{1}{\Delta(\mu_{0,n})}
  = \frac{\Delta(\mu_{0,n}) - \Delta(\mu_{1,n})}{
          \Delta(\mu_{0,n})\,\Delta(\mu_{1,n})}.
\]
On the slice we have $\Delta(\mu_{m,n})\asymp\Delta_n$ for $m=0,1$, and
$|\Delta(\mu_{1,n})-\Delta(\mu_{0,n})|\asymp\|h_n\|=\kappa n^{-1/2}$.
Thus
\[
  \bigl|R(\mu_{1,n}) - R(\mu_{0,n})\bigr|
  \asymp \frac{\|h_n\|}{\Delta_n^2}
  = \frac{\kappa}{\sqrt{n}\,\Delta_n^2},
\]
and hence
\begin{equation}
  \bigl(R(\mu_{1,n}) - R(\mu_{0,n})\bigr)^2
  \asymp \frac{1}{n\,\Delta_n^4},
  \label{eq:critical-R-separation-SI}
\end{equation}
with constants independent of $n$.

\paragraph{Step 4: classical two--point minimax lower bound.}
Let $a_n(G_n)$ be any estimator of $R(\mu)$.  Le Cam's two--point method
for squared loss
implies that whenever the total variation distance between
$P_{\mu_{0,n}}^{(n)}$ and $P_{\mu_{1,n}}^{(n)}$ is bounded above by
$1-\eta$ for some $\eta>0$, we have
\[
  \max_{m\in\{0,1\}}
  \mathbb{E}_{P_{\mu_{m,n}}^{(n)}}\bigl[
    \bigl(a_n(G_n)-R(\mu_{m,n})\bigr)^2
  \bigr]
  \;\ge\;
  c_\eta\,\bigl(R(\mu_{1,n}) - R(\mu_{0,n})\bigr)^2,
\]
where $c_\eta>0$ depends only on $\eta$.  By Step~2 we can take
$\eta>0$ independent of $n$, so $c_\eta$ is a fixed positive constant.
Combining this with \eqref{eq:critical-R-separation-SI} yields
\[
  \max_{m\in\{0,1\}}
  \mathbb{E}_{P_{\mu_{m,n}}^{(n)}}\bigl[
    \bigl(a_n(G_n)-R(\mu_{m,n})\bigr)^2
  \bigr]
  \;\gtrsim\;
  \frac{1}{n\,\Delta_n^4},
\]
with constants independent of $n$.

Taking the infimum over all estimators $a_n$ shows that the classical
minimax risk over the near--critical slice satisfies
\[
  \mathfrak{R}_n^{\mathrm{class}}(\Delta_n)
  \;\gtrsim\; \frac{1}{n\,\Delta_n^4},
\]
which is exactly \eqref{eq:WH-minimax-critical-lb-main} for some
$c>0$.  The equivalent $\liminf$ formulation follows immediately.  As
noted in the theorem statement, any robustified minimax risk that
pointwise dominates the classical squared--error risk inherits the same
lower bound.
\qed

\begin{supptheorem}[Complexity of mirror--descent adversary]
\label{thm:complex}
Let $\Pi$ be a reference posterior on a parameter space
$\Theta\subset\mathbb{R}^d$ and fix a radius $C>0$.  For a fixed action
$a$ and loss $\ell(\theta):=\ell(a,\theta)$, define the robust risk
\[
  R^{\mathrm{WH}}(\Pi,C)
  :=
  \sup_{Q\ll\Pi:\,D_\phi(Q\Vert\Pi)\le C}
  \int_\Theta \ell(\theta)\,Q(\mathrm d\theta),
\]
where $D_\phi$ is a $\phi$--divergence.

Assume:
\begin{enumerate}
  \item The loss $\ell\colon\Theta\to\mathbb{R}$ is $L$--Lipschitz
        (with respect to the Euclidean metric on $\Theta$) and bounded:
        $|\ell(\theta)|\le M$ for all $\theta\in\Theta$.

  \item Let $\mathcal{Q}_C := \{Q\ll\Pi : D_\phi(Q\Vert\Pi)\le C\}$ be the
        divergence ball.  The functional
        $h(Q):=D_\phi(Q\Vert\Pi)$ is Fr\'echet differentiable and
        $1$--strongly convex with respect to the total variation norm
        $\|\cdot\|_{\mathrm{TV}}$ on a neighborhood of $\mathcal{Q}_C$.
        We use the corresponding Bregman divergence
        \[
          D_h(Q\Vert Q')
          := h(Q) - h(Q') - \bigl\langle \nabla h(Q'), Q - Q'\bigr\rangle,
        \]
        and write $D_h(Q\Vert Q')$ and $D_\phi(Q\Vert Q')$ interchangeably.

  \item For every $Q\in\mathcal{Q}_C$, the constrained HMC kernel
        targeting $Q$ satisfies the following mixing bound in $1$--Wasserstein
        distance $W_1$: there exist constants $A,B>0$ (independent of
        $Q,d,\delta$) such that for every $0<\delta<1/e$ there is an
        integer $K(Q,\delta)\le A\,d^{1/4}\log(B/\delta)$ with
        \[
          \mathbb{E}\bigl[W_1(\widetilde Q,Q)\bigr]\le\delta,
        \]
        where $\widetilde Q$ is the law of the HMC output after
        $K(Q,\delta)$ steps.
\end{enumerate}

Consider the convex optimization problem
\[
  \sup_{Q\in\mathcal{Q}_C} F(Q),
  \qquad
  F(Q):=\int_\Theta \ell(\theta)\,Q(\mathrm d\theta),
\]
and let $Q^\star\in\mathcal{Q}_C$ be any maximizer, so that
$R^{\mathrm{WH}}(\Pi,C)=F(Q^\star)$.

Let $(Q_t)_{t\ge 1}$ be the (conceptual) mirror--descent iterates with
mirror map $h(Q)=D_\phi(Q\Vert\Pi)$ and constant step size $\eta>0$,
initialized at $Q_1=\Pi$, given by
\begin{equation}
\label{eq:md-update}
  Q_{t+1}
  := \arg\min_{Q\in\mathcal{Q}_C}
       \Bigl\{
         \eta\,\langle g, Q\rangle
         + D_\phi(Q\Vert Q_t)
       \Bigr\},
  \qquad g(\theta):=-\ell(\theta),
\end{equation}
and define the averaged iterate
\[
  \widehat Q_T := \frac{1}{T}\sum_{t=1}^T Q_t.
\]
At each iteration $t$, we approximately sample from $Q_t$ using
constrained HMC with accuracy parameter $\delta$, producing a random
sample $\theta_t\sim\widetilde Q_t$, where
$\widetilde Q_t$ is the law of the HMC output.  Define the Monte Carlo
estimate of the robust risk
\[
  \widehat R_T
  := \frac{1}{T}\sum_{t=1}^T \ell(\theta_t).
\]

Then for any target accuracy $\varepsilon\in(0,1)$, if we choose
\[
  T
  \;\ge\;
  \frac{8M^2 C}{\varepsilon^2},
  \qquad
  \eta
  \;=\;
  \sqrt{\frac{2C}{M^2 T}},
  \qquad
  \delta
  \;=\;
  \frac{\varepsilon}{2L},
\]
we have
\[
  \bigl|
    R^{\mathrm{WH}}(\Pi,C) - \mathbb{E}[\widehat R_T]
  \bigr|
  \;\le\; \varepsilon.
\]

Moreover, with these choices, Algorithm~\ref{alg:mirror-adversary}
uses
\[
  T
  \;=\;
  O\!\left(\frac{M^2 C}{\varepsilon^2}\right)
\]
mirror--descent iterations (gradient evaluations) and a total of
\[
  O\!\left(
    d^{1/4}\,
    \frac{M^2 C}{\varepsilon^2}\,
    \log\frac{L}{\varepsilon}
  \right)
\]
constrained HMC steps.  In particular, the outer adversarial optimization
has polynomial $1/\varepsilon^2$ dependence on the target accuracy, while
the inner sampling complexity scales like $d^{1/4}$ in the dimension,
up to logarithmic factors.
\end{supptheorem}

\begin{proof}
Write
\[
  \mathcal{Q}_C := \{Q\ll\Pi : D_\phi(Q\Vert\Pi)\le C\},
\]
and define the convex functional
\[
  f(Q) := -F(Q)
  = -\int_\Theta \ell(\theta)\,Q(\mathrm d\theta),
  \qquad Q\in\mathcal{Q}_C.
\]
Maximizing $F$ over $\mathcal{Q}_C$ is equivalent to minimizing $f$ over
$\mathcal{Q}_C$; any minimizer of $f$ is a maximizer of $F$.

We use mirror descent with mirror map $h(Q):=D_\phi(Q\Vert\Pi)$ and
Bregman divergence $D_\phi(Q\Vert Q')$.  By assumption~(2), $h$ is
Fr\'echet differentiable and $1$--strongly convex with respect to the
total variation norm $\|\cdot\|_{\mathrm{TV}}$ on a neighborhood of
$\mathcal{Q}_C$, so that
\begin{equation}
\label{eq:h-strong-convex}
  D_\phi(Q\Vert Q')
  \;\ge\;
  \frac{1}{2}\,\|Q-Q'\|_{\mathrm{TV}}^2
  \qquad\text{for all $Q,Q'$ in this neighborhood.}
\end{equation}

The gradient of $f$ is the (constant) signed function
\[
  g(\theta) := -\ell(\theta),
\]
in the sense that for any signed perturbation $H$,
\[
  \mathrm D f(Q)[H]
  = -\int_\Theta \ell(\theta)\,H(\mathrm d\theta)
  = \bigl\langle g, H\bigr\rangle.
\]
We view $g$ as an element of the dual space with dual norm
\[
  \|g\|_*
  := \sup_{\|H\|_{\mathrm{TV}}\le 1}
      \bigl|\langle g,H\bigr|\;.
\]
By boundedness of $\ell$ and the definition of the total variation norm,
\[
  \|g\|_*
  = \sup_{\|H\|_{\mathrm{TV}}\le 1}
      \left|\int_\Theta -\ell(\theta)\,H(\mathrm d\theta)\right|
  \le
  \sup_{\theta}|\ell(\theta)|\,
  \sup_{\|H\|_{\mathrm{TV}}\le 1} \|H\|_{\mathrm{TV}}
  \le M.
\]

\paragraph{Step 1: Mirror--descent inequality with exact iterates.}
The mirror--descent update \eqref{eq:md-update} can be written as
\[
  Q_{t+1}
  = \arg\min_{Q\in\mathcal{Q}_C}
      \Bigl\{
        \eta\,\langle g,Q\rangle
        + h(Q) - h(Q_t) - \bigl\langle\nabla h(Q_t),Q-Q_t\bigr\rangle
      \Bigr\}.
\]
By first--order optimality, for every $Q\in\mathcal{Q}_C$ we have
\begin{equation}
\label{eq:FOC}
  \Big\langle
    \eta g + \nabla h(Q_{t+1}) - \nabla h(Q_t),
    Q - Q_{t+1}
  \Big\rangle
  \;\ge\; 0.
\end{equation}

The three--point identity for Bregman divergences gives, for any
$Q,Q',Q''$,
\[
  \bigl\langle\nabla h(Q') - \nabla h(Q''),
                         Q - Q'\bigr\rangle
  = D_\phi(Q\Vert Q'') - D_\phi(Q\Vert Q') - D_\phi(Q'\Vert Q'').
\]
Applying this with $Q'=Q_{t+1}$, $Q''=Q_t$ and $Q=Q^\star$ yields
\[
  \bigl\langle\nabla h(Q_{t+1}) - \nabla h(Q_t),
                         Q^\star - Q_{t+1}\bigr\rangle
  = D_\phi(Q^\star\Vert Q_t)
    - D_\phi(Q^\star\Vert Q_{t+1})
    - D_\phi(Q_{t+1}\Vert Q_t).
\]

Now set $Q=Q^\star$ in \eqref{eq:FOC} to obtain
\[
  \eta\,\bigl\langle g,Q^\star - Q_{t+1}\bigr\rangle
  \;\ge\;
  D_\phi(Q^\star\Vert Q_t)
  - D_\phi(Q^\star\Vert Q_{t+1})
  - D_\phi(Q_{t+1}\Vert Q_t).
\]
Rearranging,
\begin{equation}
\label{eq:gdiff-Qtp1}
  \bigl\langle g,Q_{t+1} - Q^\star\bigr\rangle
  \;\le\;
  \frac{1}{\eta}
  \Bigl(
    D_\phi(Q^\star\Vert Q_t)
    - D_\phi(Q^\star\Vert Q_{t+1})
    - D_\phi(Q_{t+1}\Vert Q_t)
  \Bigr).
\end{equation}

We now relate $Q_{t+1}-Q^\star$ to $Q_t-Q^\star$:
\[
  \bigl\langle g,Q_t - Q^\star\bigr\rangle
  =
  \bigl\langle g,Q_t - Q_{t+1}\bigr\rangle
  + \bigl\langle g,Q_{t+1} - Q^\star\bigr\rangle.
\]
Using \eqref{eq:gdiff-Qtp1} and then applying Cauchy--Schwarz and the
strong convexity \eqref{eq:h-strong-convex}, we obtain
\begin{align*}
  \bigl\langle g,Q_t - Q^\star\bigr\rangle
  &\le
    \bigl\langle g,Q_t - Q_{t+1}\bigr\rangle\\
  &\quad\;+
    \frac{1}{\eta}
    \Bigl(
      D_\phi(Q^\star\Vert Q_t)
      - D_\phi(Q^\star\Vert Q_{t+1})
      - D_\phi(Q_{t+1}\Vert Q_t)
    \Bigr)\\[0.5em]
  &\le
    \|g\|_*\,\|Q_t - Q_{t+1}\|_{\mathrm{TV}}\\
  &\quad\;+
    \frac{1}{\eta}
    \Bigl(
      D_\phi(Q^\star\Vert Q_t)
      - D_\phi(Q^\star\Vert Q_{t+1})
      - \tfrac12\|Q_{t+1}-Q_t\|_{\mathrm{TV}}^2
    \Bigr).
\end{align*}
By the elementary inequality
$ab \le \tfrac{\eta}{2}a^2 + \tfrac{1}{2\eta}b^2$ with
$a=\|g\|_*$, $b=\|Q_t - Q_{t+1}\|_{\mathrm{TV}}$, we have
\[
  \|g\|_*\,\|Q_t - Q_{t+1}\|_{\mathrm{TV}}
  \le \frac{\eta}{2}\|g\|_*^2
     + \frac{1}{2\eta}\|Q_t - Q_{t+1}\|_{\mathrm{TV}}^2.
\]
Substituting this above, the $\|Q_t - Q_{t+1}\|_{\mathrm{TV}}^2$ terms
cancel and we obtain the one--step mirror--descent inequality
\begin{equation}
\label{eq:md-onestep-final}
  \bigl\langle g,Q_t - Q^\star\bigr\rangle
  \;\le\;
  \frac{1}{\eta}
  \Bigl(
    D_\phi(Q^\star\Vert Q_t)
    - D_\phi(Q^\star\Vert Q_{t+1})
  \Bigr)
  + \frac{\eta}{2}\,\|g\|_*^2.
\end{equation}

\paragraph{Step 2: Optimization error (exact iterates).}
Summing \eqref{eq:md-onestep-final} over $t=1,\dots,T$ gives
\[
  \sum_{t=1}^T
  \bigl\langle g,Q_t - Q^\star\bigr\rangle
  \;\le\;
  \frac{1}{\eta}
  \Bigl(
    D_\phi(Q^\star\Vert Q_1)
    - D_\phi(Q^\star\Vert Q_{T+1})
  \Bigr)
  + \frac{\eta T}{2}\,\|g\|_*^2.
\]
Since $D_\phi(\cdot\Vert\cdot)\ge 0$ and $Q_1=\Pi$, we obtain
\[
  \sum_{t=1}^T
  \bigl\langle g,Q_t - Q^\star\bigr\rangle
  \;\le\;
  \frac{1}{\eta}
  D_\phi(Q^\star\Vert\Pi)
  + \frac{\eta T}{2}\,\|g\|_*^2
  \;\le\;
  \frac{C}{\eta}
  + \frac{\eta T}{2}\,\|g\|_*^2,
\]
because $Q^\star\in\mathcal{Q}_C$ implies
$D_\phi(Q^\star\Vert\Pi)\le C$.

Using $f(Q)=-F(Q)$ and the fact that $f$ is linear with gradient $g$,
we have
\[
  f(Q_t) - f(Q^\star)
  = \langle g,Q_t - Q^\star\rangle
  \quad\Longrightarrow\quad
  F(Q^\star) - F(Q_t)
  = \langle g,Q_t - Q^\star\rangle.
\]
Thus
\[
  \sum_{t=1}^T
  \bigl(F(Q^\star) - F(Q_t)\bigr)
  \;\le\;
  \frac{C}{\eta} + \frac{\eta T}{2}\,\|g\|_*^2.
\]
Dividing by $T$,
\[
  \frac{1}{T}\sum_{t=1}^T
  \bigl(F(Q^\star) - F(Q_t)\bigr)
  \;\le\;
  \frac{C}{\eta T} + \frac{\eta}{2}\,\|g\|_*^2.
\]
Since $F$ is linear (hence both convex and concave),
\[
  F(Q^\star) - F(\widehat Q_T)
  = F(Q^\star) - \frac{1}{T}\sum_{t=1}^T F(Q_t)
  = \frac{1}{T}\sum_{t=1}^T\bigl(F(Q^\star) - F(Q_t)\bigr),
\]
and we conclude
\begin{equation}
\label{eq:opt-gap}
  F(Q^\star) - F(\widehat Q_T)
  \;\le\;
  \frac{C}{\eta T} + \frac{\eta}{2}\,\|g\|_*^2.
\end{equation}

Using $\|g\|_*\le M$, \eqref{eq:opt-gap} becomes
\[
  F(Q^\star) - F(\widehat Q_T)
  \;\le\;
  \frac{C}{\eta T} + \frac{\eta M^2}{2}.
\]
The right--hand side is minimized over $\eta>0$ at
\[
  \eta^\star
  = \sqrt{\frac{2C}{M^2 T}},
\]
for which
\[
  \frac{C}{\eta^\star T}
  + \frac{\eta^\star M^2}{2}
  = \sqrt{2}\,M\sqrt{\frac{C}{T}}.
\]
Therefore
\begin{equation}
\label{eq:opt-error-final}
  F(Q^\star) - F(\widehat Q_T)
  \;\le\;
  \sqrt{2}\,M\sqrt{\frac{C}{T}}.
\end{equation}
In particular, if
\[
  T \;\ge\; \frac{8M^2 C}{\varepsilon^2},
  \qquad
  \eta = \eta^\star,
\]
then
\[
  F(Q^\star) - F(\widehat Q_T)
  \;\le\; \frac{\varepsilon}{2}.
\]

\paragraph{Step 3: Error from HMC approximation.}
At iteration $t$, let $\widetilde Q_t$ be the law of the HMC output after
$K(Q_t,\delta)$ steps targeting $Q_t$, and let
$\theta_t\sim\widetilde Q_t$.
By assumption~(3),
\[
  \mathbb{E}\bigl[W_1(\widetilde Q_t,Q_t)\bigr]\le\delta.
\]
Because $\ell$ is $L$--Lipschitz on $(\Theta,\|\cdot\|_2)$, the
Kantorovich--Rubinstein duality for $W_1$ yields
\[
  \bigl|
    \mathbb{E}_{\widetilde Q_t}[\ell]
    - \mathbb{E}_{Q_t}[\ell]
  \bigr|
  \le
  L\,W_1(\widetilde Q_t,Q_t).
\]
Taking expectations over the HMC randomness,
\[
  \bigl|
    \mathbb{E}[\ell(\theta_t)]
    - F(Q_t)
  \bigr|
  \le
  L\,\mathbb{E}[W_1(\widetilde Q_t,Q_t)]
  \le L\delta.
\]

Define
\[
  \widehat R_T
  = \frac{1}{T}\sum_{t=1}^T \ell(\theta_t),
  \qquad
  F(\widehat Q_T)
  = \frac{1}{T}\sum_{t=1}^T F(Q_t).
\]
Averaging the bound above over $t=1,\dots,T$ gives
\begin{equation}
\label{eq:hmc-bias}
  \bigl|
    \mathbb{E}[\widehat R_T]
    - F(\widehat Q_T)
  \bigr|
  \le L\delta.
\end{equation}
Hence, choosing
\[
  \delta = \frac{\varepsilon}{2L}
\]
ensures that the HMC--induced bias is at most $\varepsilon/2$.

\paragraph{Step 4: Combining optimization and sampling errors.}
We now bound the total error
\[
  R^{\mathrm{WH}}(\Pi,C) - \mathbb{E}[\widehat R_T]
  = F(Q^\star) - \mathbb{E}[\widehat R_T].
\]
By the triangle inequality and \eqref{eq:opt-error-final},
\eqref{eq:hmc-bias},
\begin{align*}
  \bigl|
    F(Q^\star) - \mathbb{E}[\widehat R_T]
  \bigr|
  &\le
    \bigl|F(Q^\star) - F(\widehat Q_T)\bigr|
    + \bigl|F(\widehat Q_T) - \mathbb{E}[\widehat R_T]\bigr|\\
  &\le
    \sqrt{2}\,M\sqrt{\frac{C}{T}} + L\delta.
\end{align*}
With $T$ and $\eta$ chosen as in Step~2 and
$\delta=\varepsilon/(2L)$ as in Step~3, we have
\[
  \sqrt{2}\,M\sqrt{\frac{C}{T}}
  \;\le\; \frac{\varepsilon}{2},
  \qquad
  L\delta = \frac{\varepsilon}{2},
\]
and therefore
\[
  \bigl|
    R^{\mathrm{WH}}(\Pi,C) - \mathbb{E}[\widehat R_T]
  \bigr|
  \le \varepsilon.
\]

\paragraph{Step 5: Complexity of HMC sampling.}
By assumption~(3), achieving Wasserstein accuracy $\delta$ for each
$Q_t$ requires at most
\[
  K(Q_t,\delta)
  \le A\,d^{1/4}\log\frac{B}{\delta}
\]
constrained HMC steps.  With $\delta=\varepsilon/(2L)$ this is
\[
  K(Q_t,\delta)
  = O\!\left(
      d^{1/4}\log\frac{L}{\varepsilon}
    \right),
\]
uniformly in $t$.  Since we run HMC once per mirror--descent iteration,
the total number of HMC steps is
\[
  T\,K(Q_t,\delta)
  = O\!\left(
      d^{1/4}\,
      \frac{M^2 C}{\varepsilon^2}\,
      \log\frac{L}{\varepsilon}
    \right),
\]
as claimed.  The number of mirror--descent iterations (and thus loss
evaluations) is $T=O(M^2 C/\varepsilon^2)$.
\end{proof}


\end{document}